\documentclass[11pt]{amsart}
\usepackage{amsmath,amssymb,amsthm}
\usepackage[latin1]{inputenc}
\usepackage{tabularx,multicol,array}
\usepackage{graphicx,float,psfrag}
 \usepackage{stmaryrd,mathrsfs}
\usepackage{url}
\usepackage{tikz}
\usepackage{mathrsfs}
\usepackage{eufrak}
\usepackage{accents}
\usepackage{subfigure}

\newcommand{\reff}[1]{(\ref{#1})}

\theoremstyle{plain}
\newtheorem{theo}{Theorem}[section]
\newtheorem{theo*}{Theorem}
\newtheorem{cor}[theo]{Corollary}
\newtheorem{prop}[theo]{Proposition}
\newtheorem{lem}[theo]{Lemma}
\newtheorem{defi}[theo]{Definition}
\theoremstyle{remark}
\newtheorem{rem}[theo]{Remark}

\newcommand{\ca}{{\mathcal A}}
\newcommand{\cb}{{\mathcal B}}
\newcommand{\cc}{{\mathcal C}}
\newcommand{\cd}{{\mathcal D}}

\newcommand{\cf}{{\mathcal F}}

\newcommand{\ck}{{\mathcal K}}

\newcommand{\cn}{{\mathcal N}}
\newcommand{\cm}{{\mathcal M}}
\newcommand{\cp}{{\mathcal P}}
\newcommand{\crr}{{\mathcal R}}

\newcommand{\cx}{{\mathbb X}}

\newcommand{\E}{{\mathbb E}}

\newcommand{\N}{{\mathbb N}}
\renewcommand{\P}{{\mathbb P}}

\newcommand{\R}{{\mathbb R}}

\newcommand{\ind}{{\bf 1}}

\newcommand{\argmax}{{\rm argmax}\;}
\newcommand{\supp}{{\rm supp}\;}

\newcommand{\Card}{{\rm Card}\;}

\newcommand{\norm}[1]{\mathop{\parallel\! #1 \! \parallel}\nolimits}
\newcommand{\val}[1]{\mathop{\left| #1 \right|}\nolimits}
\newcommand{\inv}[1]{\mathop{\frac{1}{ #1}}\nolimits}
\newcommand{\expp}[1]{\mathop {\mathrm{e}^{ #1}}}

\newcommand{\Var}{{\rm Var}\;}
\newcommand{\kl}[2]{\mathop{D\left(#1 \| #2 \right)}}

\newcommand{\pen}{{\rm pen}\;}
\newcommand{\Dom}{{\rm Dom}\;}
\newcommand{\inter}{{\rm int}\;}
\newcommand{\cv}{{\rm cv}\;}

\renewcommand{\phi}{\varphi}
\renewcommand{\epsilon}{\varepsilon}

\title[Fast adaptive estimation of log-additive exponential models]{Fast adaptive estimation of log-additive exponential models in Kullback-Leibler divergence}
\date{\today}

 \author{Cristina Butucea}
 \address{
 Cristina Butucea,
 LAMA (UPE-MLV), UPE,  Marne La Vall\'{e}e, France.}
 \email{cristina.butucea@univ-mlv.fr}

 \author{Jean-Fran\c{c}ois Delmas}
 \address{
 Jean-Fran\c{c}ois Delmas,
 CERMICS, \'{E}cole des Ponts, UPE, Champs-sur-Marne, France.}
 \email{delmas@cermics.enpc.fr}

 \author{Anne Dutfoy}
 \address{
 Anne Dutfoy, 
 EDF Research \& Development, Industrial Risk Management Department,  Palaiseau, France.}
 \email{anne.dutfoy@edf.fr}

 \author{Richard Fischer}
 \address{
 Richard Fischer, 
 CERMICS, \'{E}cole des Ponts, UPE, Champs-sur-Marne, France\\
 LAMA (UPE-MLV), UPE,  Marne La Vall\'{e}e, France \\
 EDF Research \& Development, Industrial Risk Management Department, Palaiseau, France.}
 \email{fischerr@cermics.enpc.fr}

\begin{document}

\thanks{This work is partially supported by the French ``Agence Nationale de
 la Recherche'',CIFRE n$^{\circ}$ 1531/2012, and by EDF Research \& Development, Industrial Risk Management Department}



 \begin{abstract} 
 
 We study the problem of nonparametric estimation of density functions with a product form on the domain $\triangle=\{( x_1, \ldots, x_d)\in \R^d, 0\leq x_1\leq \dots \leq  x_d \leq 1\}$. Such densities appear in the random truncation model as the joint density function of observations. They are also obtained as maximum entropy distributions of order statistics with given marginals.  We propose an estimation method based on the approximation of the logarithm of the density by a carefully chosen family of basis functions. We show that the method achieves a fast convergence rate in probability with respect to the Kullback-Leibler divergence for densities whose logarithm belongs to a Sobolev function class with known regularity. In the case when the regularity is unknown, we propose an estimation procedure using convex aggregation of the log-densities to obtain adaptability. The performance of this method is illustrated in a simulation study.       
 
\end{abstract}  

\maketitle

 \section{Introduction} \label{sec:intro}

   In this paper, we estimate densities with product form on the simplex $\triangle=\{( x_1, \ldots, x_d)\in
   \R^d, 0\leq x_1\leq \dots \leq  x_d \leq 1\}$ by a nonparametric approach given a sample of $n$ independent observations $\cx^n =(X^1,\hdots, X^n)$. We restrict our attention to densities which can be written in the form, for $x=(x_1, \hdots, x_d) \in \R^d$:
   \begin{equation} \label{eq:intro_prod_form}
     f^0(x)=\exp{\left(\sum_{i=1}^d \ell_{i}^0(x_i) - \mathrm{a}_0\right)} \ind_{\triangle}(x),
   \end{equation}
   with $\ell_i^0$ bounded, centered, measurable functions on $I=[0,1]$ for all $1 \leq i \leq d$, and normalizing constant $\mathrm{a}_0$. Densities of this form arise, in particular, as solutions for the maximum entropy problem for the distribution of order statistics with given marginals, or in the case of the random truncation model.

   The first example is the random truncation model, which was first formulated in \cite{turnbull1976empirical}, and has various applications ranging from astronomy (\cite{lynden1971method}), economics (\cite{herbst1999application}, \cite{guerre2000optimal}) to survival data analysis (\cite{lagakos1988nonparametric}, \cite{joly1998penalized}, \cite{luo2009nonparametric}). For $d=2$, let $(Z_1,Z_2)$ be a pair of independent random variables on $I$ such that $Z_i$ has density function $p_i$ for $i \in \{1,2\}$. Let us suppose that we can only observe realizations of $(Z_1,Z_2)$ if $Z_1 \leq Z_2$. Let $(\bar{Z}_1,\bar{Z}_2)$ denote a pair of random variables distributed as $(Z_1,Z_2)$ conditionally on $Z_1 \leq Z_2$. Then the joint density function  $f^0$ of $(\bar{Z_1},\bar{Z_2})$ is given by, for $x=(x_1,x_2) \in I^2$:
\begin{equation} \label{eq:trunc_joint_dens}
   f^0(x) = \inv{\alpha} p_1(x_1)p_2(x_2)\ind_{\triangle}(x),
\end{equation}
with $\alpha=\int_{I^2} p_1(x_1)p_2(x_2) \ind_{\triangle}(x) \, dx$.
Notice that $f$ is of the form required in \reff{eq:intro_prod_form}:
\begin{equation*} \label{eq:intro_trunc}
   f(x)=\exp(\ell^0_1(x_1)+\ell^0_2(x_2) - \mathrm{a}_0 ) \ind_{\triangle}(x), 
\end{equation*}
with $\ell^0_i$ defined as
 $ \ell^0_i = \log(p_i) - \int_I \log(p_i)$ for $i \in \{1,2\}$.
  According to Corollary 5.7. of \cite{butucea2015maximum}, $f$ is the density of the maximum entropy distribution of order statistics with marginals $\mathbf{f}_1$ and $\mathbf{f}_2$ given by:
  \[
     \mathbf{f}_1(x_1) = \inv{\alpha} p_1(x_1) \int_{x_1}^1 p_2(s) \, ds \quad \text{ and } \quad  \mathbf{f}_2(x_2) = \inv{\alpha} p_2(x_2) \int_{0}^{x_2} p_1(s) \, ds.
  \]
  
  \bigskip
  
   More generally, in \cite{butucea2015maximum}, the authors give a necessary and sufficient condition for the existence of a maximum entropy distribution of order statistics with fixed marginal cumulative distribution functions $\mathbf{F}_i$, $1 \leq i \leq d$. See \cite{butucea2015esrel} for motivations for this problem. Moreover, its explicit expression is given as a function of the marginal distributions.
   Let us suppose, for the sake of simplicity, that all $\mathbf{F}_i$ are absolutely continuous with density function $\mathbf{f}_i$ supported on $I=[0,1]$, and that $\mathbf{F}_{i-1} > \mathbf{F}_i$ on $(0,1)$ for $2 \leq i \leq d$.  Then the maximum entropy density $f_\mathbf{F}$, when it exists, is given by, for $x=(x_1, \hdots, x_d) \in \R^d$:
   \[
      f_{\mathbf{F}}(x)=
            \mathbf{f}_{1} (x_1) \prod_{i=2}^d h_i(x_i) \exp\left(-\int_{x_{i-1}}^{x_i}
           h_i(s) \, ds \right) \ind_{\triangle}(x),
   \] 
   with $h_i=\mathbf{f}_i/(\mathbf{F}_{i-1} - \mathbf{F}_i)$ for $2 \leq i \leq d$. This density is of the form required in \reff{eq:intro_prod_form} with $\ell_i^0$ defined as:
   \[
      \ell_1^0 = \log(\mathbf{f}_1)+K_2 \quad \text{ and } \quad  \ell_i^0=\log\left(h_i \right) - K_i + K_{i+1} \quad \text{for} \quad 2 \leq i \leq d,
   \]
   with $K_i$, $2 \leq i \leq d$ a primitive of $h_i$ chosen such that $\ell_i^0$ are centered, 
   and $K_{d+1}=c$ a constant.

   \bigskip
   
   We present an additive exponential series model specifically designed to estimate such densities. This exponential model is a multivariate version of the exponential series estimator considered in \cite{barron1991approximation} in the univariate setting. Essentially, we approximate the functions $\ell_i^0$ by a family of polynomials $(\phi_{i,k},k \in \N)$,  which are  orthonormal for each $1\leq i \leq d$ with respect to the $i$-th marginal of the Lebesgue measure on the support $\triangle$. The model takes the form, for $\theta = (\theta_{i,k}; 1\leq i \leq d, 1\leq k \leq m_i)$ and $x=(x_1, \hdots, x_d) \in \triangle$:
   \[
    f_\theta=\exp\left( \sum_{i=1}^d \sum_{k=1}^{m_i} \theta_{i,k} \phi_{i,k}(x_i) - \psi(\theta)\right), 
   \]
    with  $\psi(\theta)=\log\left( \int_\triangle  \exp\left( \sum_{i=1}^d \sum_{k=1}^{m_i} \theta_{i,k} \phi_{i,k}(x_i)  \right) \,dx\right)$. Even though the polynomials $(\phi_{i,k}$, $k \in \N)$ are orthonormal for each $1\leq i \leq d$, if we take $i\neq j$, the families $(\phi_{i,k}, k \in \N)$ and $(\phi_{j,k}, k \in \N)$ are not completely orthogonal with respect to the Lebesgue measure on $\triangle$. The exact definition and further properties of these polynomials can be found in the Appendix. We estimate the parameters of the model by $\hat{\theta} = (\hat{\theta}_{i,k}; 1\leq i \leq d, 1\leq k \leq m_i)$, obtained by solving the maximum likelihood equations:
    \[
       \int_\triangle \phi_{i,k}(x_i) f_{\hat{\theta}}(x) \, dx = \inv{n} \sum_{j=1}^n \phi_{i,k} (X^j_i)  \quad \quad \quad \text{ for } 1\leq i \leq d, \, 1\leq k \leq m_i.
    \]
    
   Approximation of log-densities by polynomials appears in \cite{good1963maximum} as an application of the maximum entropy principle, while \cite{crain1977information} shows existence and consistency of the maximum likelihood estimation.
    We measure the quality of the estimator $f_{\hat{\theta}}$ of $f^0$ by the Kullback-Leibler divergence $\kl{ f^0}{ f_{\hat{\theta}}}$ defined as:
    \[
      \kl{ f^0}{ f_{\hat{\theta}}} = \int_\triangle f^0 \log\left(f^0/f_{\hat{\theta}}\right). 
    \]
    Convergence rates for nonparametric density estimators have been given by \cite{hall1987kullback} for kernel density estimators, \cite{barron1991approximation} and \cite{wu2010exponential} for the exponential series estimators, \cite{barron1992distribution} for histogram-based estimators, and \cite{koo1996wavelet} for wavelet-based log-density estimators. Here, we give results for the convergence rate in probability when the functions $\ell_i^0$ belong to a Sobolev space with regularity $r_i >d$ for all $1 \leq i \leq d$. We show that if we take $m=m(n)=(m_1(n),\hdots, m_d(n))$ members of the families $(\phi_{i,k},k \in \N)$, $1 \leq i \leq d$, and let $m_i$ grow with $n$ such that $(\sum_{i=1}^d m_i^{2d})(\sum_{i=1}^d m_i^{-2r_i})$ and $(\sum_{i=1}^d m_i)^{2d+1}/n$ tend to $0$, then the maximum likelihood estimator $f_{\hat{\theta}_{m,n}}$ verifies:
    \[
       \kl{ f^0 }{ f_{\hat{\theta}_{m,n}}}= O_\P\left( \sum_{i=1}^d \left(m_i^{-2r_i} + \frac{m_i}{n}\right)\right).
    \]
    Notice that this is the sum of the same univariate convergence rates as in \cite{barron1991approximation}. By choosing $m_i$ proportional to $ n^{1/(2r_i+1)}$, which gives the optimal convergence rate $O_\P( n^{-2r_i/(2r_i+1)})$ in the univariate case as shown in \cite{yang1999information}, we achieve a convergence rate of $O_\P(n^{-2\min(r)/(2\min(r)+1)})$. Therefore by exploiting the special structure of the underlying density, and carefully choosing the basis functions, we managed to reduce the problem of estimating a $d$-dimensional density to $d$ one-dimensional density estimation problems. We highlight the fact that this constitutes  a significant gain over convergence rates  of general nonparametric multivariate density estimation methods. 
    
    In most cases the smoothness parameters $r_i$, $1 \leq i \leq d$, are not available, therefore a method which adapts to the unknown smoothness is required to estimate the density with the best possible convergence rate. Adaptive methods for function estimation based on a random sample include Lepski's method, model selection, wavelet thresholding and aggregation of estimators. 

  Lepski's method, originating from \cite{lepskiui1990problem}, consists of constructing a grid of regularities, and choosing among the minimax estimators associated to each regularity the best estimator by an iterative procedure based on the available sample. This method was extensively applied for Gaussian white noise model, regression, and density estimation, see \cite{butucea2001exact} and references therein. Adaptation via model selection with a complexity penalization criterion was considered by \cite{birge1996} and \cite{barron1999risk}  for a large variety of models including wavelet-based density estimation. Loss in the Kullback-Leibler distance for model selection was studied in \cite{yang2000mixing} 
  and \cite{catoni1997mixture} for  mixing strategies, and in \cite{zhang2006} for the information complexity minimization strategy. More recently, bandwidth selection for multivariate kernel density estimation was addressed in \cite{goldenshluger2011bandwidth} for $L^s$  risk, $1\leq s < \infty$,  and \cite{lepski2013multivariate} for $L^\infty$ risk.        
  Wavelet based adaptive density estimation with thresholding was considered in \cite{kerkyacharian1996lp} and \cite{donoho1996density}, where an upper bound for the rate of convergence was given for a collection of Besov-spaces.  Linear and convex aggregate estimators appear in the more recent work \cite{rigollet2007linear} with an application to adaptive density estimation in expected $L^2$ risk, with sample splitting. 

   Here we extend the convex aggregation scheme for the estimation of the logarithm of the density proposed in \cite{butucea2016optimal} to achieve adaptability. We take the estimator $f_{\hat{\theta}_{m,n}}$ for different values of $m \in \cm_n$, where $\cm_n$ is a sequence of sets of parameter configurations with increasing cardinality. These estimators are not uniformly bounded as required in \cite{butucea2016optimal}, but we show that they are uniformly bounded in probability and that it does not change the general result.
   The different values of $m$ correspond to different values of the regularity parameters. The convex aggregate estimator $f_\lambda$ takes the form:
   \[
      f_\lambda = \exp\left( \sum_{m\in \cm_n} \lambda_m \left(\sum_{i=1}^d \sum_{k=1}^{m_{i}} \theta_{i,k} \phi_{i,k}(x_i)\right)  -\psi_\lambda\right)\ind_{\triangle},
   \]
   with $\lambda \in \Lambda^+ =\{ \lambda=(\lambda_m, m \in \cm_n), \lambda_m \geq 0 \text{ and } \sum_{m \in \cm_n} \lambda_m=1\}$ and normalizing constant $\psi_\lambda$ given by:
    \[
       \psi_\lambda = \log\left(\int_\triangle \exp\left( \sum_{m \in \cm_n} \lambda_m \left(\sum_{i=1}^d \sum_{k=1}^{m_{i}} \theta_{i,k} \phi_{i,k}(x_i)\right) \right) \, dx \right).
    \]
   To apply the aggregation method, we split our sample $\cx^n$ into two parts $\cx^n_1$  and $\cx^n_2$, with size proportional to $n$. We use the first part to create the estimators $f_{\hat{\theta}_{m,n}}$, then we use the second part to determine the optimal choice of the aggregation parameter $\hat{\lambda}^*_n$. We select $\hat{\lambda}^*_n$ by maximizing  a penalized version of the log-likelihood function. We show that this method gives a sequence of estimators $f_{\hat{\lambda}^*_n}$, free of the smoothness parameters $r_1,\hdots,r_d$,  which verifies:
   \[
      \kl{f^0}{f_{\hat{\lambda}^*_n} } = O_\P \left(  n^{-\frac{2 \min(r)}{2 \min(r) +1}} \right).  
   \]

    The rest of the paper is organized as follows. In Section \ref{sec:not} we introduce the notation used in the rest of the paper. In Section \ref{sec:estim_stat}, we describe the additive exponential series model and the estimation procedure, then we show that the estimator converges to the true underlying density with a convergence rate that is the sum of the convergence rates for the same type of univariate model, see Theorem \ref{theo:main_stat}. We consider an adaptive method with convex aggregation of the logarithms of the previous estimators to adapt to the unknown smoothness of the underlying density in Section \ref{sec:adaptation}, see Theorem \ref{theo:adapt}. We assess the performance of the adaptive estimator via a simulation study in Section \ref{sec:simul}. The definition of the basis functions and their properties used during the proofs are given in Section \ref{sec:orth_poly}. The detailed proofs of the results in Section \ref{sec:estim_stat} and \ref{sec:adaptation} are contained in Sections \ref{sec:prelim}, \ref{sec:proof_theo_stat}  and \ref{sec:proof_theo_adapt}.

 \section{Notation} \label{sec:not}
Let  $I=[0,1]$, $d \geq 2$ and   $\triangle = \{ (x_1, \hdots, x_d) \in I^d, x_1 \leq x_2 \leq \hdots \leq x_d  \}$ denote  the simplex  of $I^d$.  
For  an  arbitrary  real-valued  function  $h_i$  defined  on  $I$ with $1 \leq i \leq d$,  let
$h_{[i]}$  be  the  function  defined   on  $\triangle$  such  that  for
$x=(x_1,\hdots,x_d) \in \triangle$: 
\begin{equation}
   \label{eq:def-hi}
h_{[i]}(x) = h_i(x_i)\ind_\triangle(x).
\end{equation}

Let $q_i$, $1 \leq i \leq d$ be the  one-dimensional marginals of the
Lebesgue measure on $\triangle$:
 \begin{equation} \label{eq:def_qi}
   q_i(dt)  = \inv{(d-i)! (i-1)!} (1-t)^{d-i}  t^{i-1}\, \ind_I(t) \,dt. 
 \end{equation}
If $h_i\in L^1(q_i)$, then we have:
$
\int_\triangle h_{[i]} = \int_I h_i q_i.
$

For a measurable function $f$,  let $\norm{f}_{\infty}$ be the usual sup
norm of $f$ on its domain of definition. For $f$ defined on $\triangle$,
let $\norm{f}_{L^2}=\sqrt{\int_\triangle f^2}$. For  $f$ defined on $I$,
let $\norm{f}_{L^2(q_i)}=\sqrt{\int_I f^2 q_i}$.

For a vector $x =(x_1, \hdots, x_d) \in \R^d$, let $\min(r)$ ($\max(r)$) denote the smallest (largest) component.

Let us denote the support of a probability density $g$ by $\supp(g) =\{x \in \R^d,
g(x)>0\}$. Let $\cp(\triangle)$   denote   the   set  of   probability   densities   on
$\triangle$. For $g,h\in  \cp(\triangle)$, the Kullback-Leibler distance
$\kl{g}{h}$ is defined as:
\[
  \kl{g}{h} = \int_{\triangle} g
\log\left(g/h\right).
\]
Recall that $\kl{g}{h} \in [0, +\infty]$.
\begin{defi}
   \label{defi:prod-f}
   We say that a  probability density  $f^0\in \cp(\triangle)$ has a  product form if
   there  exist  $(\ell_i^0, 1\leq i\leq
   d)$ bounded  measurable functions  defined on  $I$ such  that $\int_I
   \ell_i^0 q_i=0$ for $1\leq i\leq d$ and a.e. on $\triangle$:
\begin{equation}
   \label{eq:Pi-form}
f^0=\exp{\left(\ell^0 - \mathrm{a}_0\right)}\ind_\triangle,
\end{equation}
with $\ell^0=\sum_{i=1}^d \ell^0_{[i]}$ and $\mathrm{a}_0=\log\left(\int_\triangle \exp{ (\ell^0)} \right)$,
that is $f^0(x)=\exp{\left(\sum_{i=1}^d \ell_{i}^0(x_i) - \mathrm{a}_0\right)}$ for
a.e. $x=(x_1, \ldots, x_d)\in \triangle$. 
\end{defi}

 Definition \ref{defi:prod-f} implies that $\supp(f^0) = \triangle$ and $f^0$ is bounded.
 Let $\cx^n=(X^1, \hdots, X^n)$ denote an i.i.d. sample of size $n$ from the density $f^0$.

For $1\leq i\leq  d$, let $(\phi_{i,k}, k\in \N)$ be the  family of
orthonormal polynomials on $I$ with respect to the measure $q_i$; see
Section \ref{sec:orth_poly} for a precise definition of those
polynomials and some of their properties. Recall
$\varphi_{[i],k}(x)=\varphi_{i,k}(x_i)$ for $x=(x_1, \ldots, x_d)\in
\triangle$. Notice that $( \phi_{[i],k}, 1\leq i\leq d, k\in
\N)$ is a family of normal polynomials  with
respect to the Lebesgue measure on $\triangle$, but not orthogonal. 

Let $m=(m_1,\hdots,m_d) \in (\N^*)^d$ and set $\val{m}=\sum_{i=1}^d m_i$.
 We define  the $\R^{\val{m}}$-valued  function $\phi_m =(\phi_{[i],k};
 1\leq k \leq m_i,1\leq i\leq d )$ and the $\R^{{m_i}}$-valued  functions
$\phi_{i,m}=(\phi_{i,k}; 1\leq k \leq m_i)$ for $1\leq i\leq d $. 
For  $\theta=(\theta_{i,k}; 1 \leq k
\leq m_i,1 \leq i \leq d )$ and $\theta'=(\theta'_{i,k}; 1 \leq k
\leq m_i,1 \leq i \leq d )$ elements of $\R^{\val{m}}$, we denote the
scalar product:
\[
\theta \cdot \theta' = \sum_{i=1}^d \sum_{k=1}^{m_i} \theta_{i,k}
\theta'_{i,k} 
\]
and the norm $\norm{\theta}=\sqrt{\theta \cdot \theta}$.  We define the
 function $ \theta\cdot \varphi_m$ as follows, for $x\in \triangle$:
\[
(\theta\cdot \varphi_m)(x)=\theta\cdot \varphi_m(x). 
\]

 For a positive sequence $(a_{n})_{n \in \N}$, the notation $O_\P(a_{n})$ of
 stochastic boundedness for a sequence of random variables $(Y_{n}, n
 \in \N)$ means that for every $\varepsilon>0$, there exists $C_{\varepsilon}>0$ such that:
  \[
   \P\left(\val{Y_{n}/a_{n}} > C_\varepsilon \right) < \varepsilon \quad \text{for
     all } n \in \N.
  \]

\section{Additive exponential series model}
\label{sec:estim_stat}

In this Section, we study the problem of estimation of an unknown density $f^0$ with a product form on the set $\triangle$, as described in \reff{eq:Pi-form}, given the sample $\cx^n$ drawn from $f^0$. Our goal is to give an estimation method based on a sequence of regular exponential models, which suits the special characteristics of the target density $f^0$. Estimating such a density with standard multidimensional nonparametric techniques naturally suffer from the curse of dimensionality, resulting in slow convergence rates for high-dimensional problems.  We show that by taking into consideration that $f^0$ has a product form, we can recover the one-dimensional convergence rate for the density estimation, allowing for fast convergence of the estimator even if $d$ is large.  The quality of the estimators is measured by the Kullback-Leibler distance, as it has strong connections to the maximum entropy framework of \cite{butucea2015maximum}.

We  propose to  estimate  $f^0$ using  the
following additive exponential series model, for $m\in (\N^*)^d$:
 \begin{equation} \label{eq:f_theta}
    f_\theta=\exp\left( \theta \cdot \phi_m -
      \psi(\theta)\right)\ind_{\triangle}, 
 \end{equation}
 with  $\psi(\theta)=\log\left( \int_\triangle  \exp\left( \theta  \cdot
     \phi_m  \right)\right)$.   This  model  is similar  to  the  one
 introduced  in  \cite{wu2010exponential},  but   there  are  two  major
 differences.   First, we  have only  kept the  univariate terms  in the
 multivariate  exponential series estimator  of \cite{wu2010exponential}  
 since the  target probability  density is  the
 product of univariate functions.  Second,  we have restricted our model
 to $\triangle$ instead of the hyper-cube  $I^d$, and we have chosen the
 basis functions  $( ( \phi_{i,k}, k\in  \N), 1\leq i\leq d)$  which are
 appropriate for this support.
 
 \begin{rem}
  In the genaral case, one has to be careful when considering a density $f^0$ with a product form and a support different from $\triangle$.
  Let  $f^0_i$ denote the $i$-th marginal density function of $f^0$. If $\supp(f^0_i) = A \subset\R$ for all $1 \leq i \leq d$, we can apply a strictly 
  monotone mapping of $A$ onto $I$ to obtain a  distribution with a product form supported on $\triangle$. 
   When the supports of the marginals differ, there is no transformation that yields a random vector with a density as in Definition \ref{defi:prod-f}. 
   A possible way to treat this case consists of constructing a family of basis functions which has similar properties with respect to $\supp(f^0)$ as the family $( ( \phi_{i,k}, k\in  \N), 1\leq i\leq d)$ with respect to $\triangle$, which we discuss in detail in Section \ref{sec:orth_poly}. Then we could define an exponential series model with this family of basis functions and support restricted to $\supp(f^0)$ to estimate $f^0$.    
\end{rem}

 Let $m\in (\N^*)^d$. We define the following function on $\R^{\val{m}}$
 taking values in $\R^{\val{m}}$ by:
\begin{equation} \label{eq:alpha_m}
  A_m(\theta)=\int_\triangle \phi_m f_\theta, \quad  \theta   \in
  \R^{\val{m}}.                           
\end{equation}
According to Lemma 3 in \cite{barron1991approximation}, we have the following result on $A_m$.

\begin{lem} \label{lem:Am_bij}
The function $A_m$ is one-to-one from $\R^{\val{m}}$ to $\Omega_m = A_m(\R^{\val{m}})$.  
\end{lem}
We denote by $\Theta_m : \Omega_m \mapsto \R^{\val{m}}$ the inverse of $A_m$. 
The empirical mean of the  sample $\cx^n$ of size $n$ is:
\begin{equation}\label{eq:def_emp_mean}
\hat{\mu}_{m,n}=\inv{n} \sum_{j=1}^n \phi_{m} (X^j).
\end{equation}
In Section \ref{sec:variance} we show that $\hat{\mu}_{m,n} \in \Omega_m$ a.s. when $n \geq 2$.

For $n \geq 2$, we define a.s. the maximum  likelihood  estimator $\hat{f}_{m,n} = f_{\hat{\theta}_{m,n}}$  of  $f^0$  by choosing:
\begin{equation} \label{eq:max_likelihood}
                                                          \hat{\theta}_{m,n} = \Theta_m(\hat{\mu}_{m,n}).
\end{equation}

The  loss   between  the  estimator
$\hat{f}_{m,n}$ and the true underlying density $f^0$ is measured by the
Kullback-Leibler divergence $\kl{f^0}{\hat{f}_{m,n}}$.

For $r \in  \N^*$, let $W^2_r(q_i)$ denote the Sobolev  space of functions in
$L^2(q_i)$, such that the $(r-1)$-th  derivative is absolutely continuous and
the $L^2$ norm  of the $r$-th derivative is finite:
\[
  W^2_r (q_i)= \left\{ h \in L^2(q_i); h^{(r-1)} \text{ is absolutely continuous and } h^{(r)} \in L^2(q_i)  \right\}. 
\]
The main result is
given by the following theorem whose proof is given in Section \ref{sec:proof_stat}.

\begin{theo} \label{theo:main_stat} 
  Let $f^0\in \cp(\triangle)$ be a probability density with a product form, see Definition \ref{defi:prod-f}. 
Assume  the functions $\ell^0_i$, defined in \reff{eq:Pi-form} belong to the Sobolev space $W^2_{r_i}(q_i)$, $r_i \in \N$ with $r_i > d$ for all  $1 \leq i \leq d$. Let $(X^n, n\in \N^*)$ be i.i.d. random variables with density distribution $f^0$.  We consider a sequence $(m(n)=(m_1(n), \hdots, m_d(n)), n \in \N^*)$ such that $\lim_{n\rightarrow \infty} m_i(n) = +\infty$ for all $1 \leq i \leq d$, and which satisfies:
  \begin{equation}\label{eq:cond_m_theo_1}
     \lim_{n \rightarrow \infty} \val{m}^{2d} \left( \sum_{i=1}^d m_i^{-2r_i} \right) = 0,
  \end{equation}
    \begin{equation}\label{eq:cond_m_theo_2}
     \lim_{n \rightarrow \infty} \frac{\val{m}^{2d+1}}{n}  = 0.
    \end{equation}
    The  Kullback-Leibler distance $\kl{f^0}{\hat{f}_{m,n}}$ of the maximum likelihood estimator $\hat{f}_{m,n} $ defined by
    \reff{eq:max_likelihood} to $f^0$ converges in probability to $0$ with the convergence rate:
  \begin{equation} \label{eq:conv_rate}
     \kl{f^0}{\hat{f}_{m,n}} = O_\P\left( \sum_{i=1}^d m_i^{-2r_i} +
     \frac{\val{m}}{n}\right) .    
  \end{equation}
\end{theo}

\begin{rem}
   \label{rem:optim}
  Let us take $(m^{\circ}(n)=(m_1^{\circ}(n), \hdots, m_d^{\circ}(n)),n \in \N^*)$ with $m^{\circ}_i(n)=\lfloor n^{1/(2r_i+1)} \rfloor$. This choice constitutes a balance between the bias and the variance term.  Then the conditions
  \reff{eq:cond_m_theo_1} and \reff{eq:cond_m_theo_2} are satisfied, and
  we obtain that :
  \[
  \kl{f^0}{\hat{f}_{m^{\circ},n}} = O_\P\left(\sum_{i=1}^d n^{-2r_i/(2r_i+1)}\right) = O_\P\left(n^{-2 \min(r)/(2\min(r)+1)} \right).
  \]
  Thus the convergence rate corresponds to the least smooth $\ell^0_i$. This rate can also be obtained with a choice where all $m_i$ are the same. Namely, with $(m^*(n)=(v^*(n), \hdots, v^*(n)),n \in \N^*)$ and $v^*(n)=\lfloor n^{1/(2 \min(r)+1)} \rfloor$.
\end{rem}

 For $r=(r_1, \hdots, r_d) \in (\N^*)^d$, $r_i > d$ for $1 \leq i \leq d$, and a constant $\kappa >0$, let :
 \begin{equation} \label{eq:def_KrK}
\ck_r(\kappa)=\left\{ f^0=\exp \left(\sum_{i=1}^d\ell_{[i]}^0-\mathrm{a}_0\right) \in \cp(\triangle); \norm{\ell_i^0}_\infty \leq \kappa, \norm{(\ell_i^0)^{(r_i)}}_{L^2(q_i)} \leq \kappa  \right\}.
\end{equation}
 The constants $\mathfrak{A}_1$ and $\mathfrak{A}_2$, appearing in the upper bounds during the proof of Theorem \ref{theo:main_stat} (more precisely in Propositions \ref{prop:general_1} and \ref{prop:general_2}), are uniformly bounded on $\ck_r(\kappa)$, thanks to Corollary \ref{cor:gamma_m} and $\norm{\log(f^0)}_\infty \leq 2 d \kappa + \val{\log(d!)}$, which is due to \reff{eq:f_l_a_1}. This yields the following corollary for the uniform  convergence in probability on the set $\ck_r(\kappa)$ of densities:
 
 \begin{cor} \label{cor:unif_conv}
 Under the assumptions of Theorem \ref{theo:main_stat}, we get the following result:
 \[
    \lim_{K\rightarrow \infty} \limsup_{n \rightarrow \infty} \sup_{f^0 \in \ck_r(\kappa)} \P\left(\kl{f^0}{\hat{f}_{m,n}} \geq \left(\sum_{i=1}^d m_i^{-2r_i} + \frac{\val{m}}{n}\right)K\right) =0.
 \]
\end{cor}

\begin{rem}
 Since we let $r_i$ vary for each $1\leq i \leq d$, our class of densities $\ck_r(\kappa)$ has an anisotropic feature. Estimation of anisotropic multivariate functions for $L^s$ risk, $1\leq s \leq \infty$, was considered in multiple papers. For a Gaussian white noise model, \cite{kerkyacharian2001nonlinear} obtains minimax convergence rates on anisotropic Besov classes for $L^s$ risk, $1\leq s < \infty$ ,while \cite{bertin2004asymptotically} gives the minimax rate of convergence on anisotropic H\"{o}lder classes for the $L^\infty$ risk. For kernel density estimation, results on the minimax convergence rate for anisotropic Nikol'skii classes for $L^s$ risk, $1\leq s < \infty$, can be found in \cite{goldenshluger2011bandwidth}. These papers conclude in general, that if the considered class has smoothness parameters $\tilde{r}_i$ for the $i$-th coordinate, $1 \leq i \leq d$ , then the optimal convergence rate becomes $n^{-2\tilde{R}/(2\tilde{R}+1)}$ (multiplied with a logarithmic factor for $L^\infty$ risk), 
with $\tilde{R}$ defined by the equation $1/\tilde{R} = \sum_{i=1}^d 1/\tilde{r}_i$. Since $\tilde{R} < \tilde{r}_i$ for all $1\leq i \leq d$, the convergence rate $n^{-2 \min(r)/(2\min(r)+1)}$ is strictly better than the convergence rate for these anisotropic classes. In the isotropic case, when $r_i=r$ for all $1\leq i \leq d$, the minimax convergence rate specializes to $n^{-2r/(2r+d)}$ (which was obtained in \cite{wu2010exponential} as an upper bound).
 This rate decreases exponentially when the dimension $d$ increases. However, by exploiting the multiplicative structure of the model, we managed to obtain the univariate convergence rate $n^{-2r/(2r+1)}$, which is minimax optimal, see \cite{yang1999information}.  
\end{rem}

\section{Adaptive estimation} \label{sec:adaptation}

Notice that the choice of the optimal series of estimators $\hat{f}_{m^*,n}$ with $m^*$ defined in Remark \ref{rem:optim} requires the knowledge of $\min(r)$ at least. When this knowledge is not available, 
we propose an adaptive method based on the proposed estimators in Section \ref{sec:estim_stat}, which can mimic asymptotically the behaviour of the optimal choice. 
Let us introduce some notation first. We separate the sample $\cx^n$ into two parts $\cx^n_1$ and $\cx^n_2$ of size $n_1= \lfloor C_e n \rfloor$ and $n_2 = n-\lfloor C_e n \rfloor$ respectively, with some constant $C_e \in (0,1)$. The first part of the sample  will be used to create our estimators, and the second half will be used in the aggregation procedure.
Let $(N_n, n \in \N^*)$ be a sequence of non-decreasing positive integers depending on $n$ such that $\lim_{n\rightarrow \infty} N_n= + \infty$. Let us denote:
\begin{equation} \label{eq:def_Mn}
 \cn_n = \left\{  \lfloor n^{1/(2(d+j)+1)} \rfloor , 1 \leq j \leq N_n \right\} \quad \text{ and } \quad \cm_n=\left\{ m=(v,\hdots, v) \in \R^d, v \in \cn_n \right\}. 
\end{equation}
 For $m \in \cm_n$ let $\hat{f}_{m,n}$ be the additive exponential series estimator based on the first half of the sample, namely:
   \[
      \hat{f}_{m,n}= \exp\left(\hat{\theta}_{m,n} \cdot \phi_{m}-\psi(\hat{\theta}_{m,n}) \right)\ind_{\triangle},
   \]
   with $\hat{\theta}_{m,n}$ given by \reff{eq:max_likelihood} using the sample $\cx^n_1$ (replacing $n$ with $n_1$ in the definition \reff{eq:def_emp_mean} of $\hat{\mu}_{m,n}$). 
    Let :  
   \begin{equation*} \label{eq:def_Mn'}   
    \cf_n = \{ \hat{f}_{m,n}, m \in \cm_n \}
   \end{equation*}
   denote the set of different estimators obtained by this procedure.    Notice that $\Card(\cf_n) \leq \Card(\cm_n) \leq N_n$. 
Recall that by Remark \ref{rem:optim}, we have that for $r = (r_1, \hdots, r_d)$ with $r_i > d$  and $n \geq \bar{n}$, where $\bar{n}$ is given by:
\begin{equation} \label{eq:def_bar_n}
    \bar{n} = \min\{n \in \N, N_{n} \geq \min(r)-d+1\},
\end{equation}
 the sequence of estimators $\hat{f}_{m^*,n}$, with $m^*=m^*(n)=(v^*, \hdots, v^*) \in \cm_n$ given by $v^* = \lfloor n^{1/(2 \min(r)+1)} \rfloor$, achieves the optimal convergence rate  $O_\P( n^{-2\min(r)/(2 \min(r)+1)})$. By letting $N_n$ go to infinity, we ensure that for every combination of regularity parameters $r = (r_1, \hdots, r_d)$ with $r_i > d$, the sequence of optimal estimators $\hat{f}_{m^*,n}$ is included in the sets $\cf_n$ for $n$ large enough.  
   
   We use the second part of the sample $\cx_2^n$ to create an aggregate estimator based on $\cf_n$, which asymptotically mimics the performance of the optimal sequence  $\hat{f}_{m^*,n}$.
   We will write $\hat{\ell}_{m,n} =\hat{\theta}_{m,n} \cdot \phi_{m}$ to ease notation. We define the convex combination  $\hat{\ell}_\lambda$ of the functions $\hat{\ell}_{m,n}$, $m \in \cm_n$:
   \[
      \hat{\ell}_\lambda = \sum_{m \in \cm_n} \lambda_m \hat{\ell}_{m,n},
   \]
   with aggregation weights $\lambda \in \Lambda^+ =\{ \lambda=(\lambda_m, m \in \cm_n) \in \R^{\cm_n}, \lambda_m \geq 0 \text{ and } \sum_{m \in \cm_n} \lambda_m=1\}$. For such a convex combination, we define the probability density function $f_\lambda$ as:
   \begin{equation} \label{eq:aggr_f}
      f_\lambda = \exp(\hat{\ell}_\lambda -\psi_\lambda)\ind_{\triangle},
   \end{equation}
    with $\psi_\lambda = \log\left(\int_\triangle \exp(\hat{\ell}_\lambda) \right)$. We apply the convex aggregation method for log-densities developed in \cite{butucea2016optimal} to get  an aggregate estimator which achieves adaptability. Notice that the reference probability measure in this paper corresponds to $d!\ind_{\triangle}(x)dx$. This implies that $\psi_\lambda$ here differs from the $\psi_\lambda$ of \cite{butucea2016optimal} by the constant $\log(d!)$, but this does not affect the calculations. The aggregation weights are chosen by maximizing the penalized maximum likelihood criterion $H_n$ defined as:  
    \begin{equation}\label{eq:def_Hn}
      H_{n}(\lambda)=\inv{ n_2} \sum_{X^j \in \cx^n_2} \hat{\ell}_\lambda(X^{j}) - \psi_\lambda - \inv{2} \pen(\lambda), 
    \end{equation}
    with the penalizing function 
    $
       \pen(\lambda)=  \sum_{m \in \cm_n} \lambda_m \kl{f_\lambda}{ \hat{f}_{m,n}}.
    $
    The convex aggregate estimator $f_{\hat{\lambda}^*_n}$ is obtained by setting:
    \begin{equation} \label{eq:aggr_lambda}
      \hat{\lambda}^*_n  = \mathop{\argmax}_{\lambda \in \Lambda^+ } H_{n}(\lambda).
    \end{equation}

     The main result of this section is given by the next theorem which asserts that if we choose $N_n=o(\log(n))$ such that $\lim_{n \rightarrow \infty} N_n = +\infty$, the series of convex aggregate estimators $f_{\hat{\lambda}^*_n}$ converge to $f^0$ with the optimal convergence rate, i.e. as if the smoothness was known. 
    
\begin{theo} \label{theo:adapt}  Let $f^0\in \cp(\triangle)$ be a probability density with a product form given  by \reff{eq:Pi-form}. 
Assume  the functions $\ell^0_i$ belongs to the Sobolev space $W^2_{r_i}(q_i)$, $r_i \in \N$ with $r_i > d$ for all  $1 \leq i \leq d$. Let $(X^n, n\in \N^*)$ be i.i.d. random variables with density
 $f^0$. Let $N_n=o(\log(n))$ such that $\lim_{n \rightarrow \infty} N_n = +\infty$. The convex aggregate estimator  $f_{\hat{\lambda}^*_n}$ defined by \reff{eq:aggr_f} with $\hat{\lambda}^*_n$ given by \reff{eq:aggr_lambda} converges to $f^0$  in probability  with the convergence rate:
\begin{equation} \label{eq:kl_adapt}
        \kl{f^0}{f_{\hat{\lambda}^*_n} } = O_\P \left(  n^{-\frac{2 \min(r)}{2 \min(r) +1}} \right). 
\end{equation}

    \end{theo}
   
 The proof of this theorem is provided in Section \ref{sec:proof_theo_adapt}. Similarly to Corollary \ref{cor:unif_conv}, we have uniform convergence over sets of densities with increasing regularity. Recall the definition \reff{eq:def_KrK} of the set $\ck_r(\kappa)$. Let $\crr_n = \{ j, d+1 \leq j \leq R_n\}$, where $R_n$ satisfies the three inequalities:
 \begin{align}
   \label{eq:R_n_1} R_n & \leq N_n + d ,   \\
   \label{eq:R_n_2} R_n & \leq \left\lfloor n^{\inv{2(d+N_n)+1}} \right\rfloor,   \\
   \label{eq:R_n_3} R_n & \leq \frac{\log(n)}{2\log(\log(N_n))} - \inv{2} \cdot
 \end{align}
 
 \begin{cor} \label{cor:adapt_est}
  Under the assumptions of Theorem \ref{theo:adapt}, we get the following result:
 \[
    \lim_{K\rightarrow \infty} \limsup_{n \rightarrow \infty} \sup_{r \in (\crr_n)^d}  \sup_{f^0 \in \ck_r(\kappa)} \P\left(\kl{f^0}{f_{\hat{\lambda}^*_n}} \geq \left( n^{-\frac{2 \min(r)}{2 \min(r)+1}}\right)K\right) =0.
 \]
 \end{cor}

\begin{rem}
   For example when $N_n=\log(n)/(2\log(\log(n)))$, then \reff{eq:R_n_1}, \reff{eq:R_n_2} and \reff{eq:R_n_3} are satisfied with $R_n = N_n$ for $n$ large enough. 
\end{rem}

 \section{Simulation study : random truncation model} \label{sec:simul}

In this section we present the results of Monte Carlo simulation studies on the performance of the additive exponential series estimator. We take the example of the random truncation model introduced in Section \ref{sec:intro} with $d=2$, which is used in many applications. This model naturally satisfies our model assumptions. 

Let $Z=(Z_1,Z_2)$ be a pair of independent random variable with density functions $p_1$, $p_2$ respectively such that $\triangle \subset \supp(p)$, where $p(x_1,x_2) = p_1(x_1)p_2(x_2)$ is the joint density function of $Z$. Suppose that we only observe pairs $(Z_1,Z_2)$ if $0 \leq Z_1 \leq Z_2 \leq 1$. Then the joint density function $f$ of the observable pairs is given by, for $x=(x_1,x_2) \in \R^2$ :
\[
   f(x) = \frac{p_1(x_1)p_2(x_2) }{\int_\triangle p(y) \, dy} \ind_{\triangle}(x).
\]
This corresponds to the form \reff{eq:trunc_joint_dens} with $g_1, g_2$ given by:
\[
  g_1=\frac{p_1 \ind_I}{\int_I p_1}  \quad \text{ and }  \quad g_2=\frac{p_2 \ind_I}{\int_I p_2} \cdot
\]
We will choose the densities $p_1, p_2$ from the following distributions:
\begin{itemize}
 \item  Normal($\mu, \sigma^2$) with $\mu \in \R, \sigma > 0$:
         \[
           f_{\mu,\sigma^2} (t) = \inv{\sqrt{2\pi\sigma^2}} \expp{-\frac{(t-\mu)^2}{2\sigma^2}},
         \]
 \item NormalMix($\mu_1, \sigma_1^2, \mu_2, \sigma_2^2, w$) with $w \in (0,1)$:
         \[
            f(t) = w f_{\mu_1,\sigma_1^2} (t) + (1-w)f_{\mu_2,\sigma_2^2} (t)  ,
         \]
 \item Beta($\alpha, \beta, a,b$) with $0<\alpha<\beta$, $a < 0$, $b>1$ :
        \[
           f(t) = \frac{(t-a)^{\alpha-1}(b-t)^{\beta-\alpha-1}}{(b-a)^{\beta-1} B(\alpha,\beta-\alpha)} \ind_{(a,b)}(t),
        \]
\item Gumbel($\alpha,\beta$) with $\alpha>0$, $\beta \in \R$:
        \[
           f(t) = \alpha \expp{-\alpha(t-\beta)-\expp{-\alpha(t-\beta)}}.
        \]

\end{itemize}

The exact choices for densities $p_1, p_2$ are given in Table \ref{tab:marg_param}. Figure \ref{fig:marginal_pdf} shows the resulting density functions $g_1$ and $g_2$ for each case.
% while Figure \ref{fig:joint_pdf} shows the corresponding joint density functions.

\begin{table}[htbp]
\begin{center}
\begin{tabular}{|c|c|c|}\hline
Model & $p_1$ & $p_2$ \\\hline
 Beta & $\text{Beta}(1,6,-1,2)$ & $\text{Beta}(3,5,-1,2)$ \\
 Gumbel & $\text{Gumbel}(4,0.3)$ &$\text{Gumbel}(2.4,0.7)$ \\
 Normal mix & \text{NormalMix}$(0.2,0.1,0.6,0.1,0.5)$ & $\text{Normal}(0.8,0.2)$ \\
 \hline
\end{tabular}
\end{center}
\caption{Distributions for the left-truncated model used in the simulation study.}
\label{tab:marg_param}
\end{table}

 \begin{figure}[htbp]
 \centering
 \subfigure[Beta]{\includegraphics[width=5cm]{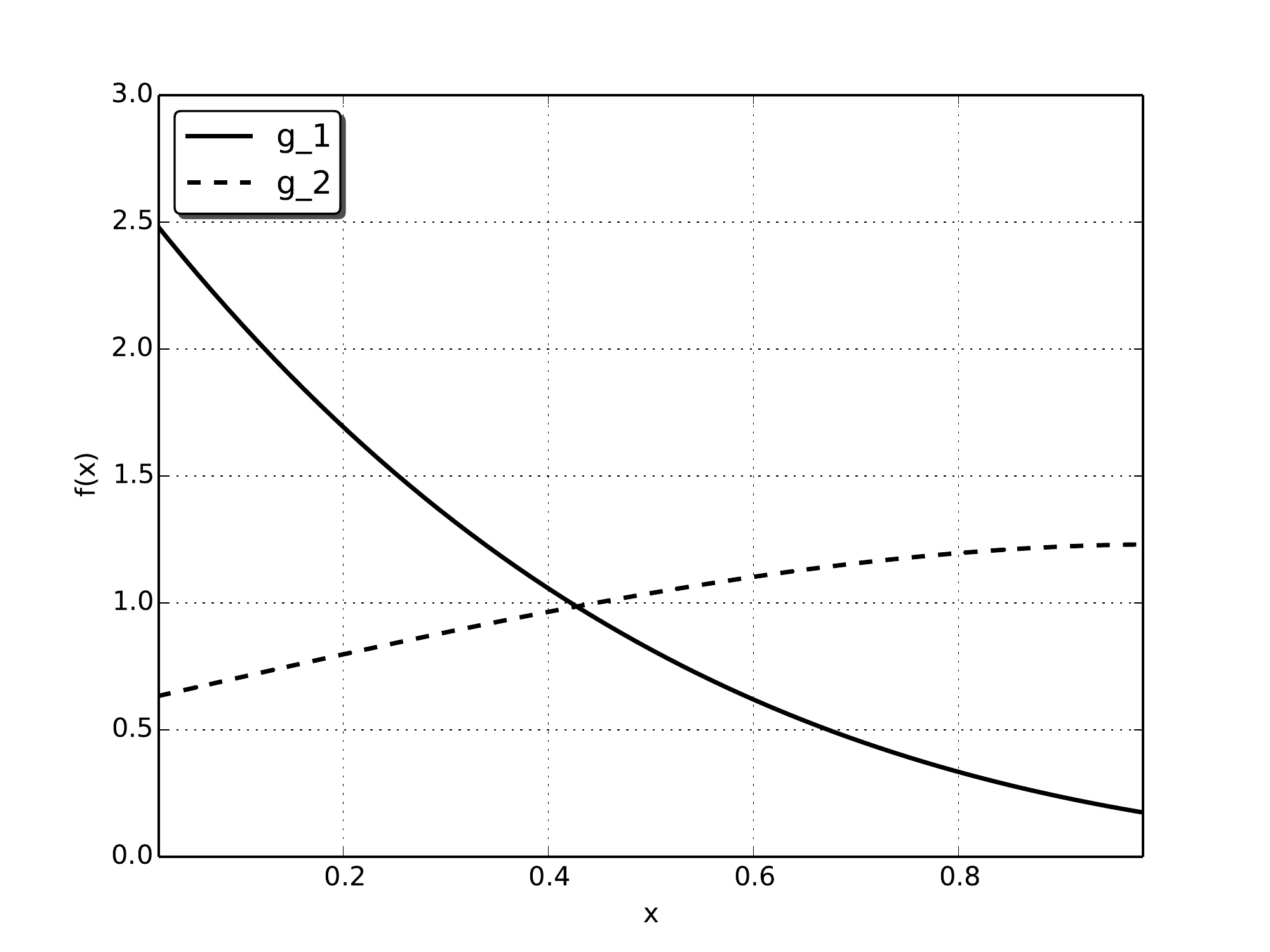}}
 \subfigure[Gumbel]{\includegraphics[width=5cm]{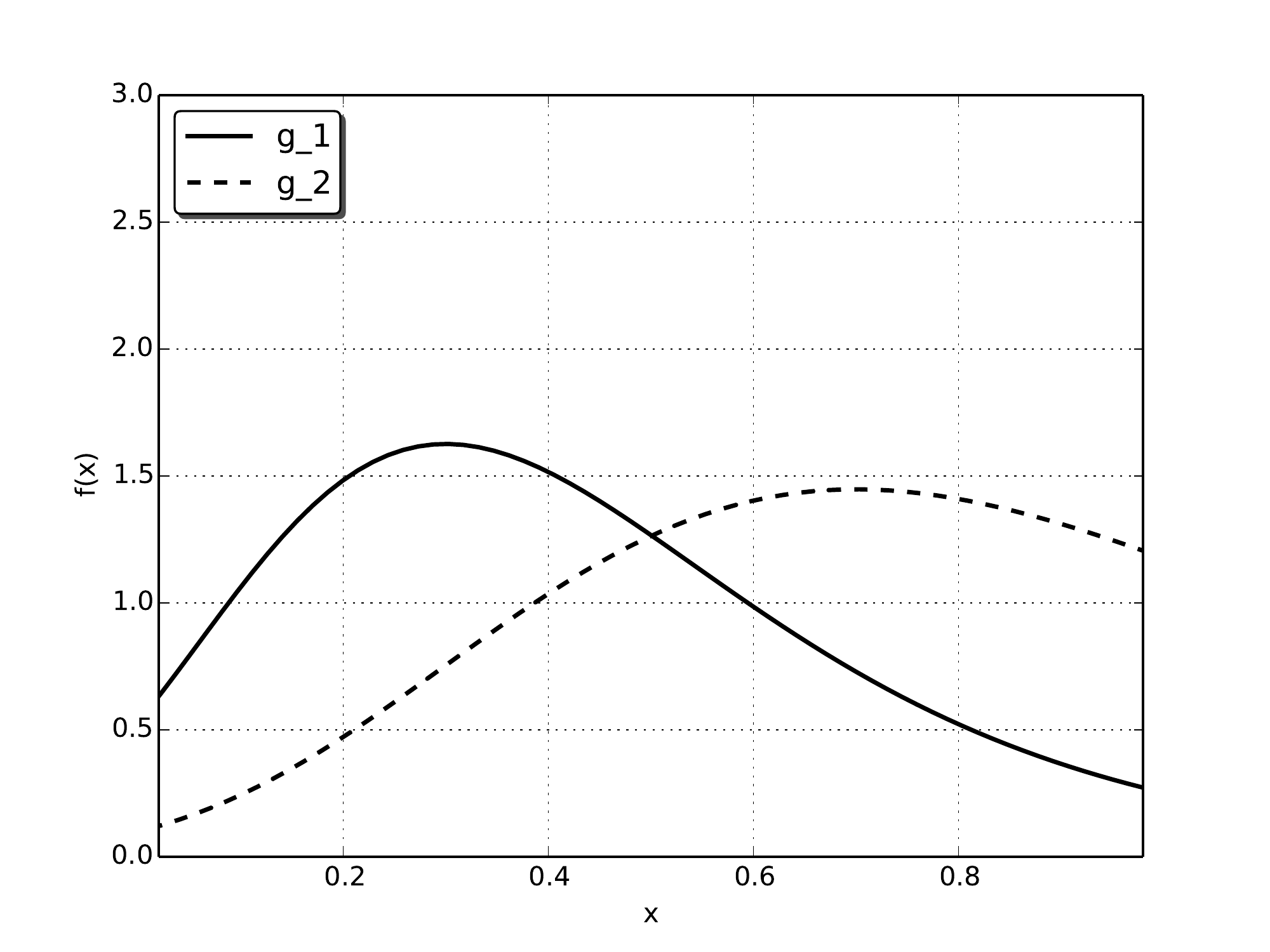}}
 \subfigure[Normal mix]{\includegraphics[width=5cm]{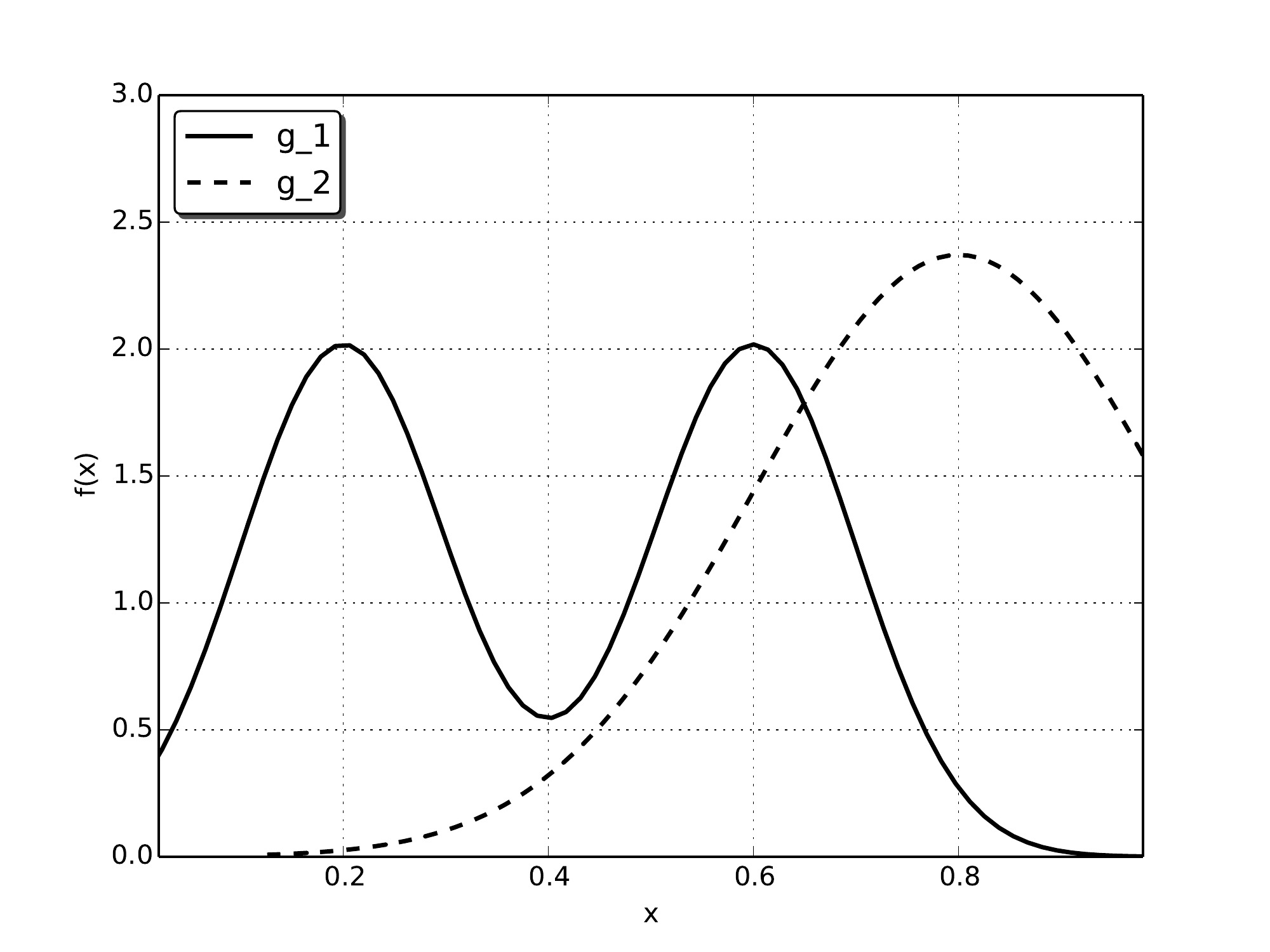}}
\caption{Density functions $g_1, g_2$ of the left-truncated models used in the simulation study.}
\label{fig:marginal_pdf}
\end{figure}

To calculate the parameters $\hat{\theta}_{m,n}$, we recall that $\hat{\theta}_{m,n}$ is the solution of the equation \reff{eq:max_likelihood}, therefore can be also characterized as:
\begin{equation} \label{eq:sim_opt_prob}
   \hat{\theta}_{m,n} = \argmax_{\theta\in \R^{\val{m}}} \theta \cdot \hat{\mu}_{m,n} - \psi(\theta),
\end{equation}
with $\hat{\mu}_{m,n}$ defined by \reff{eq:def_emp_mean}, see Lemma \ref{lem:BS3} . We use a numerical optimisation method to solve \reff{eq:sim_opt_prob} and obtain the parameters $\hat{\theta}_{m,n}$. 
We estimate our model with $m_1=m_2=\bar{m}$, and $\bar{m}=1,2,3,4$. We compute the final estimator based on the convex aggregation method proposed in Section \ref{sec:adaptation}.  
We ran $100$ estimations with increasing sample sizes $n \in \{ 200, 500, 1000\}$, and we calculated the average Kullback-Leibler distance as well as the $L^2$ distance between $f^0$ and its estimator. We used $C_e = 80\%$ of the sample to calculate the initial estimators, and the remaining $20\%$ to perform the aggregation.  The distances were calculated by numerical integration. We compare the results with a truncated kernel density estimator with Gaussian kernel functions and bandwidth selection based on Scott's rule. The results are summarized in Table \ref{tab:est_results_kl} and Table \ref{tab:est_results_L2}.

\begin{table}[htbp]
\caption{Average Kullback-Leibler distances for the additive exponential series estimator (AESE) and the truncated kernel estimator (Kernel) based on $100$ samples of size $n$. Variances provided in parenthesis.}
\begin{center}
\begin{tabular}{|l|c|c|c|c|c|c|}
\hline
\multicolumn{ 1}{|c|}{KL distances} & \multicolumn{ 2}{c|}{n=200} & \multicolumn{ 2}{c|}{n=500} & \multicolumn{ 2}{c|}{n=1000} \\ \cline{ 2- 7}
\multicolumn{ 1}{|l|}{} & \multicolumn{1}{c|}{AESE} & \multicolumn{1}{c|}{Kernel} & \multicolumn{1}{c|}{AESE} & \multicolumn{1}{c|}{Kernel} & \multicolumn{1}{c|}{AESE} & \multicolumn{1}{c|}{Kernel} \\ \hline
\multicolumn{ 1}{|c|}{Beta} & 0.0137 & 0.0524 & 0.0048 & 0.0395 & 0.0028& 0.0339 \\ %\cline{ 2- 7}
\multicolumn{ 1}{|l|}{} &\small (8.94E-05) & \small(1.73E-04) & \small(9.51E-06) & \small(4.61E-05) & \small(3.50E-06) & \small(2.14E-05) \\ \hline
\multicolumn{ 1}{|c|}{Gumbel} & 0.0204 & 0.0249 & 0.0089 & 0.0180 & 0.0050 & 0.0154 \\ %\cline{ 2- 7}
\multicolumn{ 1}{|l|}{} & \small(1.48E-04) & \small(8.03E-05) & \small(2.88E-05) & \small(2.07E-05) &\small (6.70E-06) &\small (1.03E-05) \\ \hline
\multicolumn{ 1}{|c|}{Normal mix} & 0.0545 & 0.0774 & 0.0337 & 0.0559 & 0.0259 & 0.0433 \\ %\cline{ 2- 7}
\multicolumn{ 1}{|l|}{} & \small(4.51E-04) & \small(7.29E-05) & \small(1.88E-04) & \small(2.95E-05) & \small(2.50E-05) & \small(1.52E-05) \\ \hline

\end{tabular}
\end{center}
\label{tab:est_results_kl}
\end{table}

\begin{table}[htbp]
\caption{Average $L^2$ distances for the additive exponential series estimator (AESE) and the truncated kernel estimator (Kernel) based on $100$ samples of size $n$. Variances provided in parenthesis.}
\begin{center}
\begin{tabular}{|l|c|c|c|c|c|c|}
\hline
\multicolumn{ 1}{|c|}{$\mathbb{L}^2$  distances} & \multicolumn{ 2}{c|}{n=200} & \multicolumn{ 2}{c|}{n=500} & \multicolumn{ 2}{c|}{n=1000} \\ \cline{ 2- 7}
\multicolumn{ 1}{|l|}{} & \multicolumn{1}{c|}{AESE} & \multicolumn{1}{c|}{Kernel} & \multicolumn{1}{c|}{AESE} & \multicolumn{1}{c|}{Kernel} & \multicolumn{1}{c|}{AESE} & \multicolumn{1}{c|}{Kernel} \\ \hline
\multicolumn{ 1}{|c|}{Beta} & 0.0536 & 0.2107 & 0.0200 & 0.1660 & 0.0120 & 0.1429\\ %\cline{ 2- 7}
\multicolumn{ 1}{|l|}{} & \small(1.42E-03) & \small(2.60E-03) & \small(2.27E-04) & \small(8.04E-04) & \small(7.45E-05) & \small(3.52E-04) \\ \hline
\multicolumn{ 1}{|c|}{Gumbel} & 0.0683 & 0.0856 & 0.0297 & 0.0621 & 0.0166 & 0.0522 \\ %\cline{ 2- 7}
\multicolumn{ 1}{|l|}{} & \small(1.95E-03) & \small(9.94E-04) & \small(3.61E-04) & \small(2.49E-04) & \small(8.74E-05) & \small(1.19E-04) \\ \hline
\multicolumn{ 1}{|c|}{Normal mix} & 0.2314 & 0.3534 & 0.1489 & 0.2545 & 0.1112 & 0.1952 \\ %\cline{ 2- 7}
\multicolumn{ 1}{|l|}{} & \small(1.17E-02) & \small(1.43E-03) & \small(5.53E-03) & \small(6.95E-04) & \small(9.25E-04) & \small(3.83E-04) \\ \hline

\end{tabular}
\end{center}
\label{tab:est_results_L2}
\end{table}

%%%% Beta distribution %%%%%

 \begin{figure}[H]
 \centering
 \subfigure{\includegraphics[width=3.5cm]{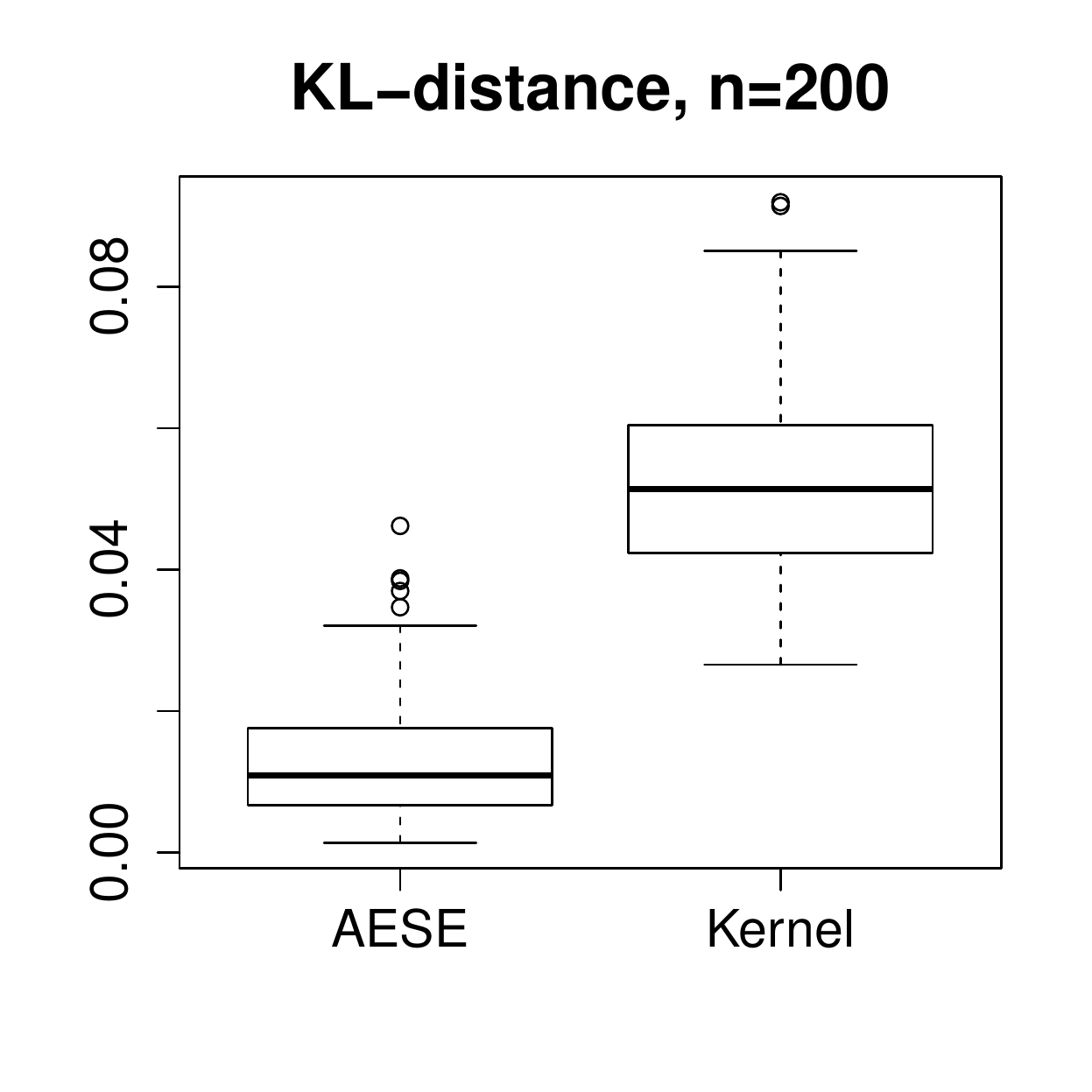}}
 \subfigure{\includegraphics[width=3.5cm]{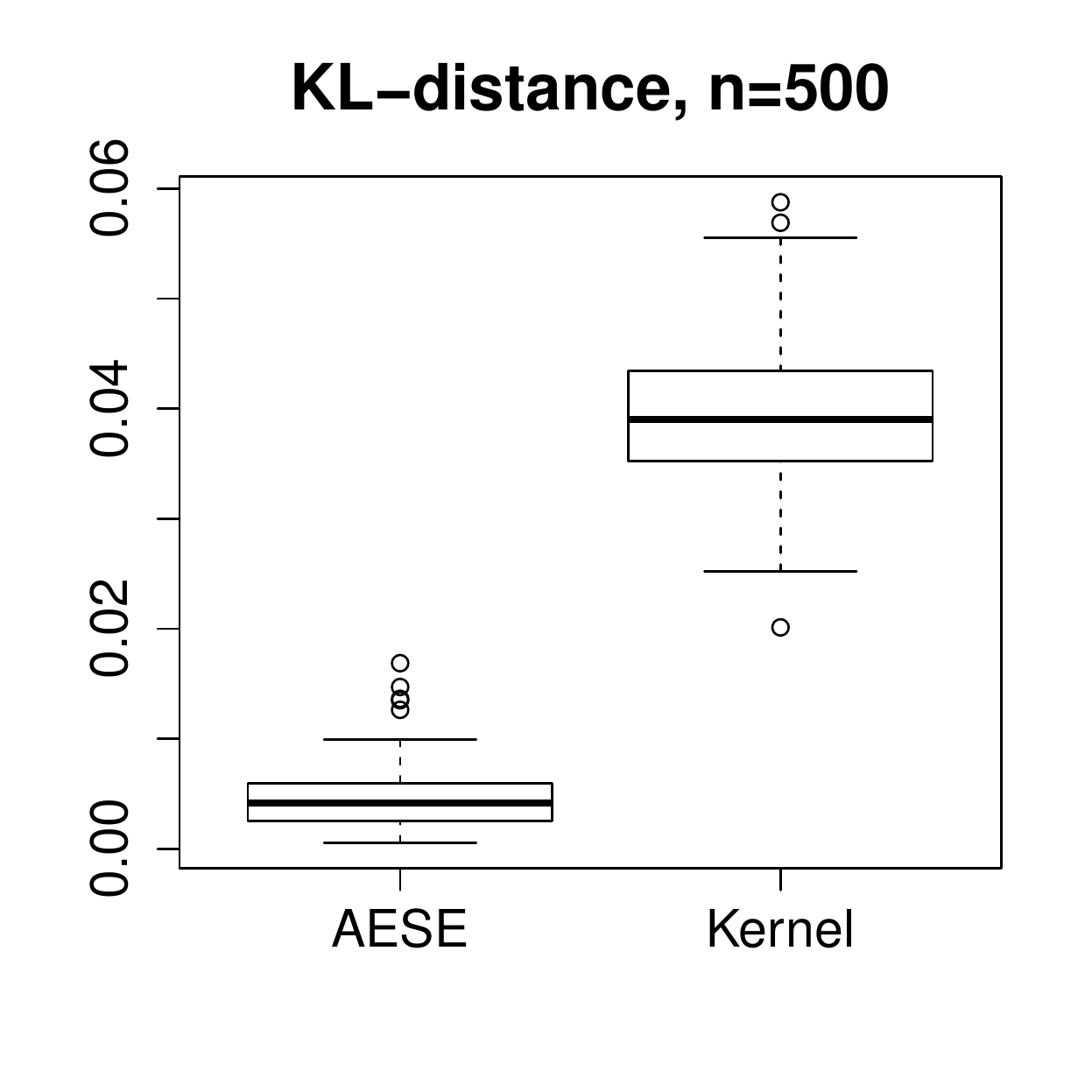}}
 \subfigure{\includegraphics[width=3.5cm]{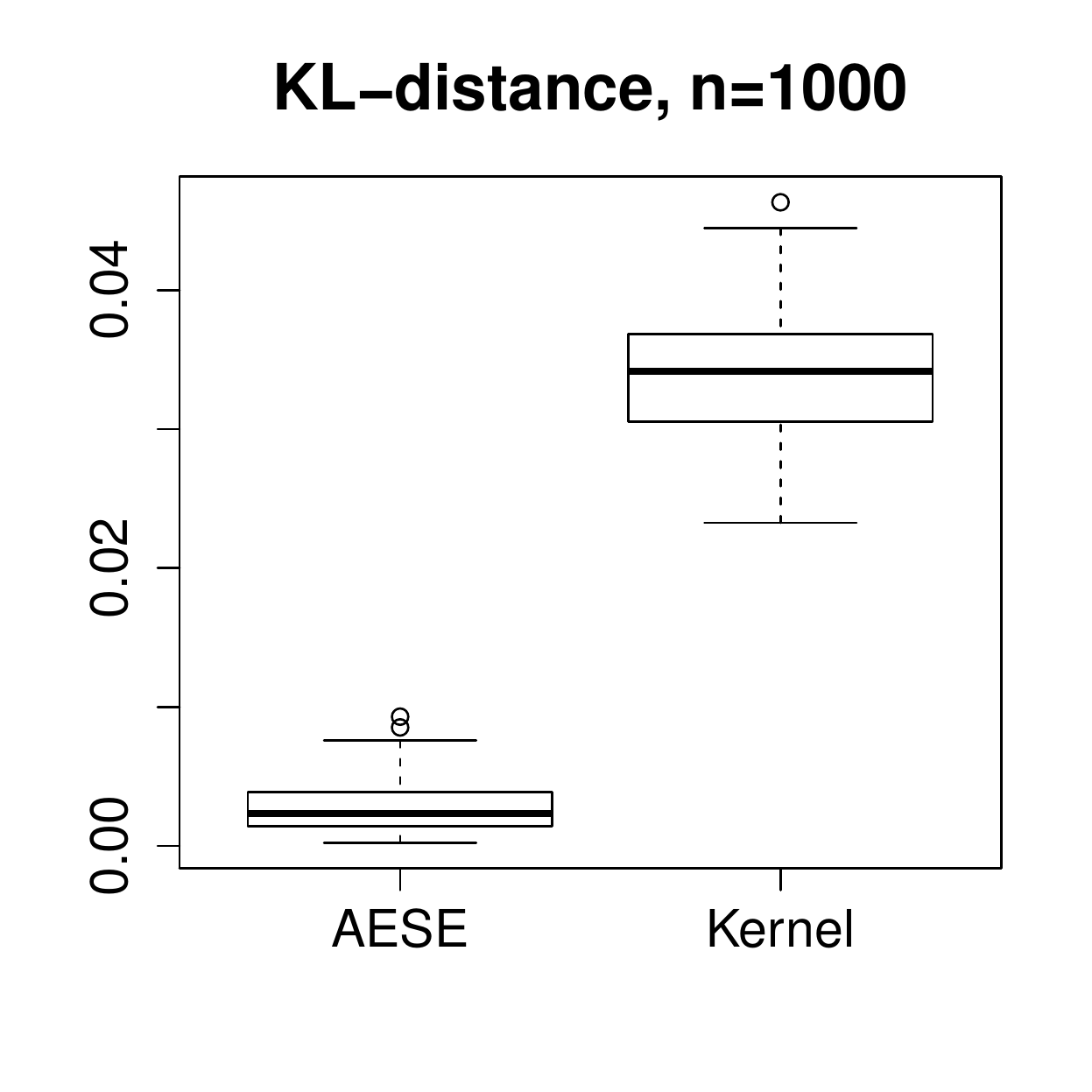}}
 
  \subfigure{\includegraphics[width=3.5cm]{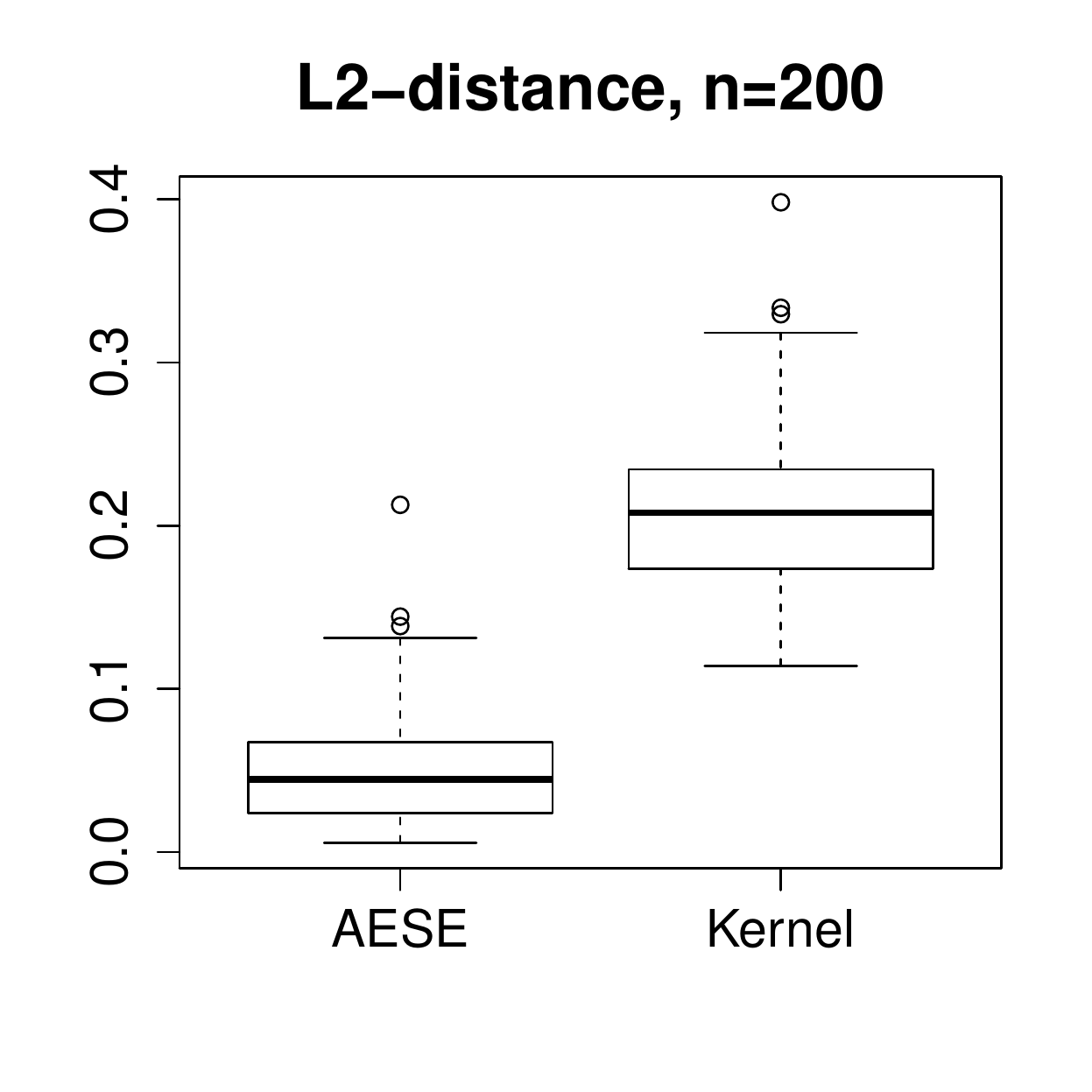}}
 \subfigure{\includegraphics[width=3.5cm]{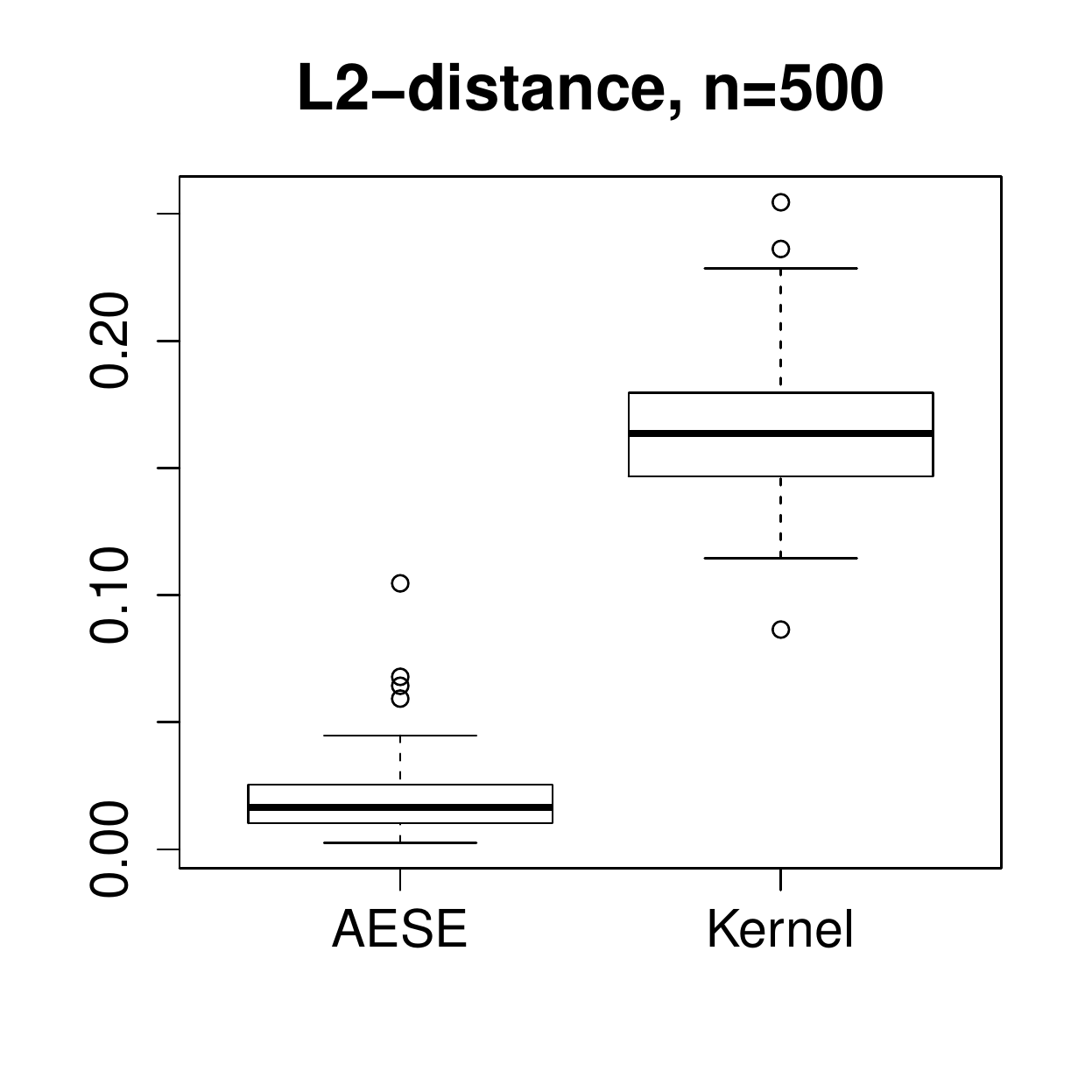}}
 \subfigure{\includegraphics[width=3.5cm]{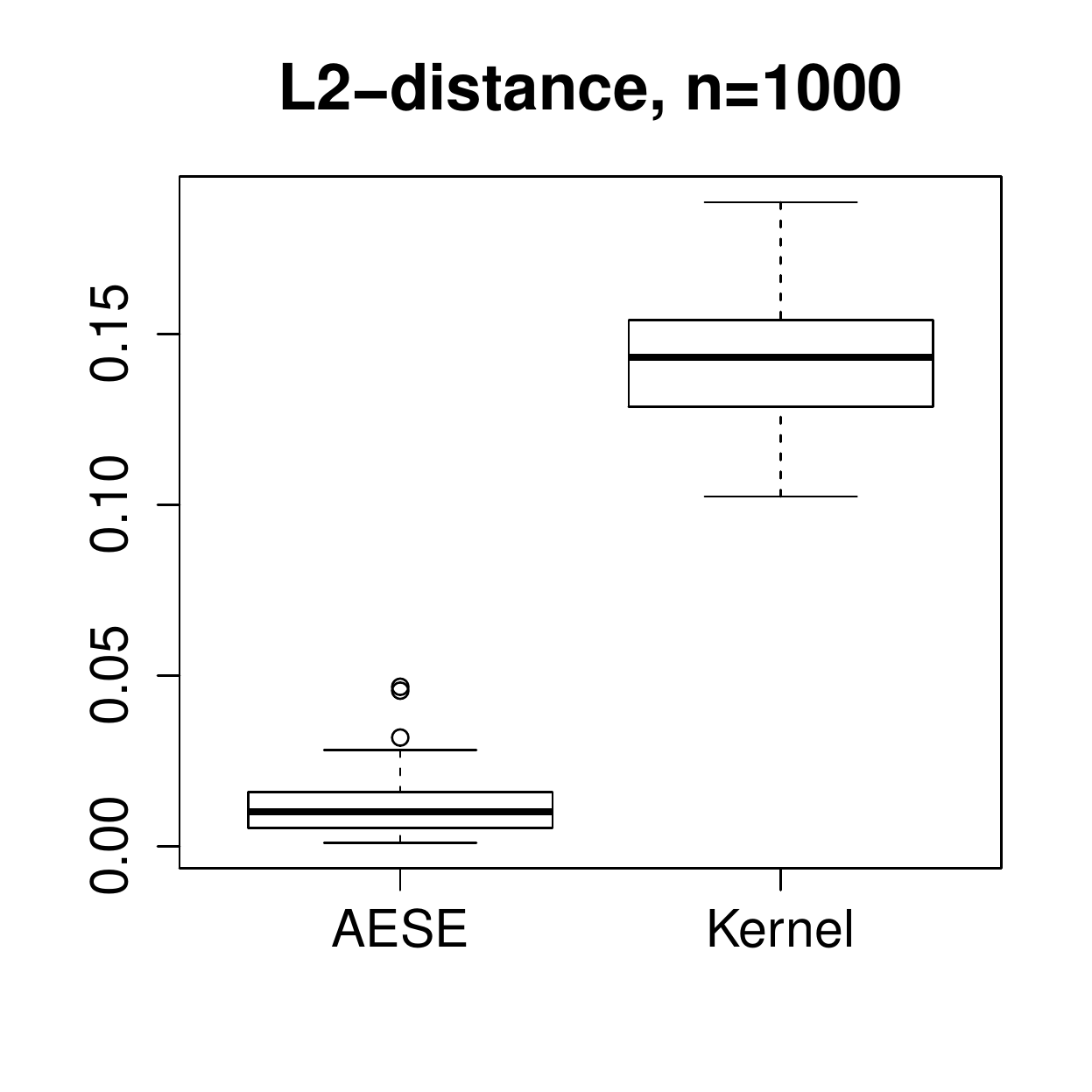}}
 \caption{Boxplot of the Kullback-Leibler and $L^2$ distances for the additive exponential series estimator (AESE) and the truncated kernel estimators with Beta marginals.}
 \label{fig:boxplot_beta_KL}
 \end{figure}

  \begin{figure}[H]
 \centering
 \subfigure[True density]{\includegraphics[width=4cm]{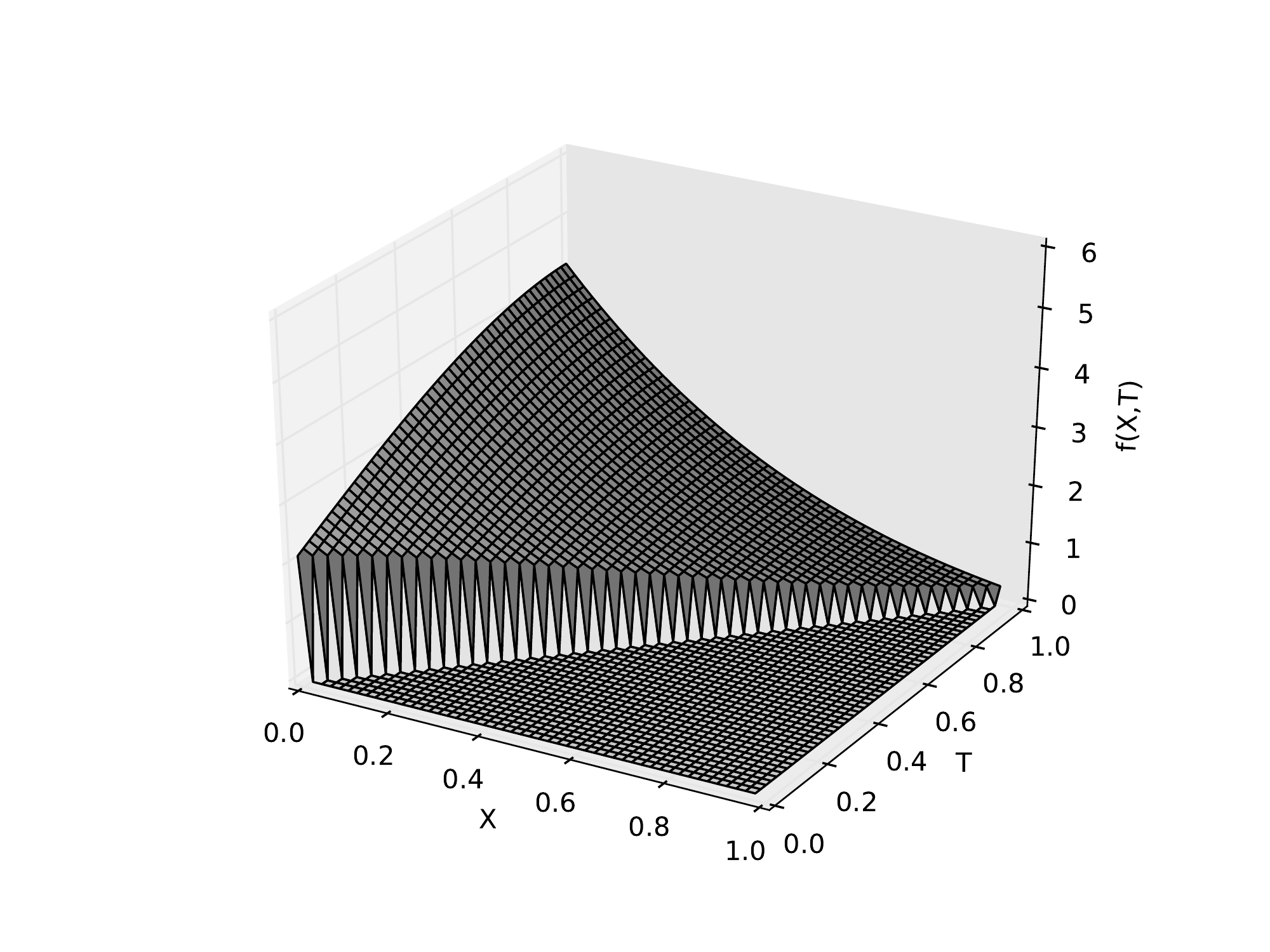}}
 \subfigure[AESE]{\includegraphics[width=4cm]{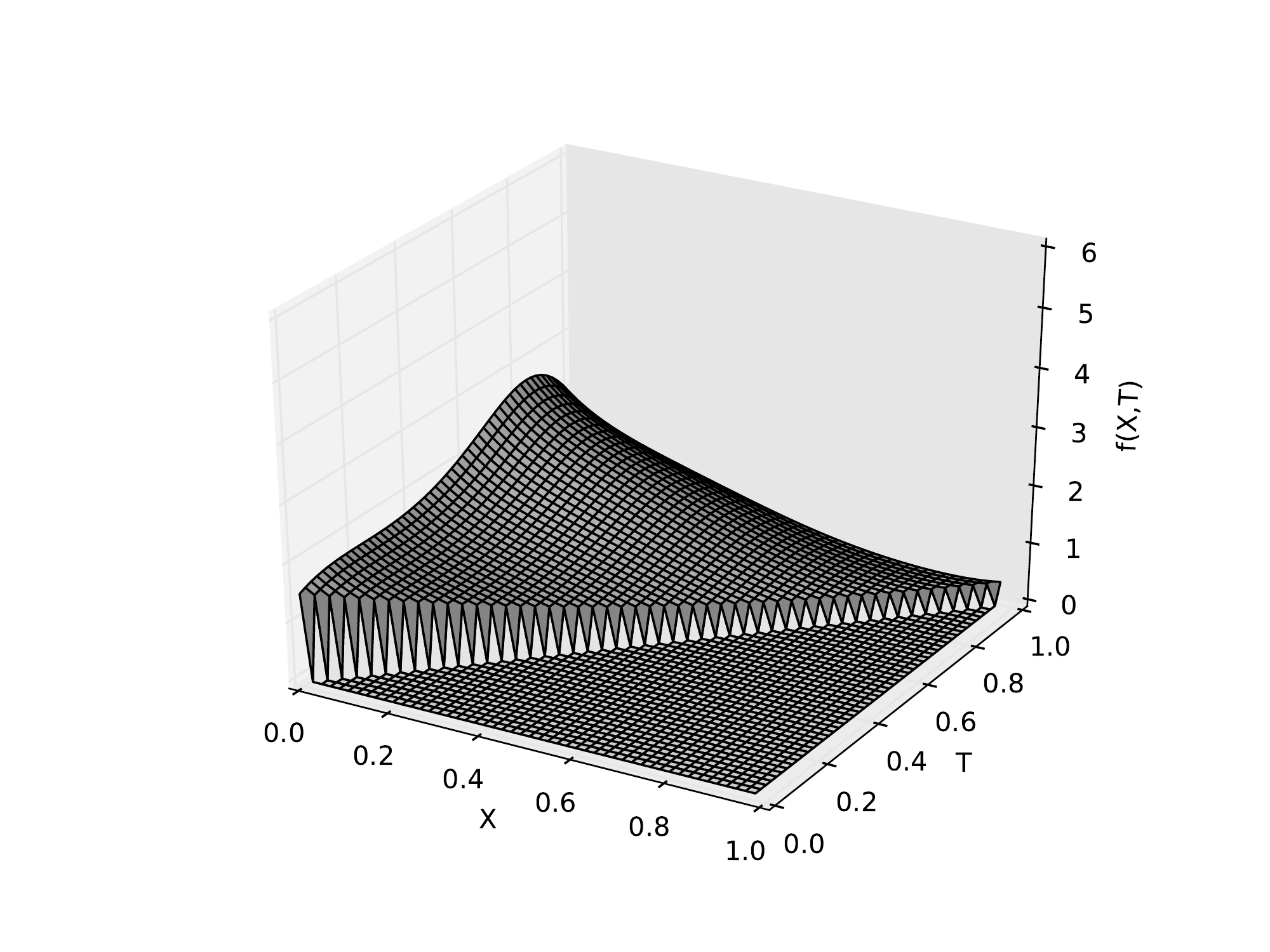}}
 \subfigure[Kernel]{\includegraphics[width=4cm]{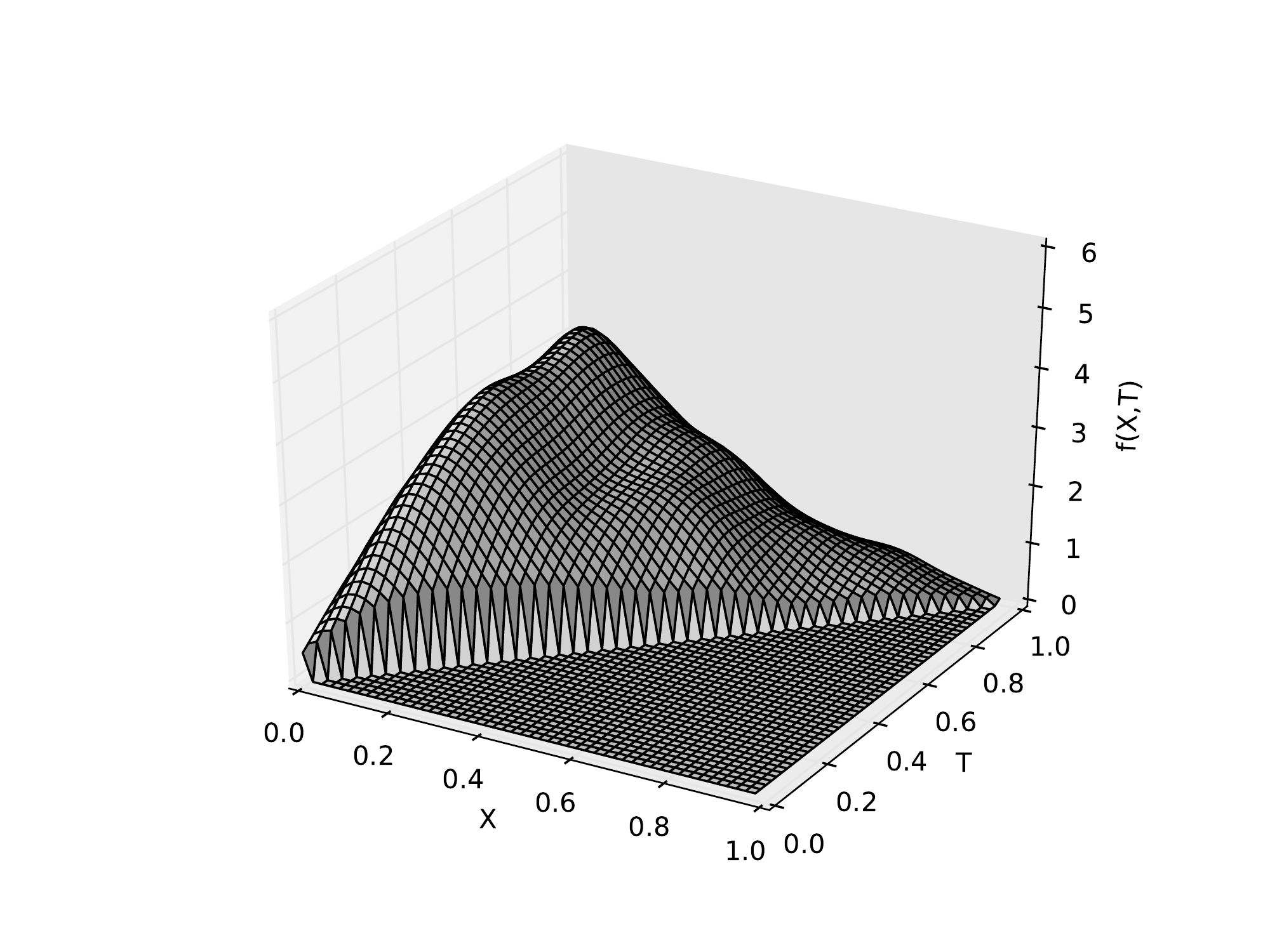}}

 \subfigure[True density]{\includegraphics[width=4cm]{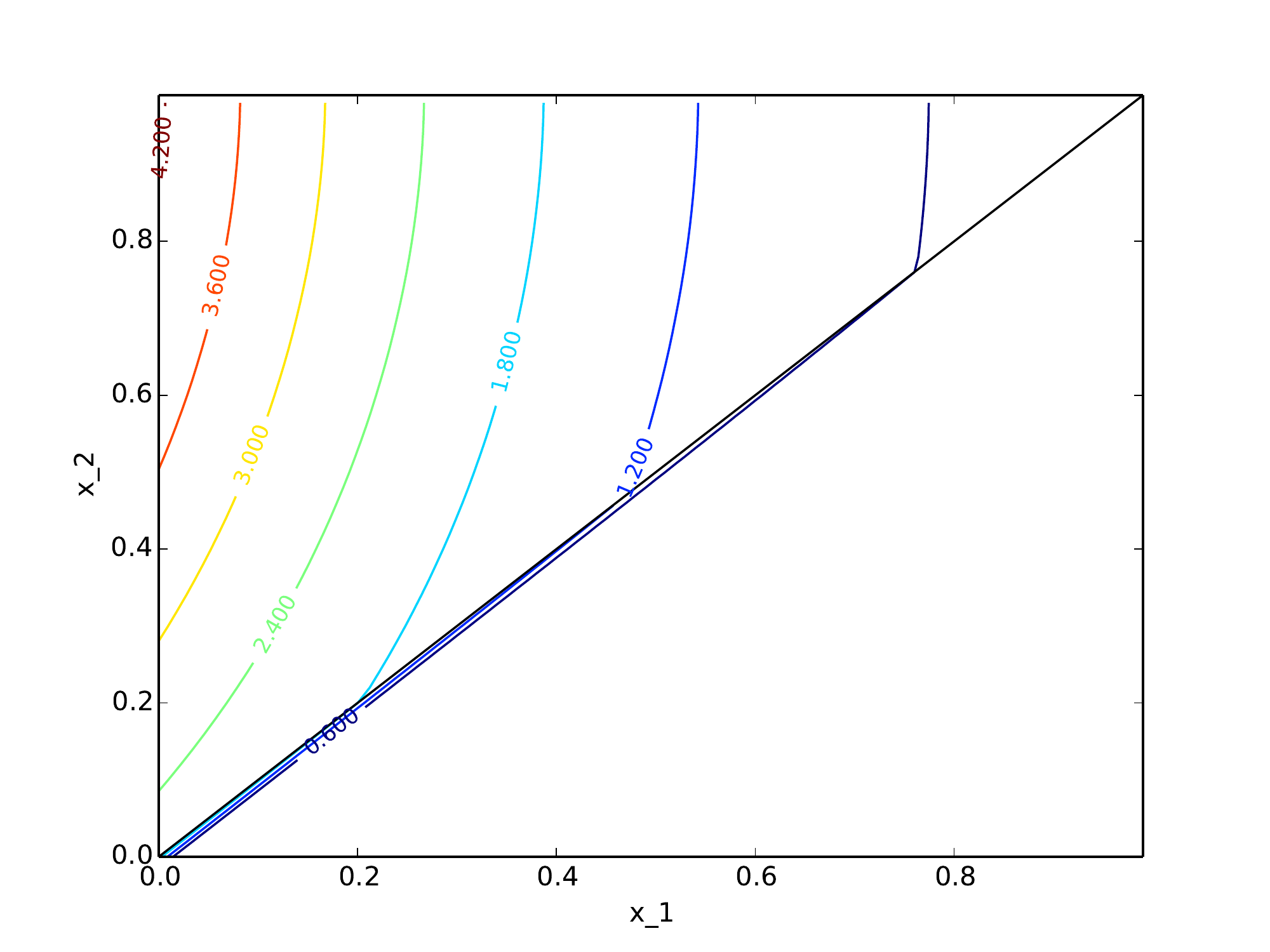}}
 \subfigure[AESE]{\includegraphics[width=4cm]{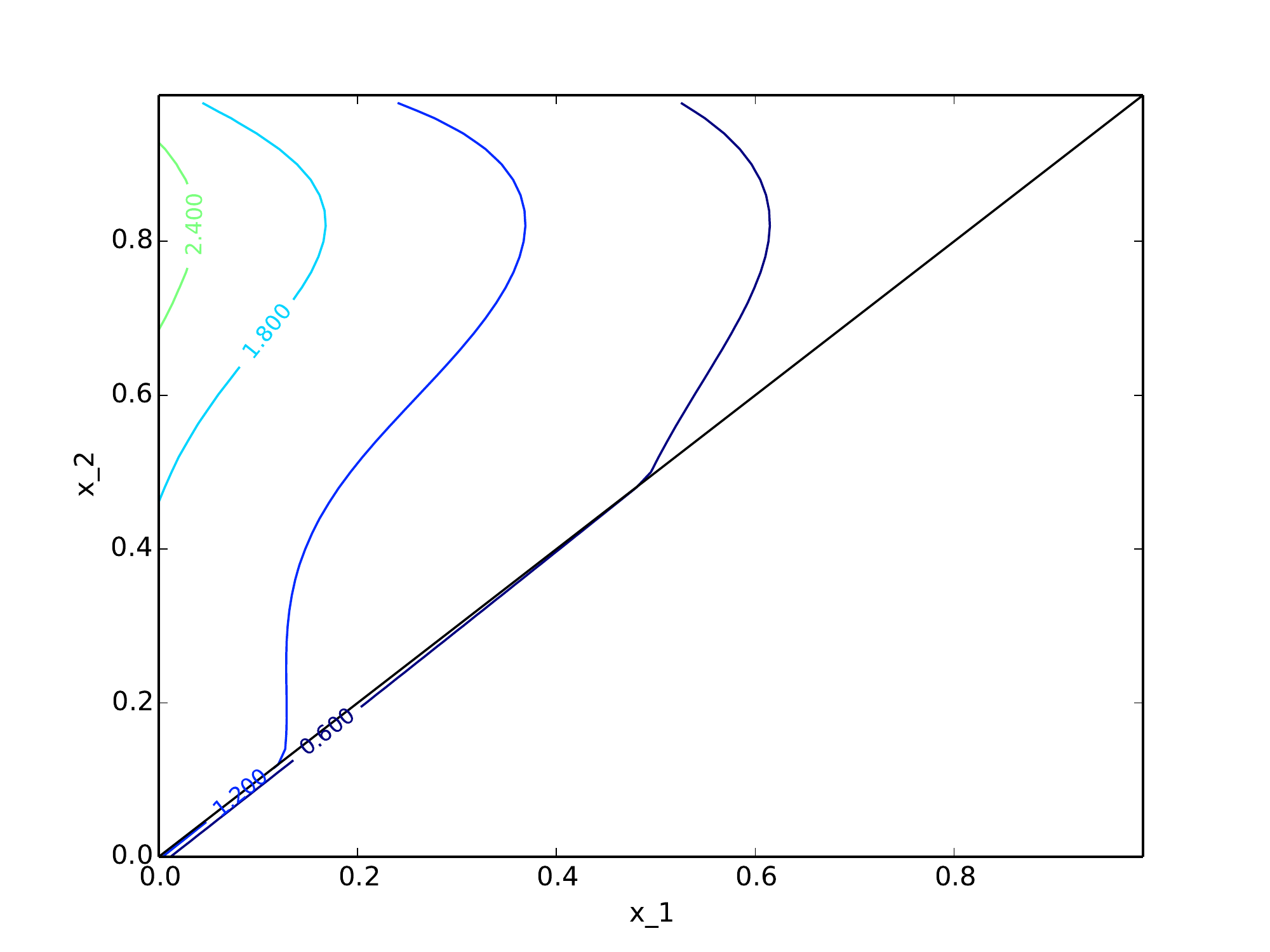}}
 \subfigure[Kernel]{\includegraphics[width=4cm]{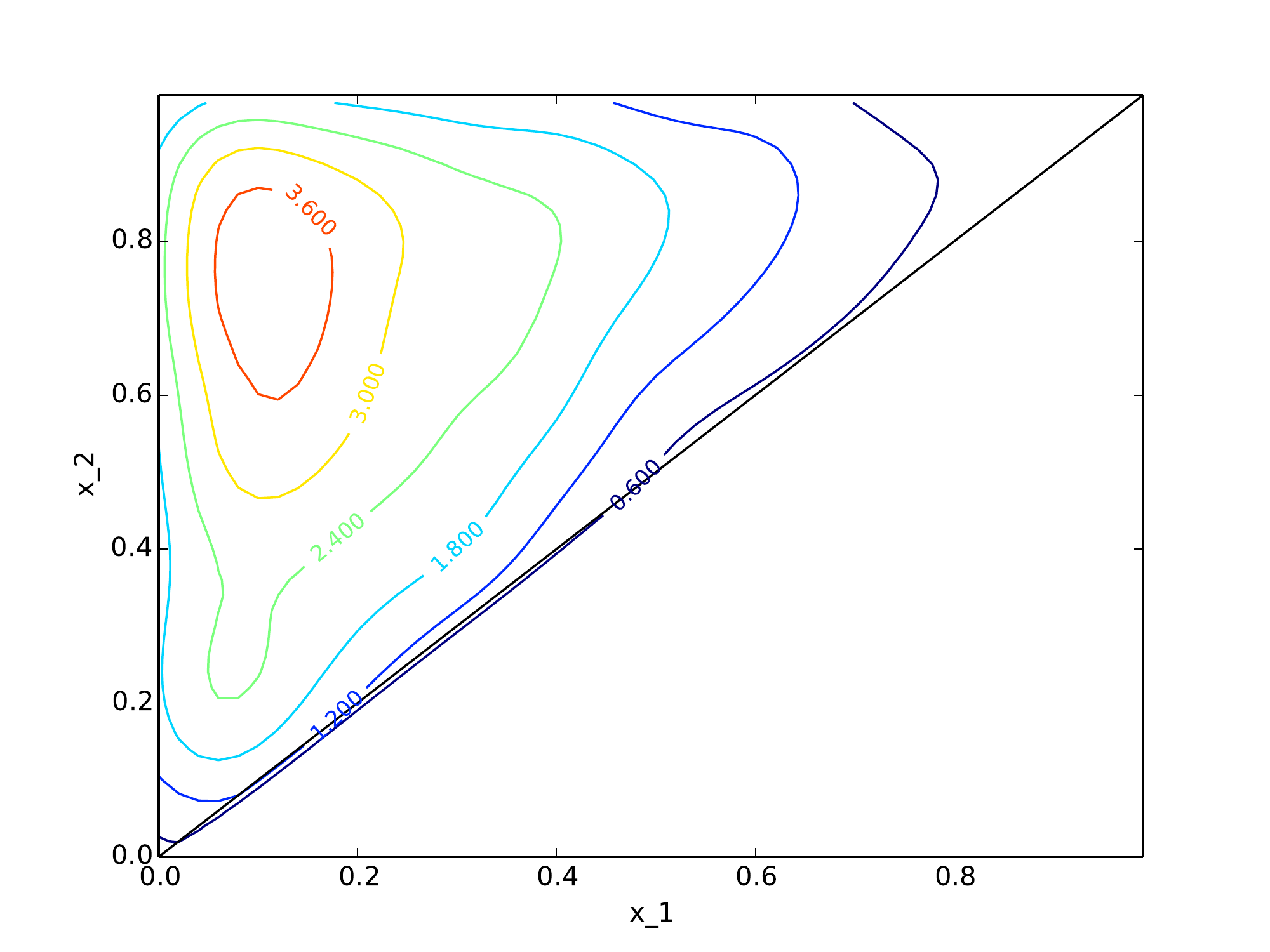}}
 \caption{Joint density functions of the true density and its estimators with Beta marginals.}
 \label{fig:kernel_ESE_beta_cont}
 \end{figure}

 %%%% Gumbel distribution %%%%%%%

 \begin{figure}[H]
 \centering
 \subfigure{\includegraphics[width=3.5cm]{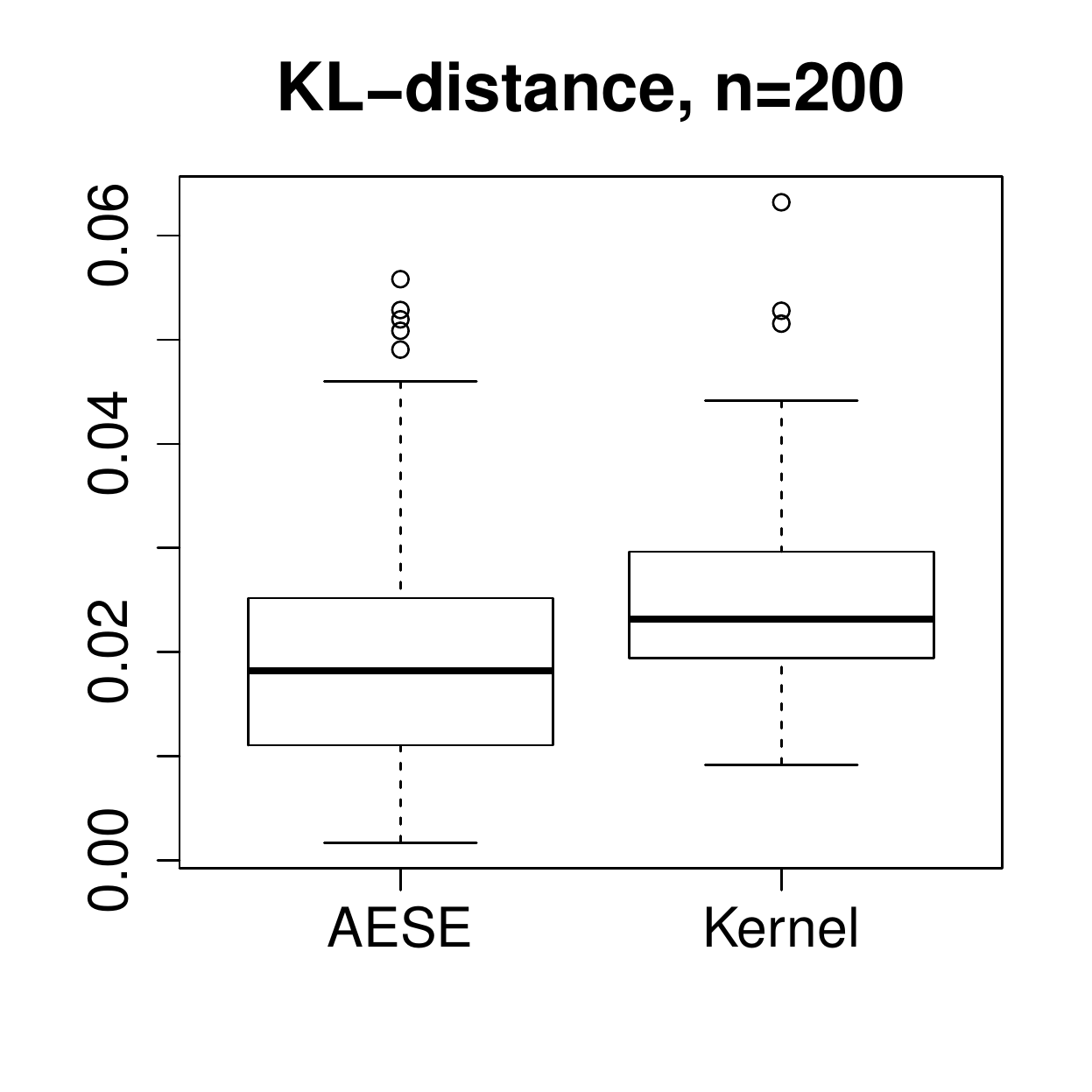}}
 \subfigure{\includegraphics[width=3.5cm]{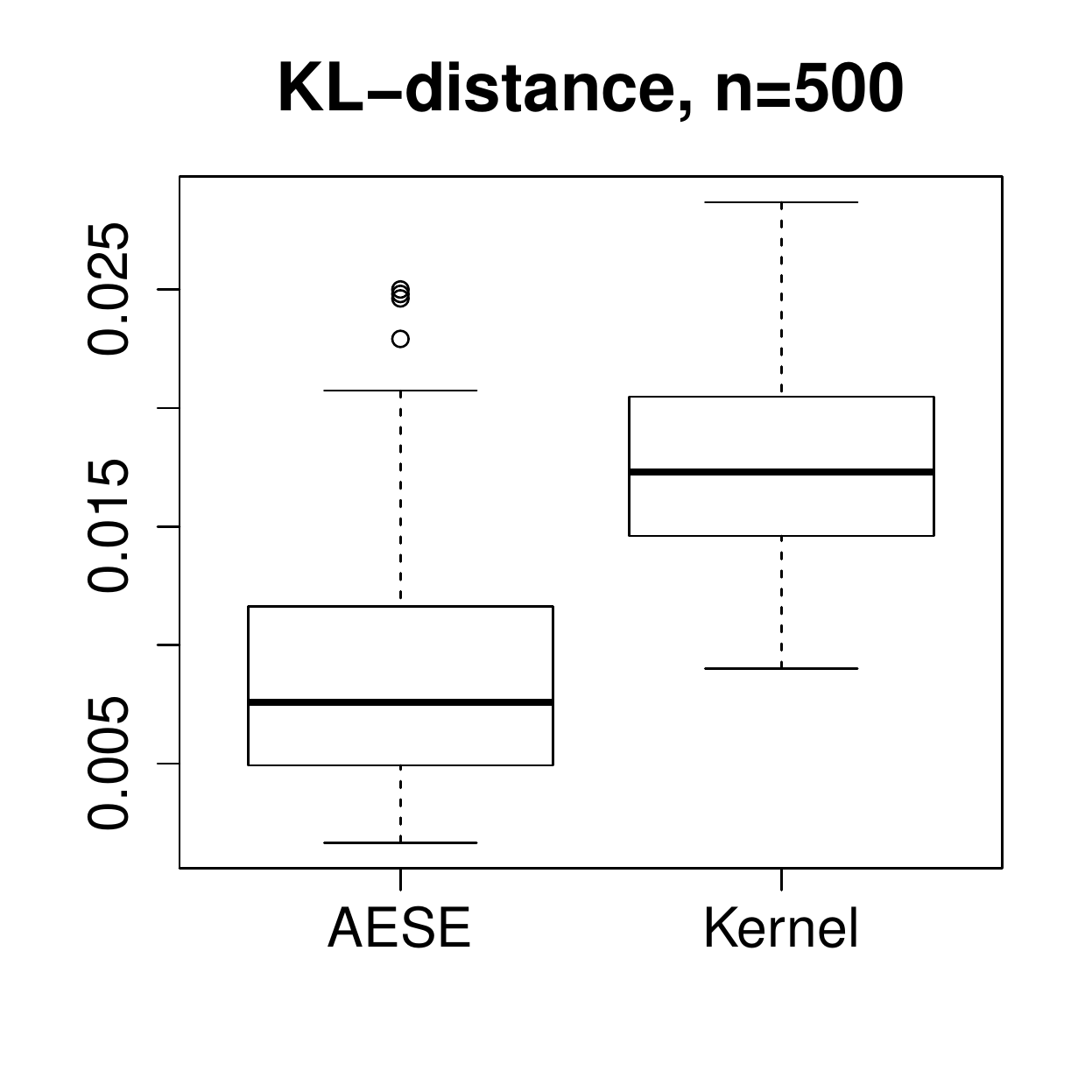}}
 \subfigure{\includegraphics[width=3.5cm]{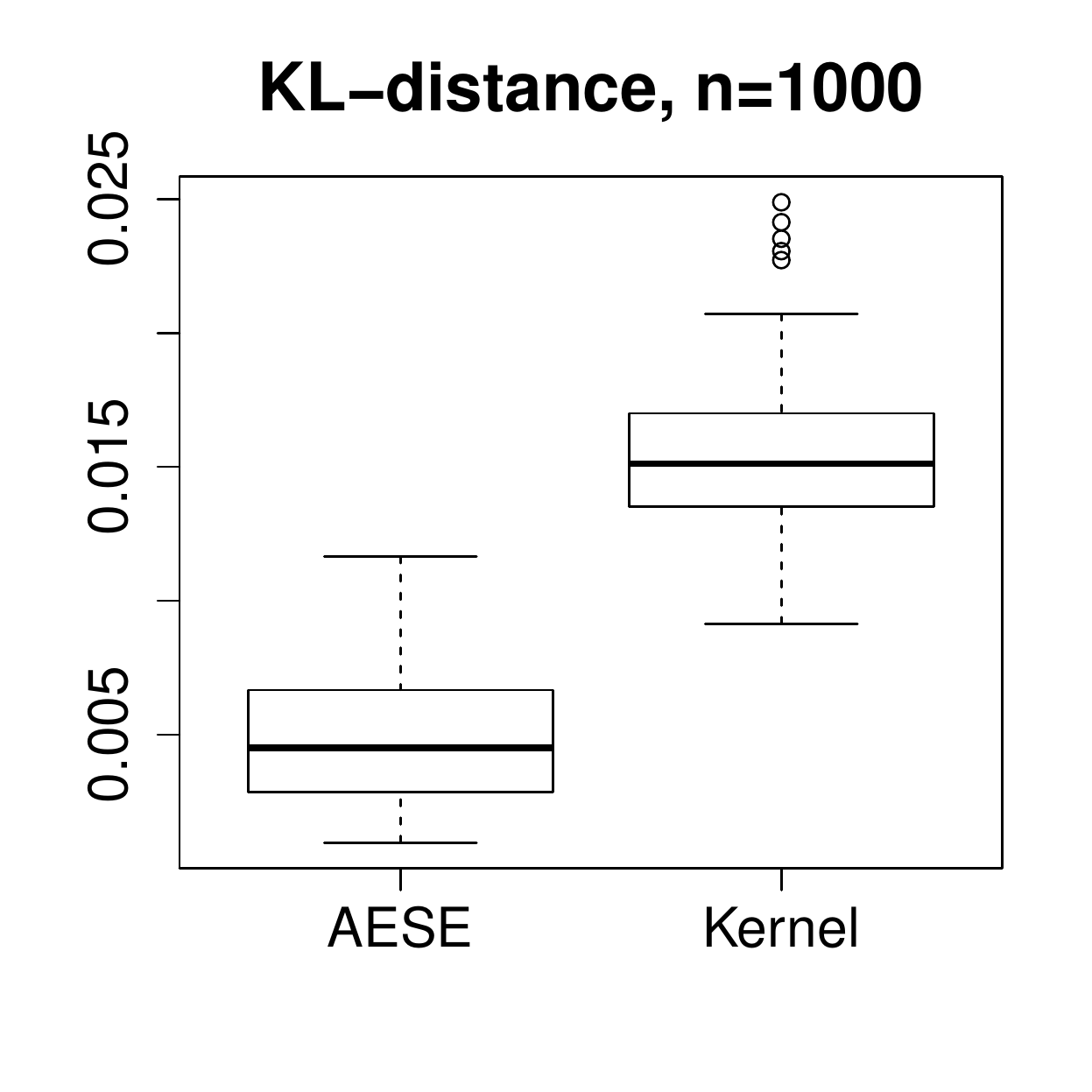}}
 
 \subfigure{\includegraphics[width=3.5cm]{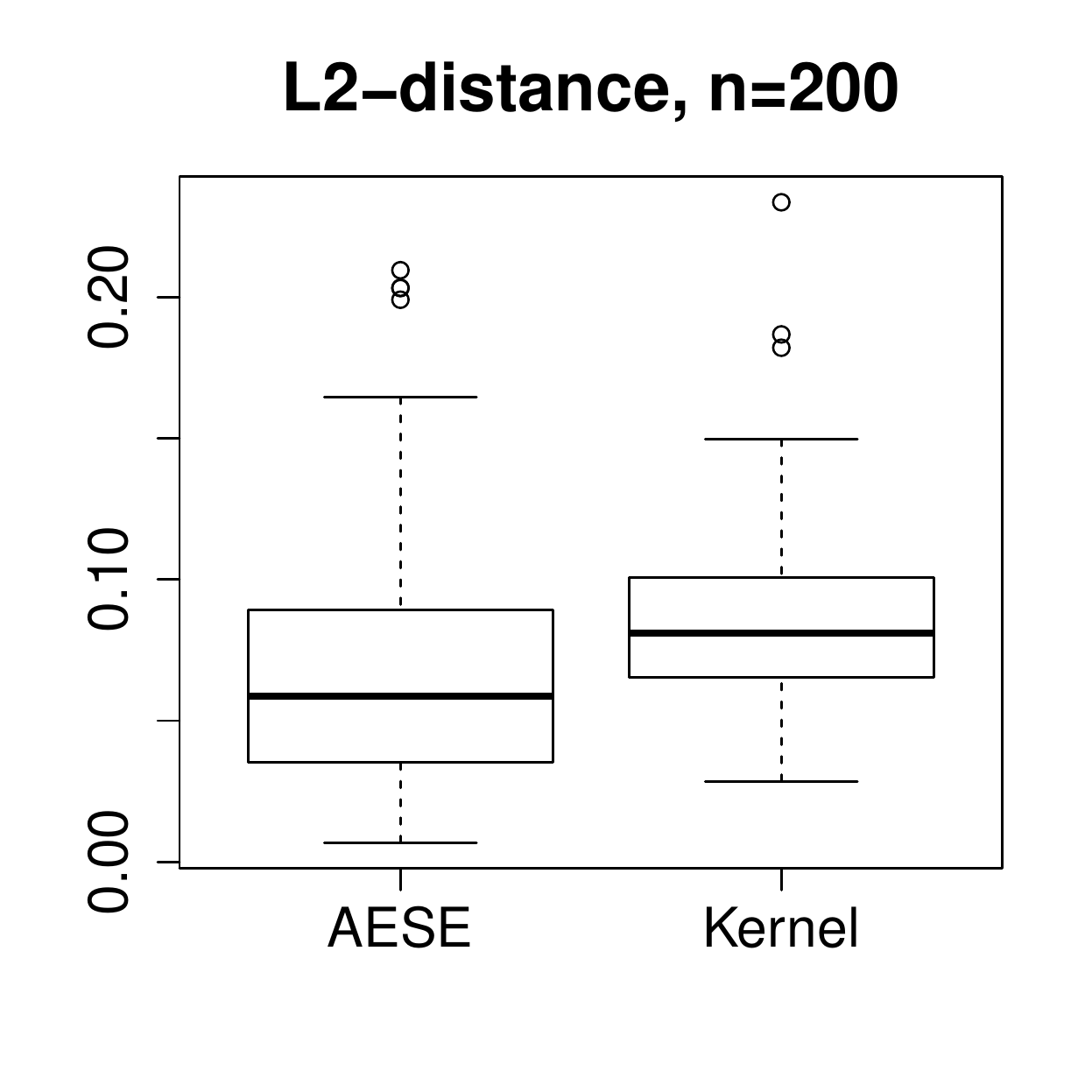}}
 \subfigure{\includegraphics[width=3.5cm]{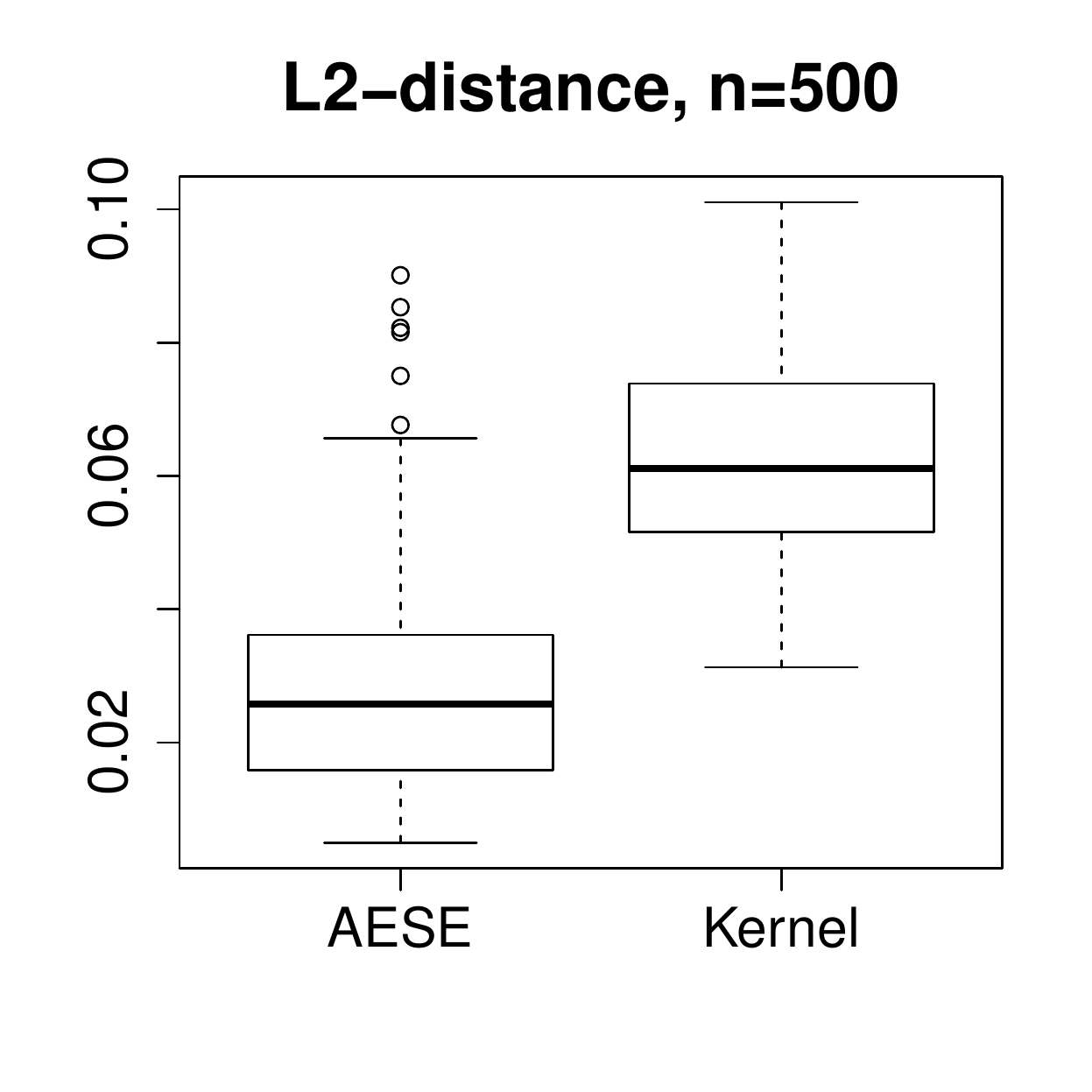}}
 \subfigure{\includegraphics[width=3.5cm]{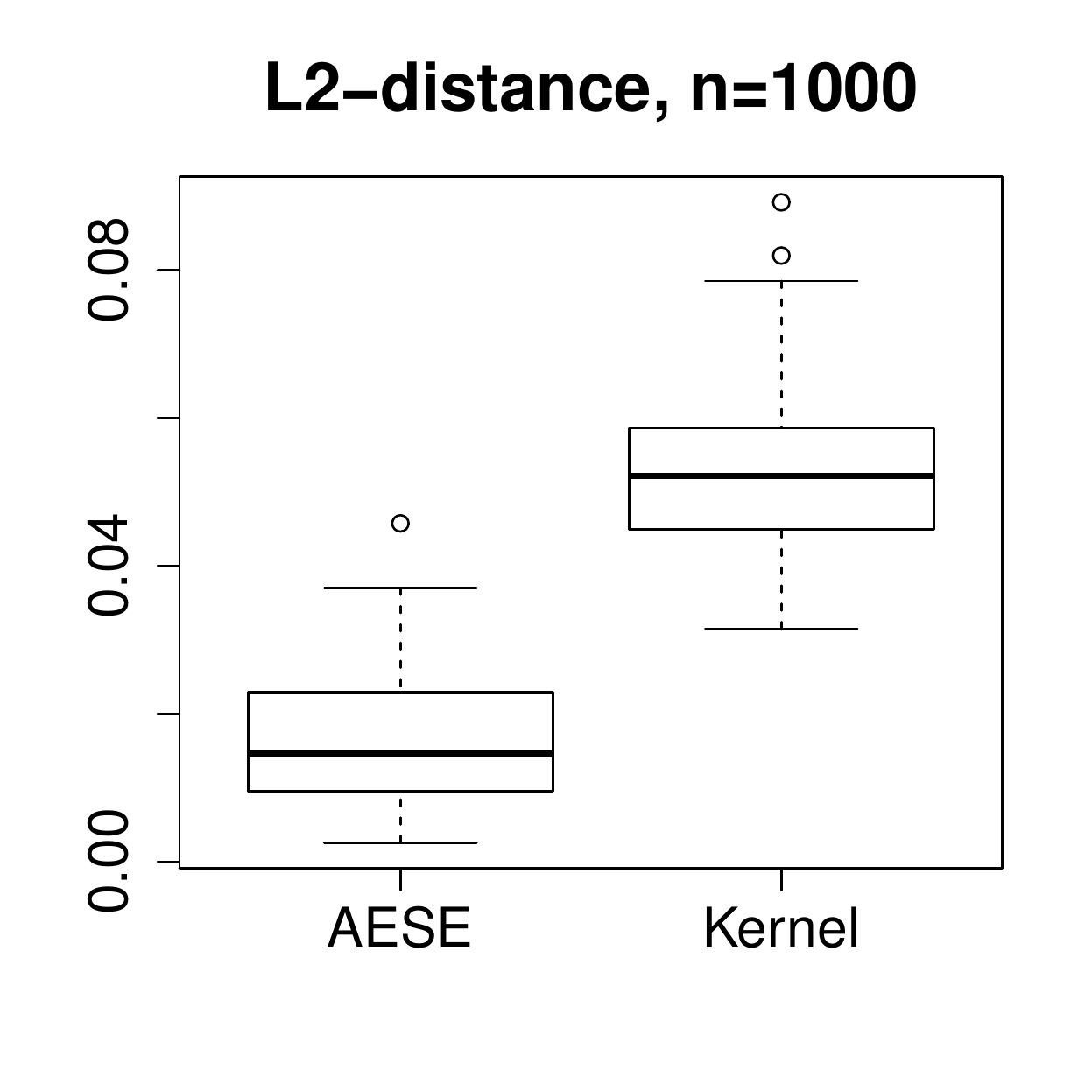}}
 \caption{Boxplot of the Kullback-Leibler and $L^2$ distances for the additive exponential series estimator (AESE) and the truncated kernel estimators with Gumbel marginals.}
 \label{fig:boxplot_gumbel_KL}
 \end{figure}

  \begin{figure}[H]
 \centering
 \subfigure[True density]{\includegraphics[width=4cm]{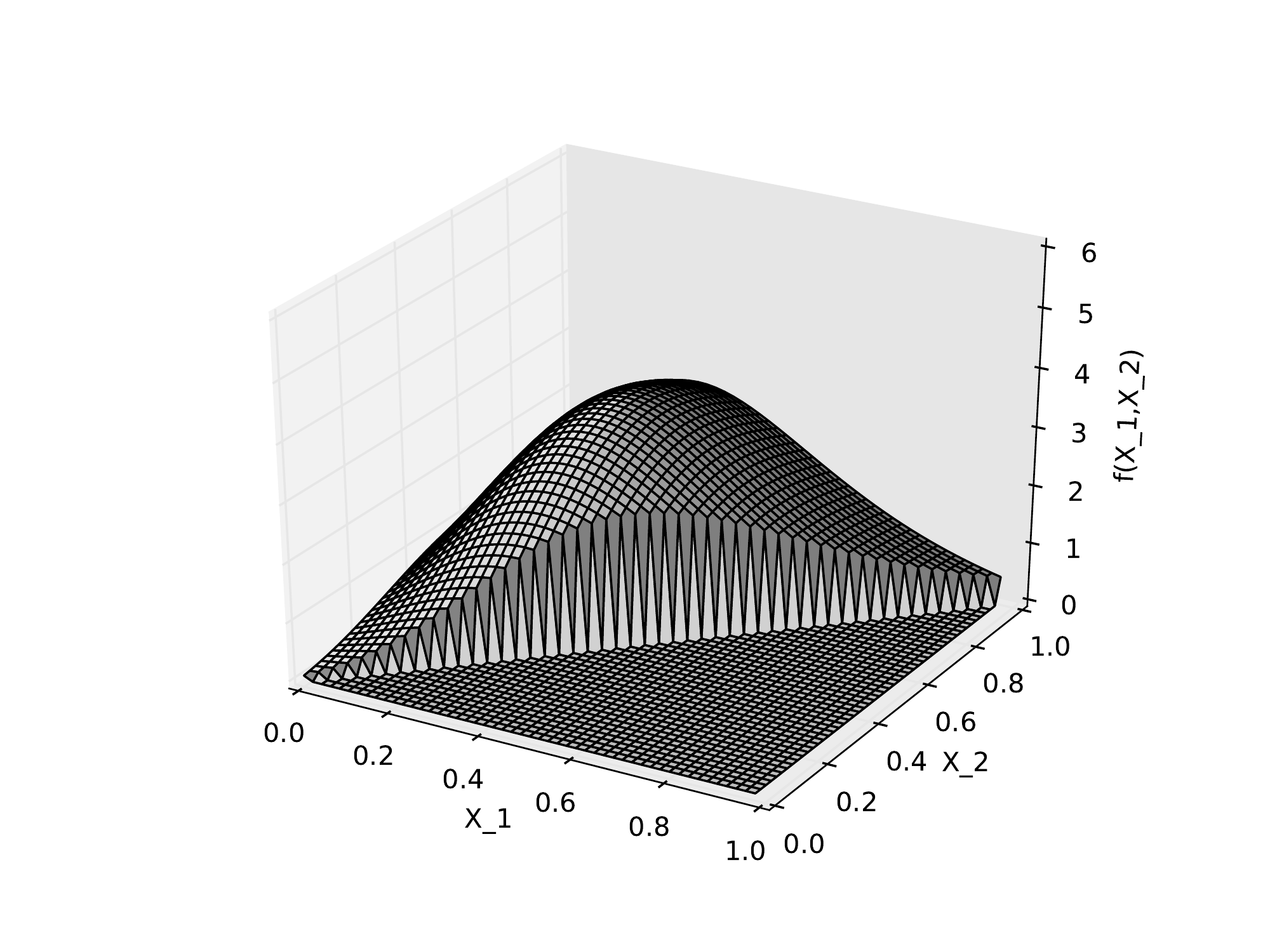}}
 \subfigure[AESE]{\includegraphics[width=4cm]{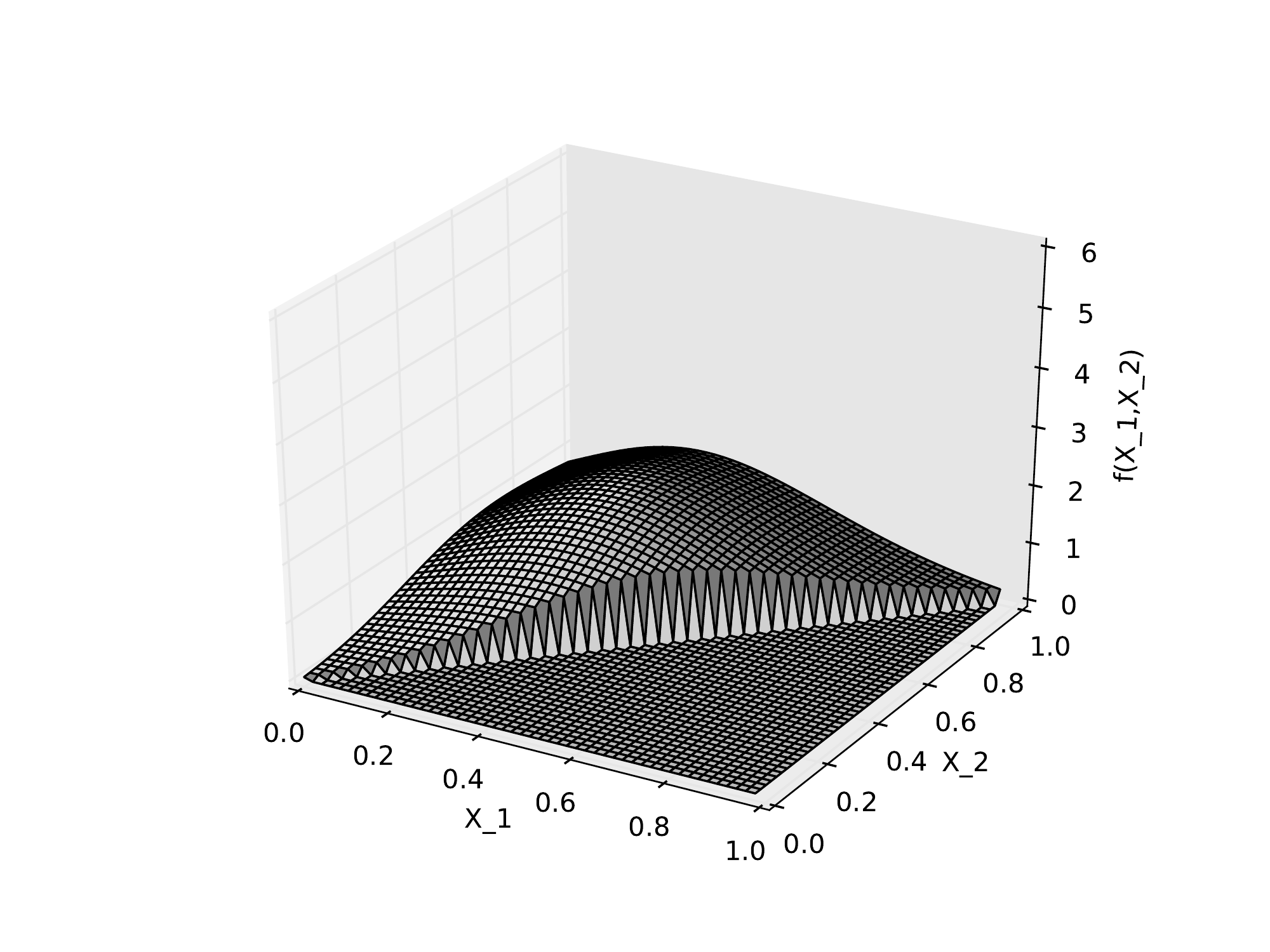}}
 \subfigure[Kernel]{\includegraphics[width=4cm]{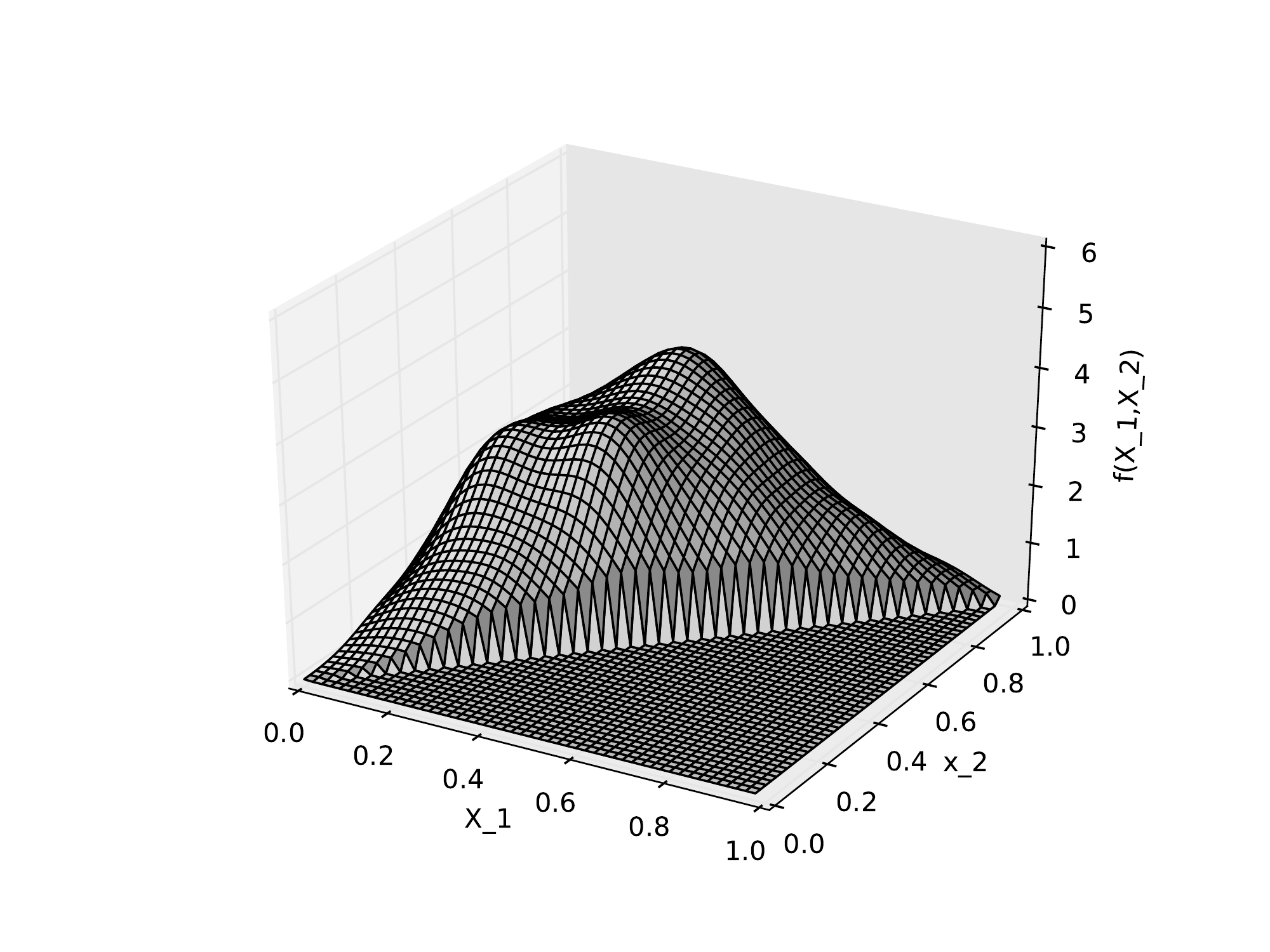}}

 \subfigure[True density]{\includegraphics[width=4cm]{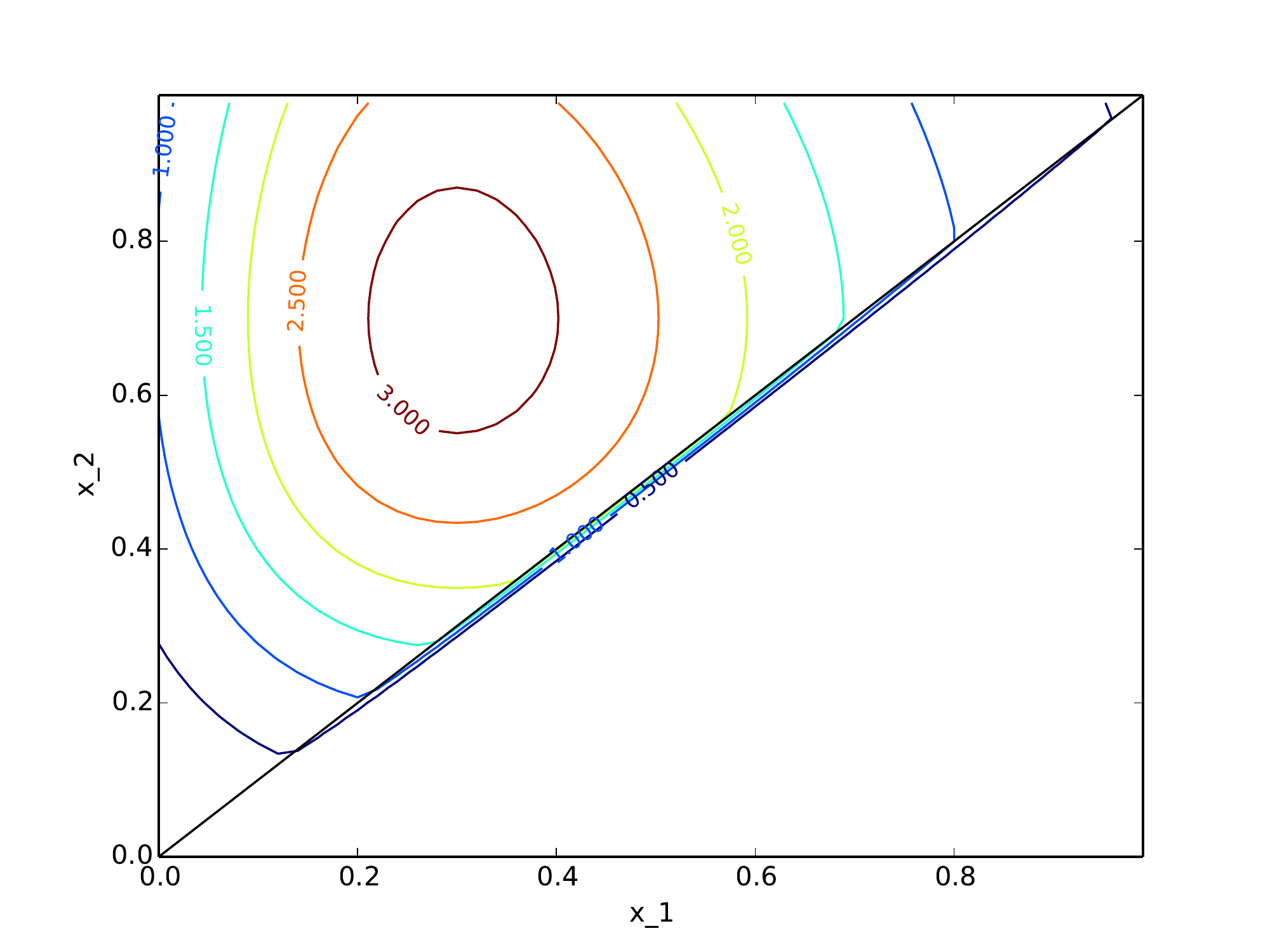}}
 \subfigure[AESE]{\includegraphics[width=4cm]{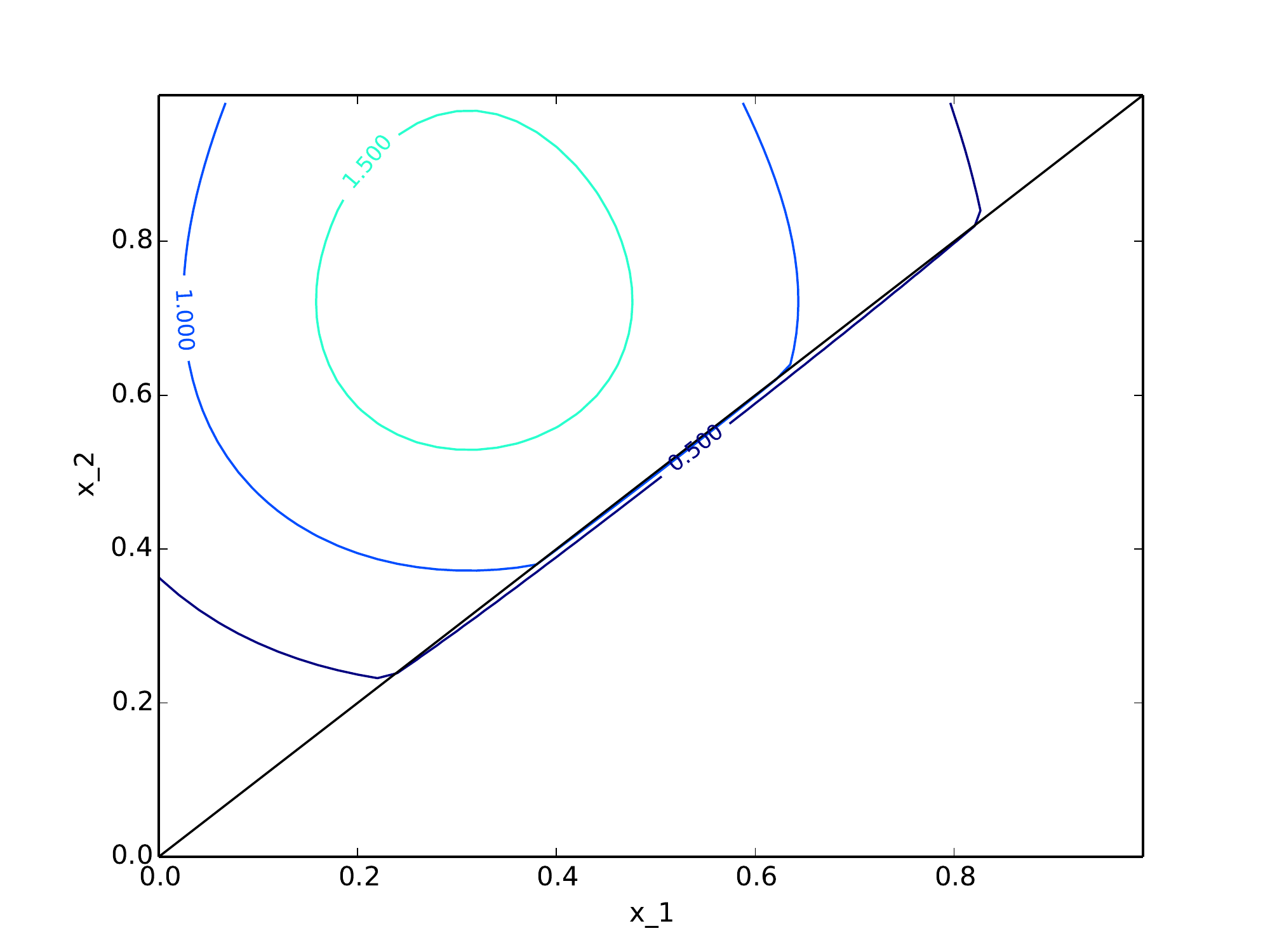}}
 \subfigure[Kernel]{\includegraphics[width=4cm]{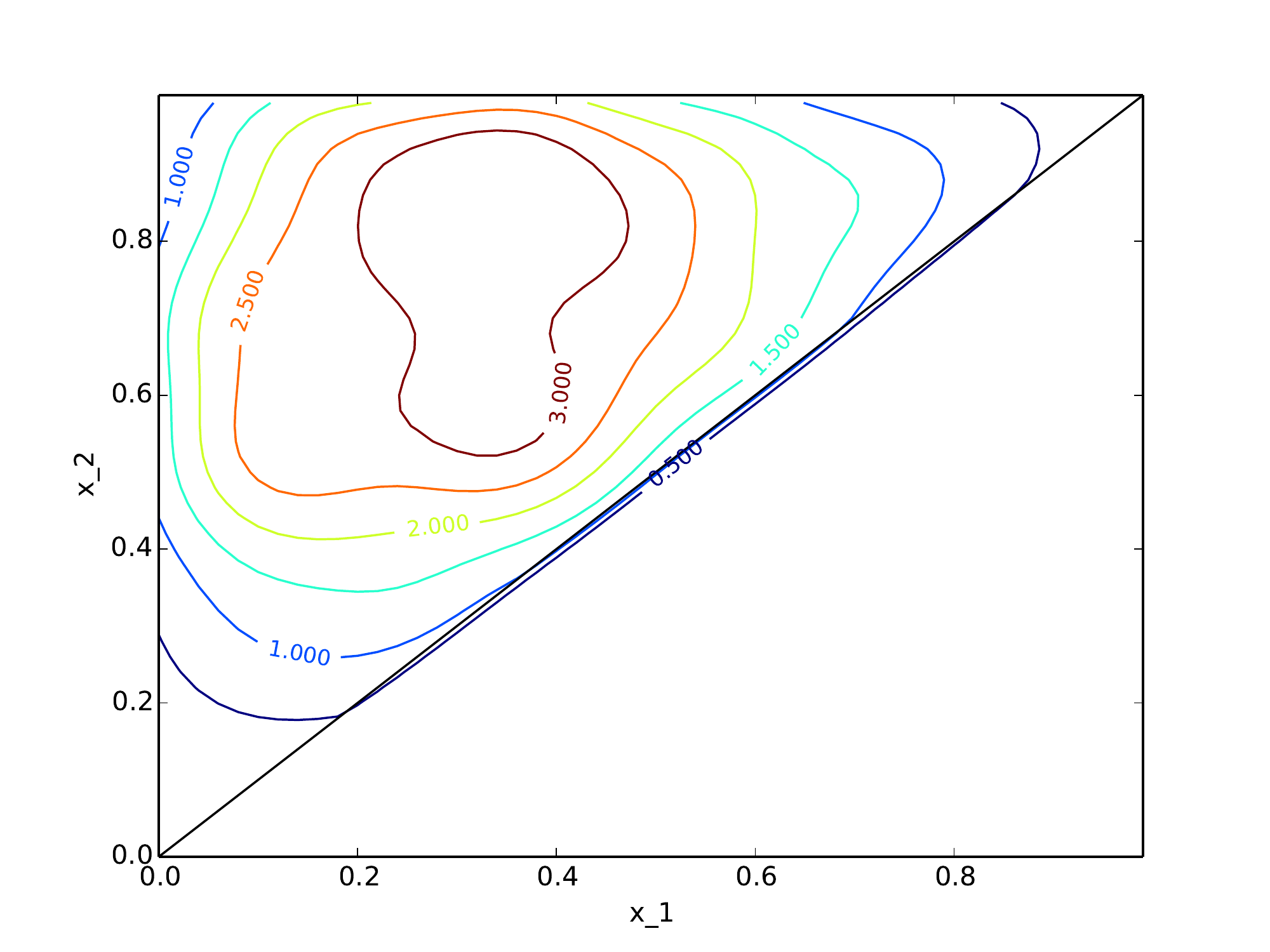}}
 \caption{Joint density functions of the true density and its estimators with Gumbel marginals.}
 \label{fig:kernel_ESE_gumbel_cont}
 \end{figure}

 %%%% Normal bimodal %%%%%

  \begin{figure}[H]
 \centering
 \subfigure{\includegraphics[width=3.5cm]{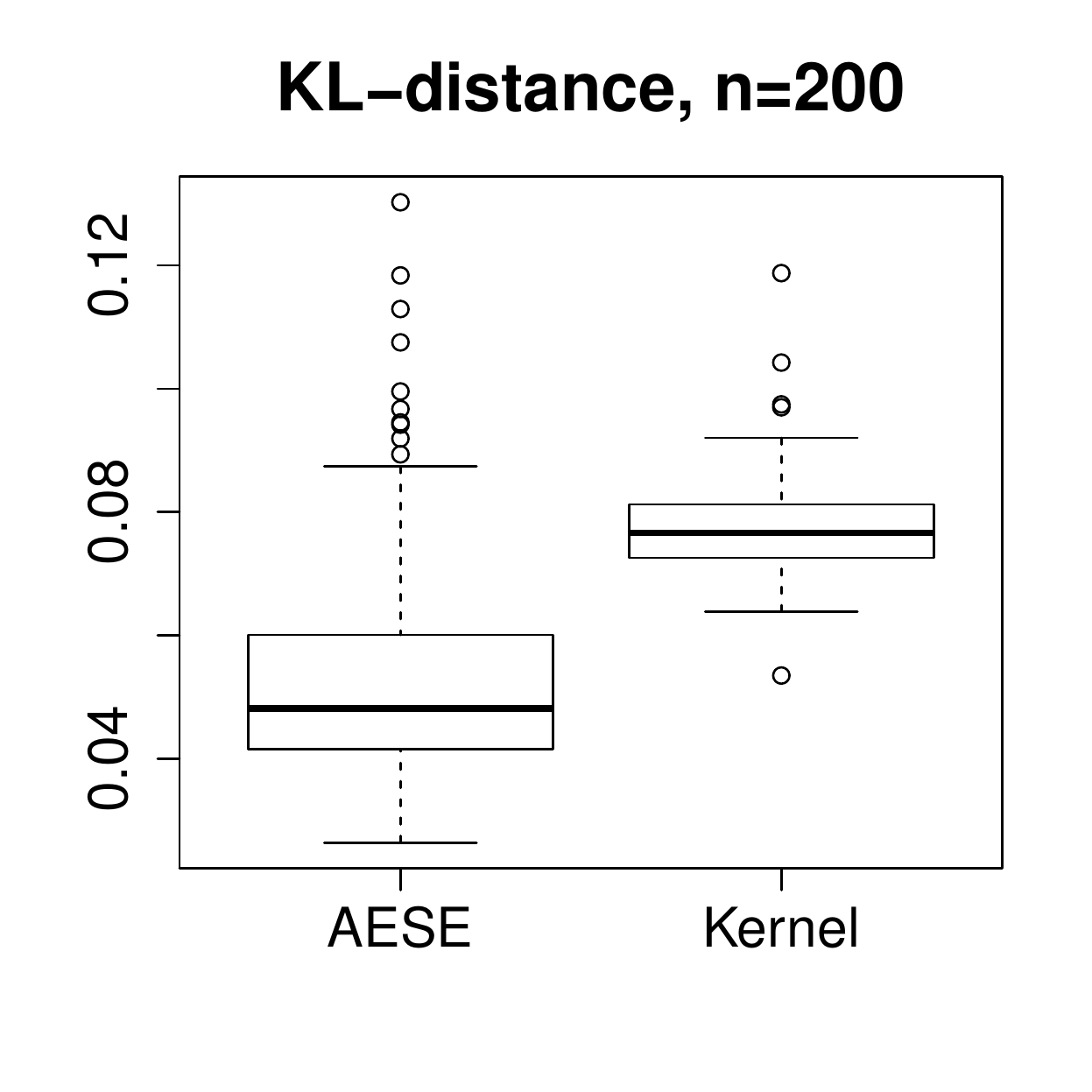}}
 \subfigure{\includegraphics[width=3.5cm]{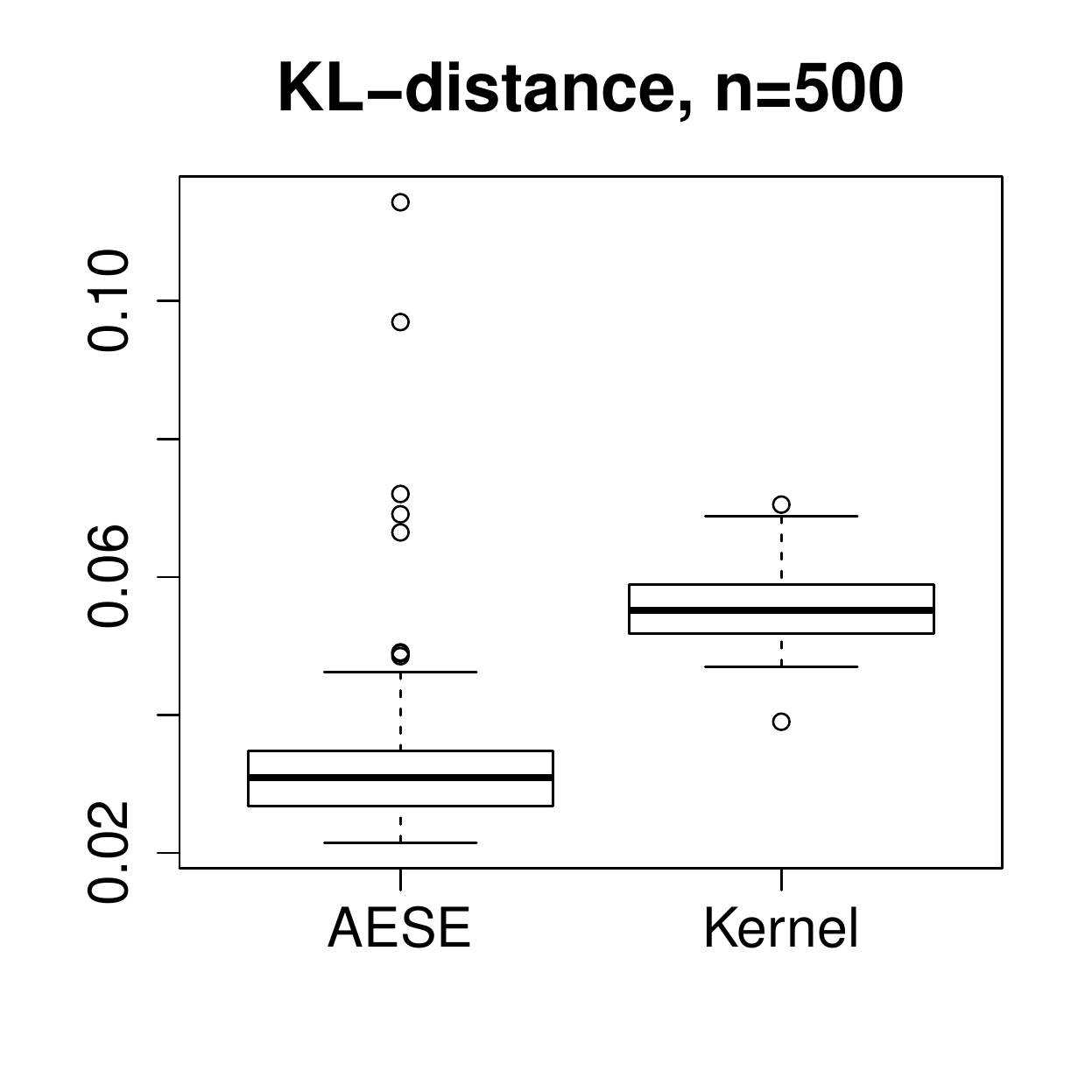}}
 \subfigure{\includegraphics[width=3.5cm]{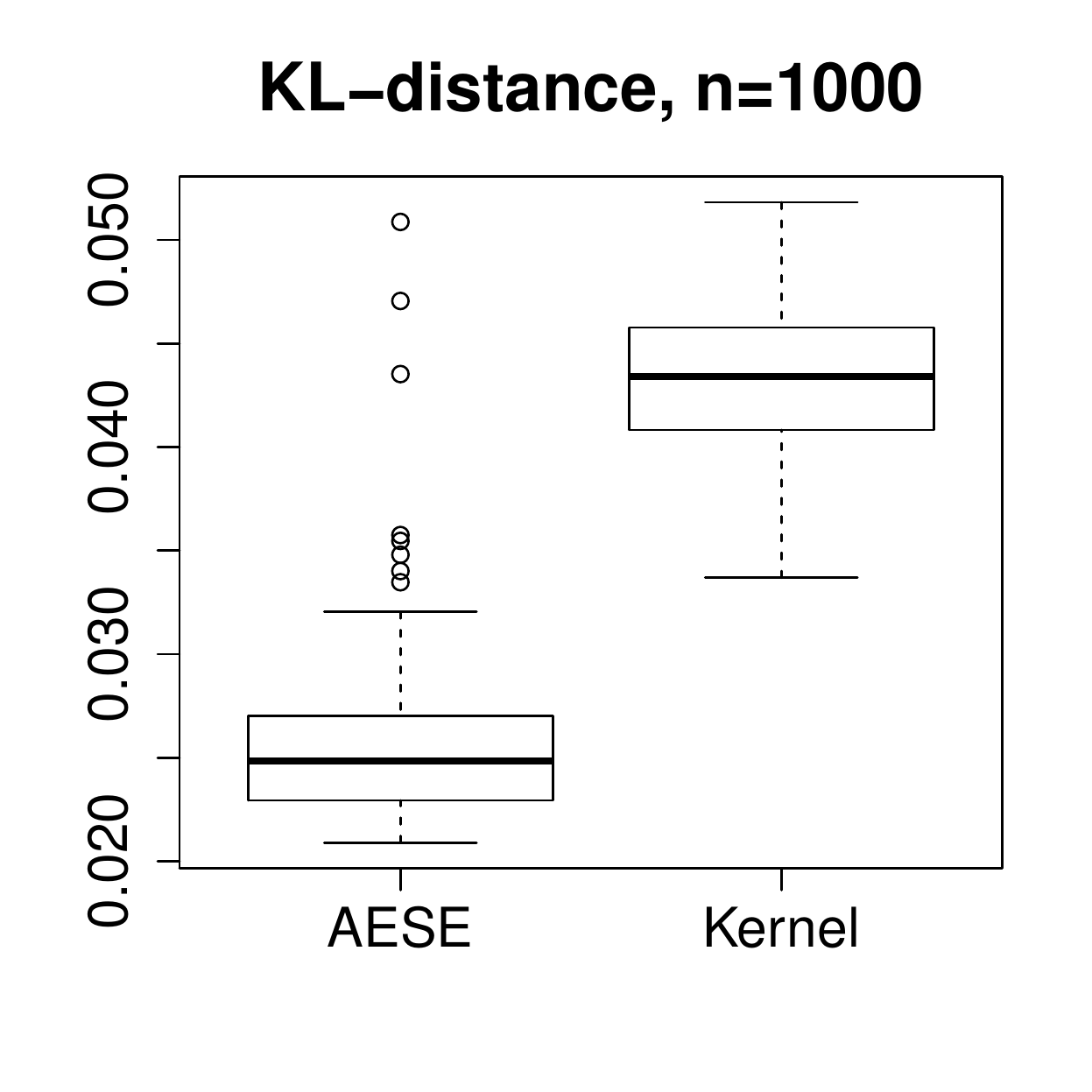}}
 
 \subfigure{\includegraphics[width=3.5cm]{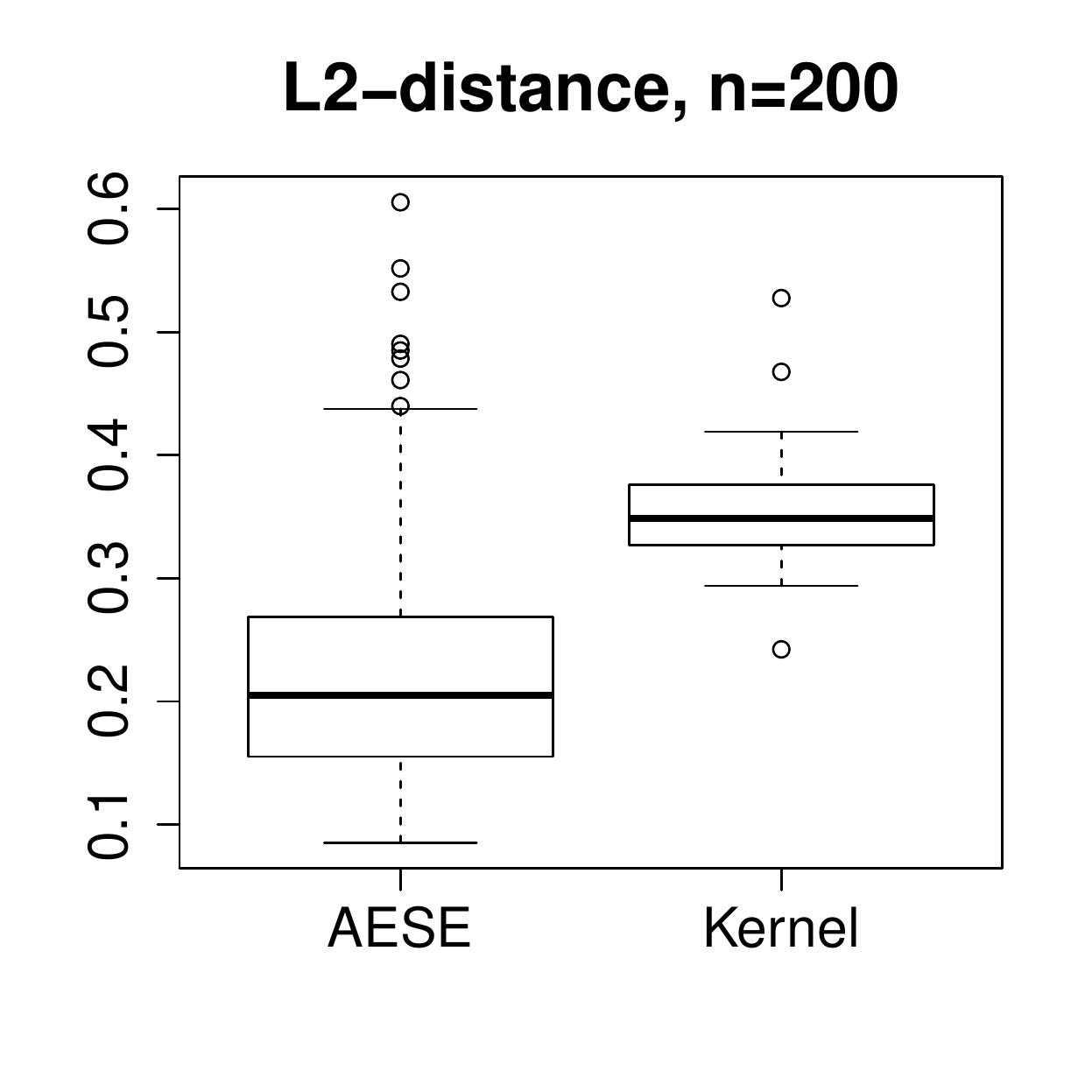}}
 \subfigure{\includegraphics[width=3.5cm]{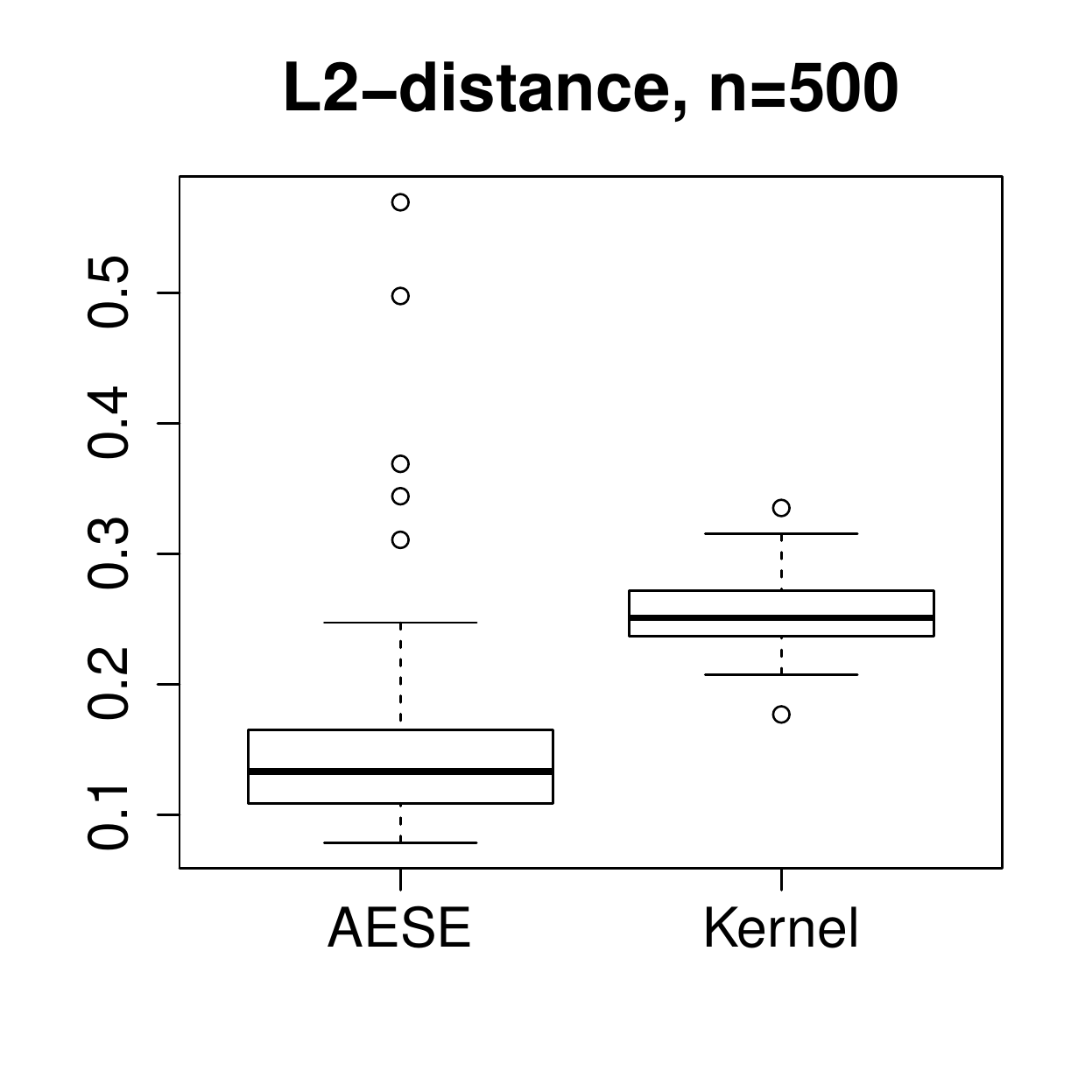}}
 \subfigure{\includegraphics[width=3.5cm]{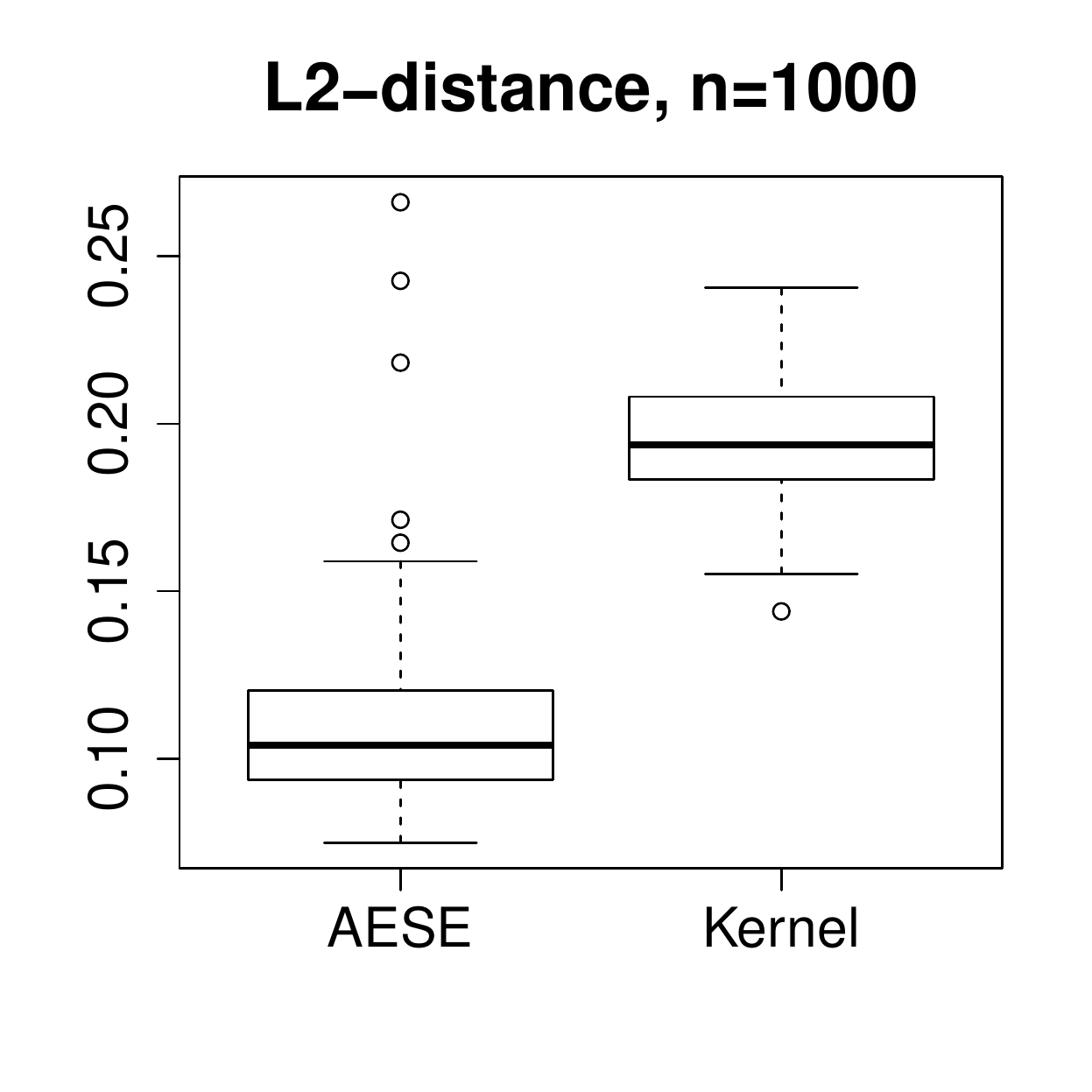}}
 
  \caption{Boxplot of the Kullback-Leibler and $L^2$ distances for the additive exponential series estimator (AESE) and the truncated kernel estimators with Normal mix marginals.}
 \label{fig:boxplot_normal_bim_KL}
 \end{figure}

 \begin{figure}[H]
 \centering
 \subfigure[True density]{\includegraphics[width=4cm]{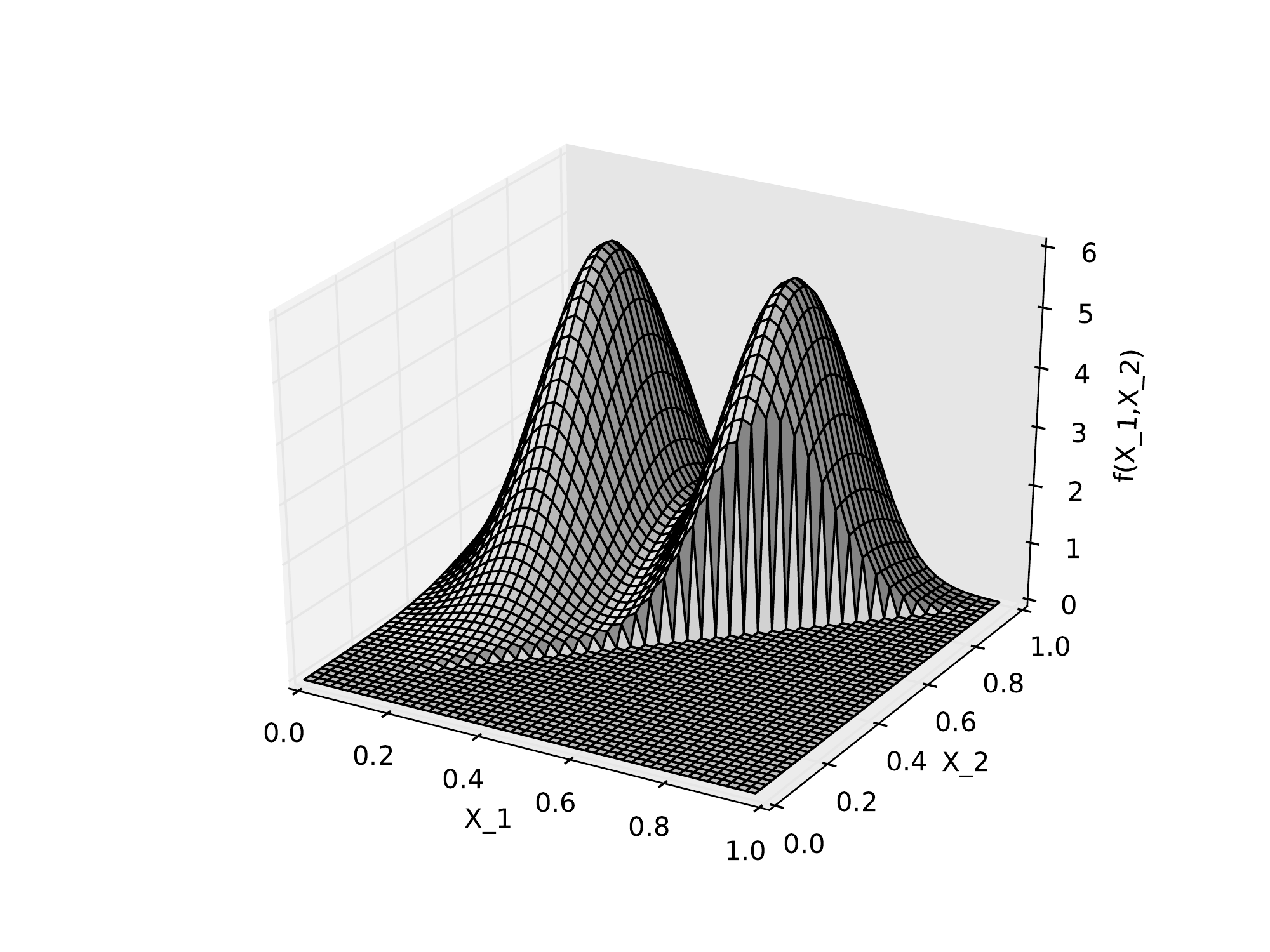}}
 \subfigure[AESE]{\includegraphics[width=4cm]{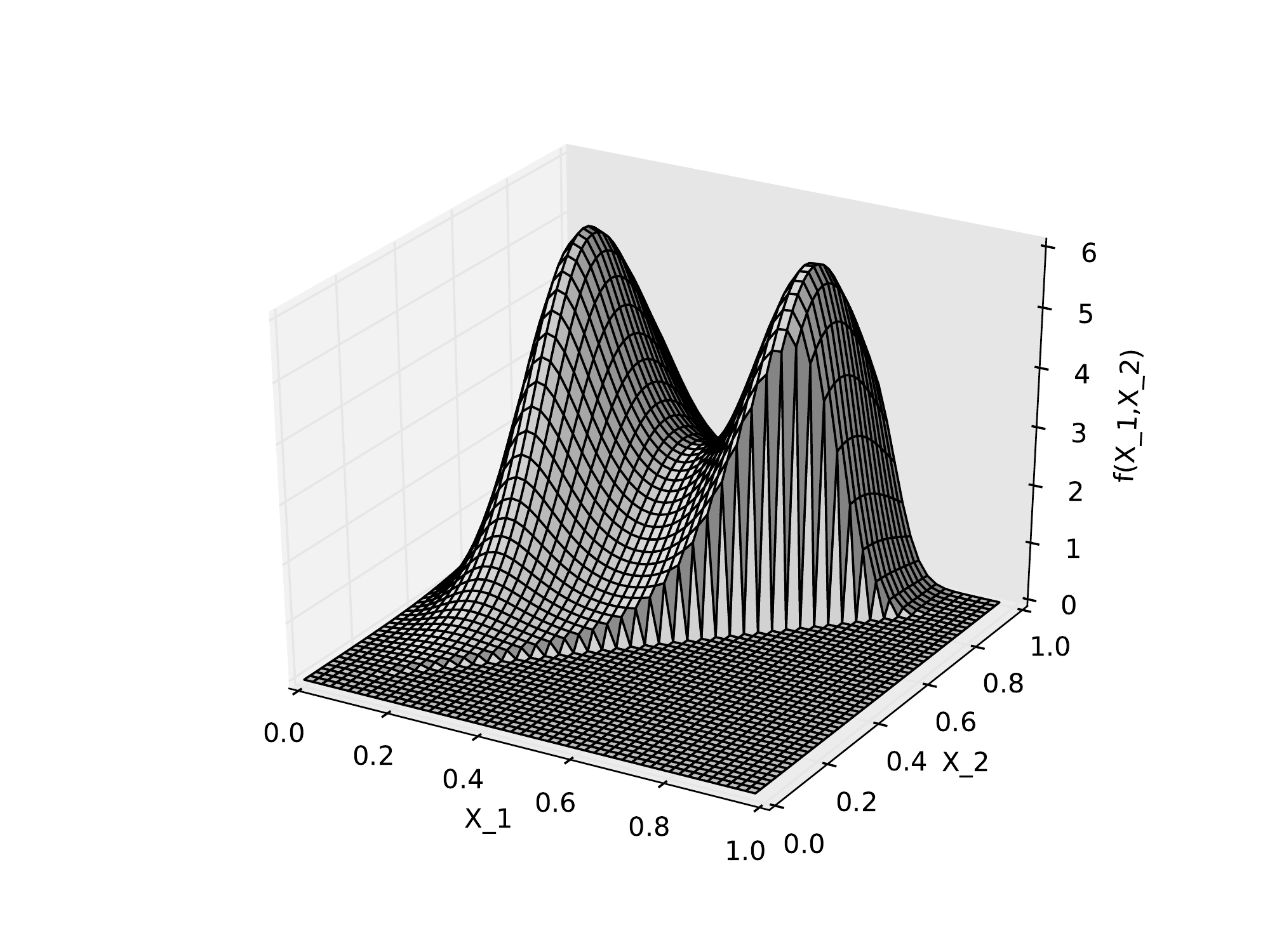}}
 \subfigure[Kernel]{\includegraphics[width=4cm]{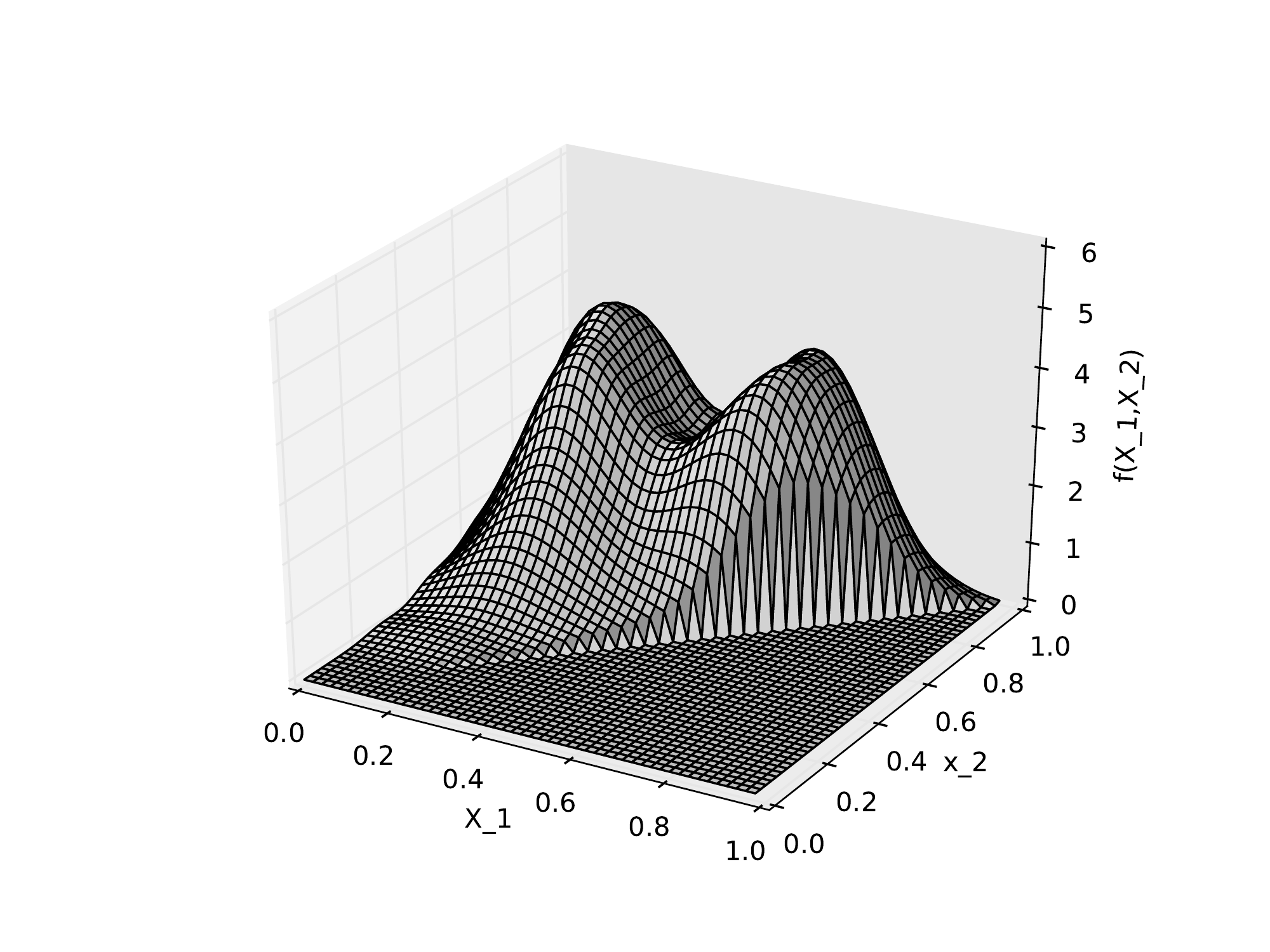}}

 \subfigure[True density]{\includegraphics[width=4cm]{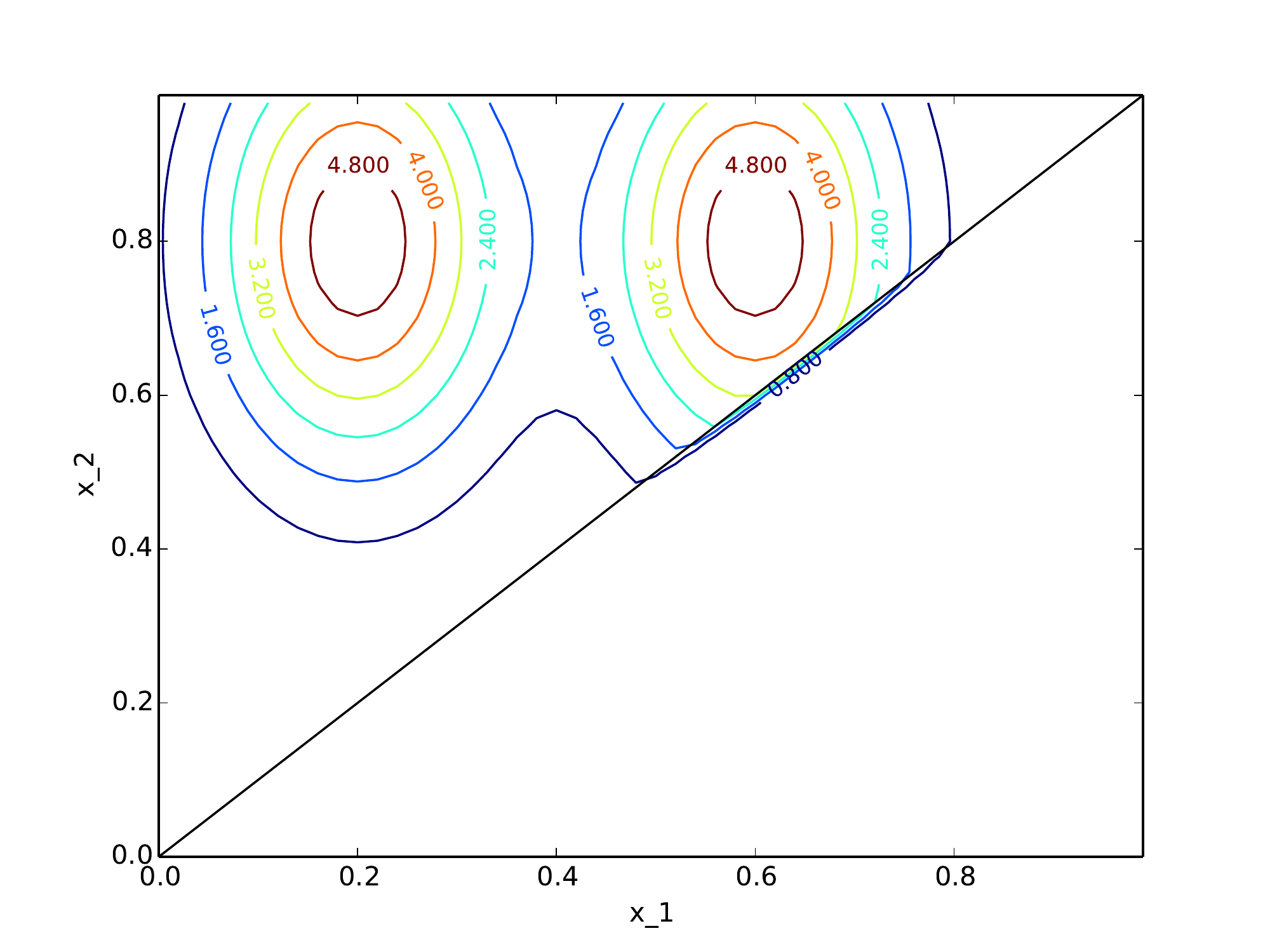}}
 \subfigure[AESE]{\includegraphics[width=4cm]{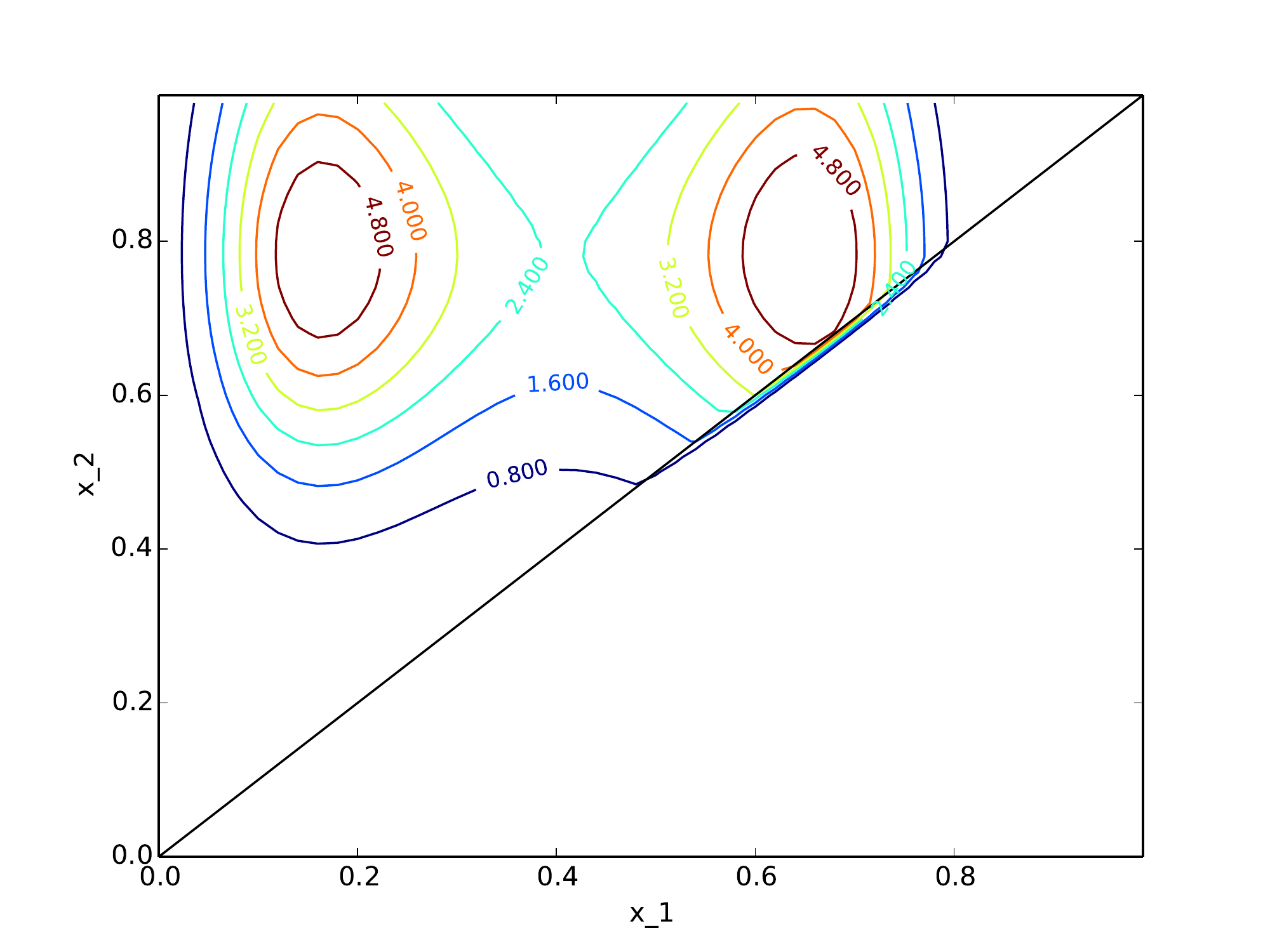}}
 \subfigure[Kernel]{\includegraphics[width=4cm]{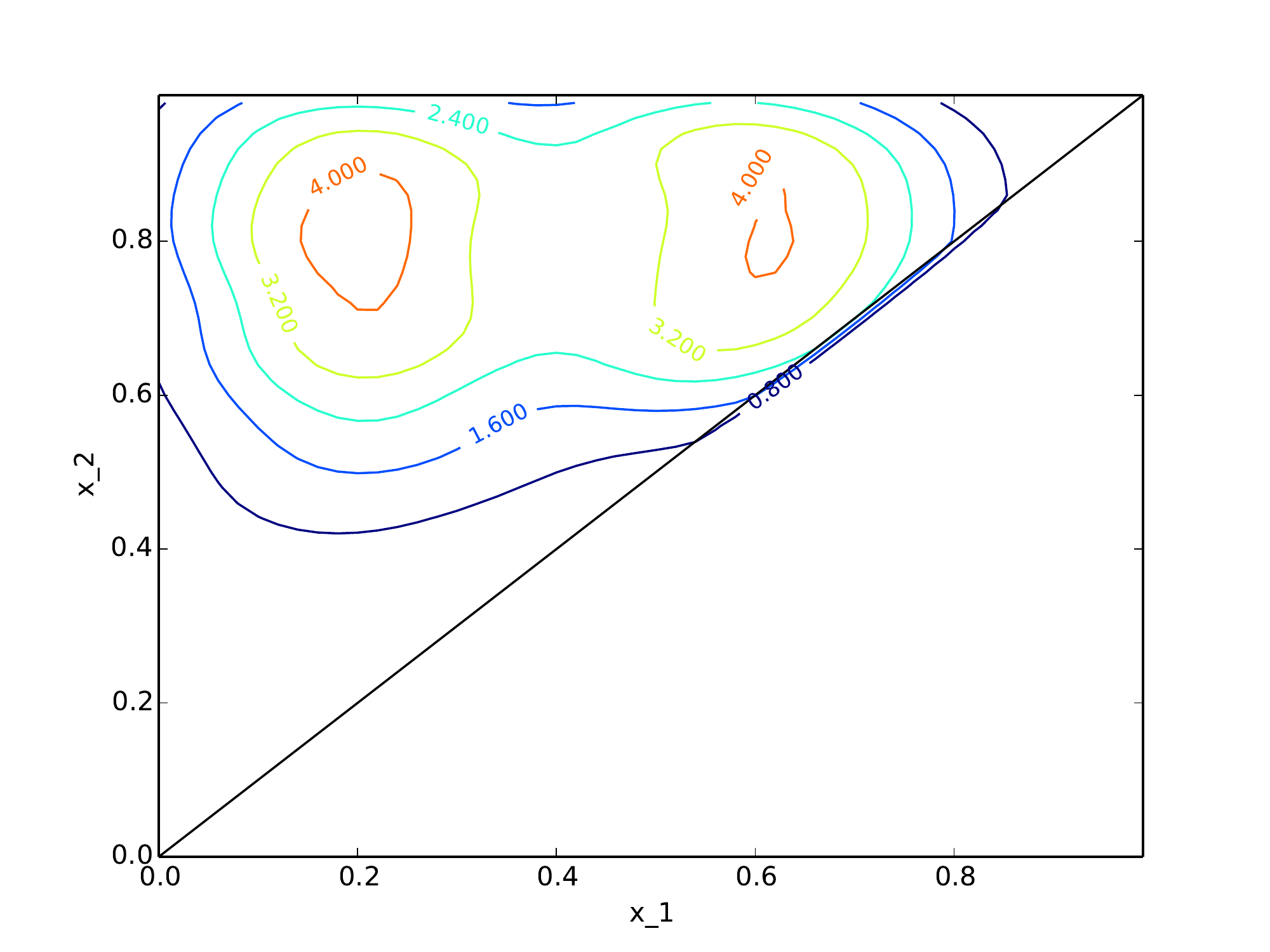}}
 \caption{Joint density functions of the true density and its estimators with Normal mix marginals.}
 \label{fig:kernel_ESE_normal_bim_cont}
 \end{figure}
 
 We can conclude that the additive exponential series estimator outperforms the kernel density estimator both with respect to the Kullback-Leibler distance and the $L^2$ distance. As expected, the performance of both methods increases with the sample size. The boxplot of the $100$ values of the Kullback-Leibler and $L^2$ distance for the different sample sizes can be found in Figures \ref{fig:boxplot_beta_KL}, \ref{fig:boxplot_gumbel_KL} and \ref{fig:boxplot_normal_bim_KL}. Figures \ref{fig:kernel_ESE_beta_cont}, \ref{fig:kernel_ESE_gumbel_cont} and \ref{fig:kernel_ESE_normal_bim_cont} illustrate the different estimators compared to the true joint density function for the three cases obtained with a sample size of $1000$. We can observe that the additive exponential series method leads to a smooth estimator compared to the kernel method. 
 
 \begin{rem}
  The additive exponential series model encompasses a lot of popular choices for the marginals $p_1, p_2$. For example, the exponential distribution is included in the model for $m_i=1$, and the normal distribution is included for $m_i=2$. Thus we expect that if we choose exponential or normal distributions for $p_1, p_2$, we obtain even better results for the additive exponential series estimator, which was confirmed by the numerical experiments (not included here for brevity).   
\end{rem}

 \section{Appendix: Orthonormal series of polynomials}  \label{sec:orth_poly} 

\subsection{Jacobi polynomials}
The following results can be found in  \cite{abramowitz1970handbook}
p. 774. 
   The Jacobi polynomials $(P^{(\alpha,\beta)}_k, k \in \N )$ for
   $\alpha,\beta \in (-1, +\infty )$ are  series of orthogonal
   polynomials  with respect to the measure $w_{\alpha, \beta}(t)
   \ind_{[-1, 1]}(t) \, dt$, with $w_{\alpha,\beta}(t)=(1-t)^{\alpha} (1+t)^{\beta}$ for $t\in [-1, 1]$.
They are
given by Rodrigues' formula, for $t\in [-1, 1]$, $k\in \N$:
\[
     P^{(\alpha,\beta)}_k(t)=\frac{(-1)^k}{2^k k! w_{\alpha,\beta}(t)}
      \frac{d^k}{dt^k}\left[ w_{\alpha,\beta}(t)(1-t^2)^k\right]. 
\]
The normalizing constants  are given by:
\begin{equation} \label{eq:Jac_norm_10}
   \int_{-1}^1  P^{(\alpha,\beta)}_k(t)P^{(\alpha,\beta)}_\ell(t) w_{\alpha,\beta}(t)\,
   dt = \ind_{\{k=\ell\}} \frac{2^{\alpha+\beta+1}}{2k+\alpha+\beta+1} \frac{\Gamma(k+\alpha+1)\Gamma(k+\beta+1)}{\Gamma(k+\alpha+\beta+1)k!} \cdot
\end{equation}

In what follows, we will be interested in Jacobi polynomials with $\alpha=d-i$ and $\beta=i-1$, which are orthogonal to the weight function $w_{d-i,i-1}(t)=\ind_{[-1, 1]}(t)(1-t)^{d-i} (1+t)^{i-1}$. The leading coefficient of $P^{(d-i,i-1)}_k$ is:
\begin{equation}
   \label{eq:omega}
\omega'_{i,k}=\frac{(2k+d-1)!}{2^k k!(k+d-1)!}\cdot
\end{equation}
Let $r\in \N^*$. Recall that $P^{(\alpha,\beta)}_k$ has degree $k$. The derivatives of the Jacobi polynomials $P^{(d-i,i-1)}_k$, $r \leq
k$,  verify, for $t\in I$ (see Proposition 1.4.15 of \cite{dunkl2001orthogonal}): 
\begin{equation} \label{eq:jac_deriv}
        \frac{d^{r}}{dt^{r}}P^{(d-i,i-1)}_k(t) =
        \frac{(k+d-1+r)!}{2^{r}( k+d-1)!}
        P^{(d-i+r,i-1+r)}_{k-r}(t).
\end{equation}
 We also have: 
      \begin{equation} \label{eq:jac_sup}
        \sup_{t\in [-1,1]} \val{P^{(d-i,i-1)}_k(t)}= \max\left(\frac{(k+d-i)!}{k!(d-i)!},\frac{(k+i-1)!}{k!(i-1)!} \right).
      \end{equation}

\subsection{Definition of the basis functions} 

Based on the Jacobi polynomials, we define a shifted version, normalized
and adapted to the interval $I=[0,1]$.

\begin{defi} \label{def:phi}
  For $1 \leq i \leq d$, $k \in \N$, we define for $t\in I$:
\[
     \phi_{i,k}(t) = \rho_{i,k}\sqrt{(d-i)!(i-1)!} \, P^{(d-i,i-1)}_k(2t-1),
\]
 with
\begin{equation}
   \label{eq:def-rik}
\rho_{i,k}=\sqrt{(2k+d)k!(k+d-1)!/((k+d-i)!(k+i-1)!)}.
\end{equation}
\end{defi}
Recall the definition \reff{eq:def_qi} of  the marginals $q_i$ of the Lebesgue
measure  on  the   simplex.  According  to  the   following  Lemma,  the
polynomials  $(\phi_{i,k},   k\in  \N)$  form  an   orthonormal  basis  of
$L^2(q_i)$  for all  $1\leq i\leq  d$. Notice  that $\varphi_{i,k}$  has
degree $k$.

\begin{lem} \label{prop:phi_P_equiv}
  For $1 \leq i \leq d$, $k,\ell \in \N$, we have:
\[
     \int_I \phi_{i,k} \phi_{i,\ell}\, q_i = \ind_{\{k=\ell\}}.
\]
\end{lem}

\begin{proof} We have, for $k,\ell \in \N$:
   \begin{align*}
    \int_I \phi_{i,k} \phi_{i,\ell}\, q_i &
    =\rho_{i,k}\rho_{i,\ell} \int_0^1  P^{(d-i,i-1)}_k(2t-1)
    P^{(d-i,i-1)}_\ell(2t-1) (1-t)^{d-i} t^{i-1} \, dt \\ 
                    & = \frac{\rho_{i,k}\rho_{i,\ell}}{2^d}  \int_{-1}^1
                    P^{(d-i,i-1)}_k(s) P^{(d-i,i-1)}_\ell(s)
                    w_{d-i,i-1}(s)\, ds\\ 
                    & =  \ind_{\{k=\ell\}},
   \end{align*}
where we used \reff{eq:Jac_norm_10} for the last equality. 
\end{proof}

\subsection{Mixed scalar products}

Recall notation \reff{eq:def-hi}, so that 
$\varphi_{[i],k}(x)=\varphi_{i,k}(x_i)$ for $x=(x_1, \ldots, x_d)\in
\triangle$. Notice that $( \phi_{[i],k}, k\in
\N)$ is a family of orthonormal polynomials  with
respect to the Lebesgue measure on $\triangle$, for all $1\leq i\leq d$.

We give the  mixed scalar products of $(\phi_{[i],k},k \in \N)$ and
 $(\phi_{[j],\ell},\ell \in \N)$, $1\leq i<j \leq d$ with respect to the
 Lebesgue measure on the simplex $\triangle$.

 \begin{lem} \label{lem:phi_ki_phi_lj}
For  $1\leq i<j \leq d$ and  $k,\ell \in \N$, we have:
\[
     \int_\triangle \phi_{[i],k}\, \phi_{[j],\ell} = \ind_{\{k=\ell\}}
     \sqrt{\frac{(j-1)!(d-i)!}{(i-1)!(d-j)!}}
     \sqrt{\frac{(k+d-j)!(k+i-1)!}{(k+d-i)!(k+j-1)!}}\cdot
\]
We also have $0 \leq \int_\triangle \phi_{[i],k}\, \phi_{[j],\ell} \leq 1$
for all $k,\ell \in \N$. 
 \end{lem}
 
 \begin{proof}
We have:
    \begin{align*}
    \int_\triangle \phi_{[i],k}\, \phi_{[j],\ell} 
& = \int_0^1 \left(\int_0^{x_j} \frac{x_i^{i-1}}{(i-1)!}
  \frac{(x_j-x_i)^{j-i-1}}{(j-i-1)!} \phi_{i,k}(x_i) \, dx_i\right)
\phi_{j,\ell}(x_j) \frac{(1-x_j)^{d-j}}{(d-j)!}  \, dx_j \\ 
         & = \int_I r_k \phi_{j,\ell}\, q_j,
    \end{align*}
    with $r_k$ a polynomial defined on $I$ given by:
\[
r_k(s)  =  (j-1)! \int_0^{1} \frac{t^{i-1}}{(i-1)!}
\frac{(1-t)^{j-i-1}}{(j-i-1)!} \phi_{i,k}(s t) \, dt. 
\]
Notice that $r_k$ is a polynomial of degree at most $k$ as $\phi_{i,k}$ is a polynomial with degree $k$. Therefore if $k<\ell$ ,
we   have   $\int_\triangle    \phi_{[i],k}\phi_{[j],\ell}   =0$   since
$\phi_{j,\ell}$ is orthogonal (with respect to the measure $q_j$) to any
polynomial of degree less than $\ell$. Similar calculations show that if
$k>\ell$, the integral is also $0$.
    
    Let us consider now the case $k=\ell$. We compute  the 
    coefficient $\nu_{k}$ of $t^k$ in the polynomial $r_k$. We deduce from
    \reff{eq:omega} that the leading coefficient $\omega_{i,k}$ of $\phi_{i,k}$ is given by:
\[
       \omega_{i,k}= \rho_{i,k} \sqrt{(d-i)!(i-1)!} \omega'_{i,k}=
       \rho_{i,k} \sqrt{(d-i)!(i-1)!} \frac{(2k+d-1)!}{k!(k+d-1)!}\cdot
\]
    Using this we obtain for $\nu_{k}$ :
\begin{align*}
       \nu_{k} & = (j-1)!  \omega_{i,k} \int_0^{1}
       \frac{t^{k+i-1}}{(i-1)!} \frac{(1-t)^{j-i-1}}{(j-i-1)!} \, dt \\ 
                   & =  \omega_{i,k} \frac{(k+i-1)!(j-1)!
                   }{(k+j-1)!(i-1)!},
\end{align*}
and thus $r_k$ has degree $k$.
    The orthonormality of $(\phi_{j,k}, k \in \N)$ ensures that  $\int_I
    r_k \phi_{j,k}\,  q_j=\nu_{k} 
    /\omega_{j,k}$. Therefore, we obtain:
\[
\int_\triangle \varphi_{[i], k} \varphi_{[j],k}=
\frac{\nu_{k}
}{\omega_{j,k}}=\sqrt{\frac{(j-1)!(d-i)!}{(i-1)!(d-j)!}}
\sqrt{\frac{(k+d-j)!(k+i-1)!}{(k+d-i)!(k+j-1)!}} \cdot
\]
Since  $(j-1)!/(i-1)! \leq  (k+j-1)!/(k+i-1)!$, and  $(d-i)!/(d-j)! \leq
(k+d-i)!/(k+d-j)!$, we can conclude that
$
     0 \leq    \int_\triangle \phi_{[i],k}\phi_{[j],k}\leq  1. 
$

 \end{proof}

 This shows that the family of functions $\phi=(\phi_{i,k}, 1\leq i \leq
 d,  k \in \N)$ is not orthogonal with respect to the Lebesgue measure
 on $\triangle$. For $k
 \in \N^*$,  let us  consider the  matrix $R_k \in  \R^{d \times  d}$ with
 elements:
 \begin{equation} \label{eq:def_Rk}
 R_k(i,j)=  \int_\triangle \phi_{[i],k}\phi_{[j],k}. 
 \end{equation}
 If $Y=(Y_{1}, \hdots, Y_{d})$ is uniformly distributed on $\triangle$, then  $R_k$ is the correlation matrix of the random variable $(\phi_{1,k}(Y_{1}), \hdots,\phi_{d,k}(Y_{d})) $. Therefore it is symmetric and positive semi-definite.
   Let $\lambda_{k,1} \leq \hdots \leq \lambda_{k,d}$ denote the eigenvalues of $R_k$. We aim to find a lower bound for these eigenvalues which is independent of $k$. 
  
\begin{lem} 
\label{lem:min_eigen_Rk}  
For $k  \in \N^*$, the  smallest eigenvalue $\lambda_{k,d}$ of  $R_k$ is
given by:
\[
   \lambda_{k,d} = \frac{k}{k+d-1},
\]
and we have $ \lambda_{k,d} \geq 1/d$.
\end{lem}

\begin{proof}
It is easy to
  check that the  inverse $R_{k}^{-1}$ of $R_k$ exists  and is symmetric
  tridiagonal with diagonal  entries $D_i$, $1 \leq i \leq  d$ and lower
  (and upper) diagonal elements $Q_i$, $1 \leq i \leq d-1$ given by:
\[
     D_i=\frac{(k+d-1)(k+1)+2(i-1)(d-i)}{k(k+d)} \quad \text{and}\quad 
    Q_{i}=-\frac{\sqrt{i(d-i)(k+i)(k+d-i)}}{k(k+d)}\cdot
\]
The matrix $R_{k}^{-1}$ is positive definite, since all of its principal minors 
have a positive determinant.
In particular, this ensures that the eigenvalues of $R_k$ and $R^{-1}_{k}$
are all positive.
Let $c_i(\lambda)$, $1\leq i \leq d$ denote the $i$-th leading principal
minor  of  the  matrix  $R^{-1}_k-\lambda   I_d$,  where  $I_d$  is  the
$d$-dimensional  identity matrix.   The  eigenvalues  of $R^{-1}_k$  are
exactly the roots of the characteristic polynomial $c_d(\lambda)$. Since
$R^{-1}_k$  is   symmetric  and  tridiagonal,  we   have  the  following
recurrence relation for $c_i(\lambda)$, $1\leq i \leq d$:
   \[
      c_{i}(\lambda)=(D_i-\lambda) c_{i-1}(\lambda)- Q^2_{i-1} c_{i-2}(\lambda),
   \]
   with initial values  $c_0(\lambda)=1$, $c_{-1}(\lambda)=0$. 

Let $M_k$ be the symmetric tridiagonal matrix $d\times d$ with diagonal
entries  $D_i$, $1  \leq i  \leq d$  and lower  (and upper)  diagonal
   elements $|Q_i|$, $1 \leq i \leq d-1$. Notice the characteristic
   polynomial of $M_k$ is also $c_d(\lambda)$. So $M_k$ and $R_k^{-1}$
   have the same eigenvalues. 

It is easy to check that $\lambda^*=(k+d-1)/k$ is an eigenvalue of $M_k$
with corresponding eigenvector $v=(v_1, \ldots, v_d)$ given by, for
$1\leq i\leq d$: 
\[
v_i=\sqrt{\frac{(d-1)!}{(d-i)!}\frac{(k+d-1)!}{(k+d-i)!}
  \frac{k!}{(k+i-1)!} \frac{1 }{(i-1)!}}\cdot
\]
(One can check that $v'=(v'_1, \ldots, v'_d)$, with $v'_i=(-1)^{i-1} v_i$, is
an eigenvector of $R_k^{-1}$ with eigenvalue $\lambda^*$.) 

 The matrix $M_k$ has non-negative elements, with positive elements in
 the diagonal, sub- and superdiagonal. Therefore $M_k$ is irreducible,
 and we can apply the Perron-Frobenius theorem for non-negative,
 irreducible matrices: the largest eigenvalue of $M_k$ has multiplicity
 one and is the only eigenvalue with corresponding eigenvector $x$ such
 that $x>0$. Since $v>0$, we deduce that $\lambda^*$ is the largest
 eigenvalue of $M_k$. It is also the largest eigenvalue of
 $R_k^{-1}$. Thus $1/\lambda^*=k/(k+d-1)$ is the lowest eigenvalue of
 $R_k$. 

Since $\lambda_{k,d}$ is increasing in $k$, we have the uniform lower bound $1/d$. 
  \end{proof}

  \begin{rem}
   We conjecture that the eigenvalues $\lambda_{k,i}$ of $R_k$ are given by, for $1 \leq i \leq d$:
 \[
    \lambda_{k,i}=\frac{k(k+d)}{(k+i)(k+i-1)}\cdot
 \]
\end{rem}

  \subsection{Bounds between different norms} \label{sec:bounds}
  
  In this Section, we will  give inequalities between different types of
  norms  for  functions  defined   on  the  simplex  $\triangle$.  These
  inequalities    are    used    during    the    proof    of    Theorem
  \ref{theo:main_stat}. Let $m=(m_1, \hdots, m_d) \in (\N^*)^d$. Recall  the notation $\phi_m$ and  $\theta \cdot
  \phi_m$    with    $\theta  = (\theta_{i,k};1\leq k  \leq m_i, 1\leq
  i\leq d)    \in    \R^{\val{m}}$    from    Section 
  \ref{sec:estim_stat}.

  For $1\leq  i\leq d$, we set $\theta_i =  (\theta_{i,k},1\leq k  \leq
  m_i)  \in \R^{m_i}$,  
  $\phi_{i,m}= (\phi_{i,k},1\leq k  \leq m_i)$ and:
\[
\theta_i\cdot    \phi_{i,m}=\sum_{k=1}^{m_i}     \theta_{i,k}
  \varphi_{i,k}\quad\text{and}\quad \theta_i\cdot
  \phi_{[i],m}=\sum_{k=1}^{m_i}  \theta_{i,k} 
  \varphi_{[i],k},
\]
with $\phi_{[i],m}= (\phi_{[i],k},1\leq  k  \leq m_i)$. In particular,
we have $\varphi_m=\sum_{i=1}^d \phi_{[i],m}$ and $\theta \cdot
  \phi_m=\sum_{i=1}^d\theta_i\cdot
  \phi_{[i],m}$. We first give lower and upper bounds on $\norm{\theta \cdot \phi_m}_{L^2}$.

  \begin{lem}\label{lem:thetaphi_theta}
   For all $\theta \in \R^{\val{m}}$ we have:
   \[
   \frac{\norm{\theta} }{ \sqrt{d} }   \leq \norm{\theta\cdot \phi_m}_{L^2}
  \leq \sqrt{d} \norm{\theta}.
   \]
  \end{lem}
  
  \begin{proof}
    For the upper bound, one simply has, by the triangle inequality and
    the orthonormality:
    \[
      \norm{\theta\cdot \phi_m}_{L^2} \leq \sum_{i=1}^d
      \norm{\theta_i \cdot \phi_{i,m}}_{L^2(q_i)} = \sum_{i=1}^d
      \norm{\theta_i} \leq \sqrt{d} \norm{\theta}. 
    \]
   For the lower bound, we have:
   \begin{equation} \label{eq:theta_phi_min}
   \norm{\theta \cdot \phi_m }_{L^2}^2 
 = \sum_{i=1}^d \sum_{k=1}^{m_i} \theta_{i,k}^2  + 2\sum_{i<j}
         \sum_{k=1}^{\min(m_i,m_j)} \theta_{i,k} \theta_{j,k}
         \int_\triangle \phi_{[i],k}\phi_{[j],k},  
   \end{equation}
where we used the normality of $\varphi_{[i],k}$ with respect to the
Lebesgue measure on $\triangle$ and Lemma \ref{lem:phi_ki_phi_lj} for
the cross products. 
    We can rewrite  this  in a matrix form:
    \[
      \norm{\theta \cdot \phi _m}_{L^2}^2  \geq 
      \sum_{k=1}^{\max (m) } (\theta^*_k )^T R_k \theta^*_k,
    \]
    where $R_k \in \R^{d\times d}$ is given by \reff{eq:def_Rk} and
    $\theta^*_k = (\theta^*_{1,k}, \hdots, \theta^*_{d,k}) \in \R^d$ is defined, for $1 \leq i \leq d$, 
    $1\leq k \leq \max (m)$, as: 
    \[
        \theta^*_{i,k}=\theta_{i,k}\ind_{\{k \leq m_i\}}. 
    \]
    Since,   according   to   Lemma  \ref{lem:min_eigen_Rk},   all   the
    eigenvalues of $R_k$ are uniformly larger than $1/d$, this gives:
    \[
      \norm{\theta\cdot \phi_m }_{L^2}^2  \geq \inv{d}\sum_{k=1}^{\max (m)}
      \norm{\theta^*_k}^2=
      \frac{\norm{\theta}^2}{d}\cdot       
    \]
    This concludes the proof. 
  \end{proof}

We give an inequality between different norms for polynomials defined
on $I$. 
     \begin{lem} \label{lem:poly}
       If $h$ is a polynomial of degree less then or equal to $n$ on
       $I$, then we have for all $1\leq i \leq d$: 
       \[
          \norm{h}_\infty \leq \sqrt{2(d-1)!} (n+d)^{d}\norm{h}_{L^2(q_i)}
       \]
     \end{lem}
     \begin{proof}
      There exists $(\beta_k, 0\leq k\leq n)$ such that $h= \sum_{k=0}^n
      \beta_k \phi_{i,k}$.  By the Cauchy-Schwarz inequality, we have: 
      \begin{equation} \label{eq:CS_poly}    
         \val{h} \leq
\left(\sum_{k=0}^n \beta^2_k\right)^{1/2}
\left(\sum_{k=0}^n \phi^2_{i,k} \right)^{1/2}.
       \end{equation}

We deduce from Definition \ref{def:phi} of $\phi_{i,k}$ and
\reff{eq:jac_sup} that:
      \[
     \norm{\phi_{i,k}}_\infty = \sqrt{\frac{(2k+d)(k+d-1)!}{k!}}
     \max\left( \sqrt{\frac{(i-1)!(k+d-i)!}{(d-i)!(k+i-1)!}}, \sqrt{\frac{(d-i)!(k+i-1)!}{(i-1)!(k+d-i)!}}
       \right). 
      \]
      For all $1\leq i \leq d$, we have the uniform upper bound:
 \begin{equation} \label{eq:sup_phi_ik}
   \norm{\phi_{i,k}}_\infty \leq \sqrt{(d-1)!}\sqrt{2k+d} \frac{(k+d-1)!}{k!}\cdot
 \end{equation}
      This implies that for $t \in I$:
\[
       \sum_{k=0}^n \phi^2_{i,k}(t)  
 \leq \sum_{k=0}^n  \norm{\phi^2_{i,k}}_\infty    
\leq (d-1)!
\sum_{k=0}^n  (2k+d)  \left( \frac{(k+d-1)!}{k!} \right)^2
 \leq 2(d-1)! (n+d)^{2d}.
\]
Bessel's inequality implies that $\sum_{k=0}^n \beta^2_k
\leq \norm{h}_{L^2(q_i)} ^2$. We conclude the proof using
\reff{eq:CS_poly}. 
\end{proof}

We  recall   the  notation  $S_m$   of  the  linear  space   spanned  by
$(\phi_{[i],k}; 1 \leq k \leq m_i, 1  \leq i \leq d)$, and the different
norms introduced in Section \ref{sec:prelim}.

\begin{lem} \label{lem:Am}      
Let $m\in (\N^*)^d$ and $\kappa_m=\sqrt{2d!}\sqrt{\sum_{i=1}^d
  (m_i+d)^{2d}}$. Then we have  for every $g \in S_{m}$: 
$
\norm{g}_{\infty} \leq \kappa_m \norm{g}_{L^2}.
$
     \end{lem}
     \begin{proof}
Let $g \in S_{m}$. We can write $g=\theta \cdot \varphi_m$ for a unique
$\theta\in \R^{\val{m}}$. Let $g_i=\theta_i\cdot \varphi_{i,m}$ so that  $g=\sum_{i=1}^d g_{[i]}$, where $g_i$ is a
polynomial defined on $I$ of degree at most $m_i$ for all $1\leq i\leq
d$. We have:
\begin{align*}
\norm{g}_{\infty} 
& \leq \sum_{i=1}^d \norm{g_{i}}_\infty  \\
& \leq \sqrt{2(d-1)!} \sum_{i=1}^d (m_i+d)^{d}
\norm{g_{i}}_{L^2(q_i)} \\ 
& \leq \frac{\kappa_m}{\sqrt{d}} 
\left(\sum_{i=1}^d \norm{g_{i}}^2_{L^2(q_i)}\right)^{1/2} \\ 
& =  \frac{\kappa_m}{\sqrt{d}} 
\norm{\theta}\\
& \leq   \kappa_m
\norm{\theta\cdot \varphi_m}_{L^2}\\
&= \kappa_m\norm{g}_{L^2}.
\end{align*}
where  we   used  Lemma   \ref{lem:poly}  for  the   second  inequality,
Cauchy-Schwarz  for   the  third   inequality,   and Lemma
\ref{lem:thetaphi_theta}    for     the    fourth     inequality.
\end{proof}

      \begin{rem} \label{cor:Am}
      For $d$ fixed, $\kappa_m$ as a function of $m$ verifies:
      \[
        \kappa_m= O \left( \sqrt{\sum_{i=1}^d m_i^{2d}}\right) = O(\val{m}^d).
      \]
     \end{rem}
     
\subsection{Bounds on approximations}

     Now we bound the $L^2$ and $L^\infty$ norm of the approximation
     error of additive functions where each component belongs to a
     Sobolev space.   Let $m=(m_1, \hdots, m_d) \in (\N^*)^d$, $r=(r_1, \hdots, r_d) \in (\N^*)^d$ such that $m_i +1 \geq r_i$ for all $1 \leq i \leq d$. Let $\ell = \sum_{i=1}^d \ell_{[i]}$ with $\ell_i \in W^2_{r_i}(q_i)$ and $\int_I \ell_i q_i =0$ for $1 \leq i \leq d$. Let $\ell_{i,m_i}$ be the orthogonal projection in $L^2(q_i)$ of $\ell_i$ on the span of $(\phi_{i,k}, 0 \leq k \leq m_i)$ given by $\ell_{i,m_i} =\sum_{k=1}^{m_i} \left( \int_I \ell_i \phi_{i,k} q_i \right) \phi_{i,k}$.  Then $\ell_m = \sum_{i=1}^d \ell_{[i],m_i}$ is the approximation of $\ell$ on $S_m$ given by \reff{eq:Sm}. We start by giving a bound on the $L^2(q_i)$ norm of the error when we approximate $\ell_i$ by $\ell_{i,m_i}$.

     \begin{lem} \label{lem:delta_mi}
     For each $1 \leq i \leq d$, $m_i +1 \geq r_i$ and $\ell_i \in W^2_{r_i}(q_i)$ , we have:
     \begin{equation}\label{eq:delta_mi_lem}
        \norm{\ell_i - \ell_{i,m_i}}_{L^2(q_i)}^2 \leq \frac{2^{-2r_i}(m_i+1-r_i)!(m_i+d)!}{(m_i+1)! (m_i+d+r_i)!}  \norm{\ell^{(r_i)}_{i}}^2_{L^2(q_i)}.
     \end{equation}

     \end{lem}
     
     \begin{proof}
Notice that \reff{eq:jac_deriv} implies that the series $(\phi^{(r_i)}_{i,k},k \geq r_i)$ is orthogonal on $I$ with respect to the weight function $v_i(t)=(1-t)^{d-i+r_i}t^{i-1+r_i}$, and the normalizing constants $\kappa_{i,k} \geq 0$ are given by:
      \begin{align} \label{eq:kappa_ik}
       \nonumber \kappa_{i,k}^2 & = \int_0^1 \left( \phi^{(r_i)}_{i,k}(t) \right)^2 v_i(t) \,dt \\
       \nonumber       & = \rho_{i,k}^2 (d-i)!(i-1)! \int_0^1 \left( \frac{d^{r_i}}{dt^{r_i}}P^{(d-i,i-1)}_k(2t-1) \right)^2 v_i(t) \,dt \\
       \nonumber       & = \rho_{i,k}^2 (d-i)!(i-1)! \frac{((k+d-1+r_i)!)^2}{2^{d+2r_i} ((k+d-1)!)^2} \int_{-1}^1 \left(P^{(d-i+r_i,i-1+r_i)}_{k-r_i}(s) \right)^2 w_{d-i+r_i,i-1+r_i}(s) \,ds \\
              & = (d-i)!(i-1)!\frac{k!(k+d-1+r_i)!}{(k-r_i)!(k+d-1)!},
      \end{align}      
      where we used the definition of $\phi_{i,k}$ for the second equality, \reff{eq:jac_deriv} for the third equality and \reff{eq:Jac_norm_10} for the fourth equality. Notice that $\kappa_{i,k}$ is non-decreasing as a function of $k$. 
      Since $\ell_i - \ell_{i,m_i} = \sum_{k=m_i+1}^{\infty} \beta_{i,k} \phi_{i,k}$, we have:
      \begin{equation} \label{eq:delta_mi}
        \norm{\ell_i - \ell_{i,m_i}}_{L^2(q_i)}^2  = \sum_{k=m_i+1}^\infty \beta_{i,k}^2 
                                         \leq \inv{\kappa_{i,m_i+1}^2} \sum_{k=m_i+1}^\infty \kappa_{i,k}^2 \beta_{i,k}^2 \leq \inv{\kappa_{i,m_i+1}^2} \sum_{k=r_i}^\infty \kappa_{i,k}^2 \beta_{i,k}^2,
      \end{equation}
      where the first inequality is due to the monotonicity of $\kappa_{i,k}$ as $k$ increases. Thanks to \reff{eq:jac_deriv} and the definition of $\kappa_{i,k}$, we get that $(\phi_{i,k}^{(r_i)}/\kappa_{i,k}, k \geq r_i)$ is an orthonormal basis of $L^2(v_i)$. Therefore, we have 
      \begin{equation} \label{eq:sum_kappa_beta}
      \sum_{k=r_i}^\infty \kappa_{i,k}^2 \beta_{i,k}^2  = \int_0^1 \left( \ell^{(r_i)}_{i}(t) \right)^2 v_i(t) \,dt 
      \leq  \frac{(d-i)!(i-1)!}{2^{2r_i}} \norm{\ell^{(r_i)}_{i}}^2_{L^2(q_i)},
      \end{equation}
      since $\sup_{t\in I} q_i(t)/v_i(t) = (d-i)!(i-1)!/2^{2r_i}$. This and \reff{eq:delta_mi} implies \reff{eq:delta_mi_lem}.
            
     \end{proof}

     Lemma \ref{lem:delta_mi}  yields a simple bound on the $L^2$  norm of the approximation error $\ell-\ell_m$.   
     \begin{cor} \label{cor:delta_m}
      For $m=(m_1, \hdots, m_d)$, $m_i +1 \geq r_i$ and $\ell_i \in W^2_{r_i}(q_i)$ for all $1 \leq i \leq d$, we get:
      \[
        \norm{\ell-\ell_m}_{L^2} = O\left(\sqrt{\sum_{i=1}^d m_i^{-2r_i}}\right).
      \]
     \end{cor}
     \begin{proof}
      We have:
      \[
          \norm{\ell-\ell_m}_{L^2} \leq  \sum_{i=1}^d \norm{\ell_i-\ell_{i,m_i}}_{L^2(q_i)} =O\left(\sum_{i=1}^d m_i^{-r_i}\right) =O\left(\sqrt{\sum_{i=1}^d m_i^{-2r_i}}\right),
      \]    
      where we used \reff{eq:delta_mi_lem} for the first equality.
     \end{proof}

 Lastly, we bound the $L^\infty$ norm of the approximation error. 
 \begin{lem} \label{lem:gamma_mi}
  For each $1 \leq i \leq d$, $m_i +1 \geq r_i > d$ and $\ell_i \in W^2_{r_i}(q_i)$, we have:
     \begin{equation}\label{eq:gamma_mi_lem}
      \norm{\ell_i - \ell_{i,m_i}}_\infty \leq  \frac{2^{-r_i}\sqrt{2(d-1)!}\expp{r_i}}{\sqrt{2r_i-2d-1}} \inv{(m_i+r_i)^{r_i -d-\inv{2} }}  \norm{\ell^{(r_i)}_{i}}_{L^2(q_i)}.
     \end{equation}
  
 \end{lem}
 
 \begin{proof}
  We recall the constants $\kappa_{i,k}$, $1 \leq i \leq d$, $1 \leq k \leq m_i$ given by \reff{eq:kappa_ik}.
  Since $\ell_i - \ell_{i,m_i} = \sum_{k=m_i+1}^{\infty} \beta_{i,k} \phi_{i,k}$ we have:
      \begin{align*}
       \norm{\ell_i-\ell_{i,m_i}}_\infty & =  \left\|\sum_{k=m_i+1}^\infty \beta_{i,k} \phi_{i,k} \right\|_\infty \\
                                   & \leq \sum_{k=m_i+1}^\infty \val{\beta_{i,k}}  \left\| \phi_{i,k} \right\|_\infty \\
                                   & \leq  \sqrt{ \sum_{k=m_i+1}^\infty \frac{\left\|\phi_{i,k}\right\|^2_\infty }{\kappa^2_{i,k}}} \sqrt{\sum_{k=m_i+1}^\infty \kappa^2_{i,k} \beta^2_{i,k}} \\
                                   & \leq  \sqrt{\sum_{k=m_i+1}^\infty \frac{2(d-1)!(k+d)^{2d}}{\kappa_{i,k}^2}} \sqrt{\frac{(d-i)!(i-1)!}{2^{2r_i}}} \norm{\ell^{(r_i)}_{i}}_{L^2(q_i)}  \\
                                   & \leq  \sqrt{\sum_{k=m_i+1}^\infty \frac{2(d-1)!}{(d-i)!(i-1)!} \frac{\expp{2r_i}}{(k+r_i)^{2r_i-2d}}} \sqrt{\frac{(d-i)!(i-1)!}{2^{2r_i}}} \norm{\ell^{(r_i)}_{i}}_{L^2(q_i)}  \\
                                   & \leq  \frac{2^{-r_i} \sqrt{2(d-1)!}\expp{r_i}}{\sqrt{2r_i-2d-1} \sqrt{(m_i+r_i)^{2r_i-2d-1}}}  \norm{\ell^{(r_i)}_{i}}_{L^2(q_i)},                             
      \end{align*}
      where we used Cauchy-Schwarz for the second inequality, \reff{eq:sup_phi_ik} and \reff{eq:sum_kappa_beta} for the third inequality, $\kappa^2_{i,k} \geq (d-i)!(i-1)! (k+r_i)^{2r_i} \expp{-2r_i}$ for the fourth inequality, and $\sum_{k=m_i+1}^\infty (k+r_i)^{-2r_i+2d} \leq (2r_i-2d-1)^{-1}(m_i+r_i)^{-2r_i+2d+1}$ for the fifth inequality.
 \end{proof}
 \begin{cor} \label{cor:gamma_m}
    There exists a constant $\cc>0$ such that for all  $\ell_i \in W^2_{r_i}(q_i)$ and $m_i+1 \geq r_i > d$ for all $1 \leq i \leq d$, we have:  
    \[
       \norm{\ell - \ell_m}_\infty \leq \cc \sum_{i=1}^d  \norm{\ell^{(r_i)}_{i}}_{L^2(q_i)}.
    \]
 \end{cor}
 \begin{proof} Notice that for $m_i+1 \geq r_i > d$, we have:
 \begin{equation*}
  \frac{2^{-r_i}\sqrt{2(d-1)!}\expp{r_i}}{\sqrt{2r_i-2d-1}} \inv{(m_i+r_i)^{r_i -d-\inv{2} }} \leq  \frac{2^{-r_i}\sqrt{2(d-1)!}\expp{r_i}}{\sqrt{2r_i-2d-1}} \inv{(2r_i-1)^{r_i -d-\inv{2} }},
 \end{equation*}
   and that the right hand side is bounded by a constant $\cc>0$ for all $r_i \in \N^*$. Therefore:
   \[
     \norm{\ell - \ell_m}_\infty \leq \sum_{i=1}^d \norm{\ell_i-\ell_{i,m_i}}_\infty \leq \cc \sum_{i=1}^d  \norm{\ell^{(r_i)}_{i}}_{L^2(q_i)}.
   \]
 \end{proof}

\bibliographystyle{abbrv}
\bibliography{Max_entr_ord_biblio}

 \section{Preliminary elements for the proof of Theorem \ref{theo:main_stat}} \label{sec:prelim}
 
  We adapt the results from \cite{barron1991approximation} to our setting.  
  Let us recall Lemmas 1 and 2 of \cite{barron1991approximation}.

 \begin{lem}[Lemma 1 of \cite{barron1991approximation}] \label{lem:BS1}
  Let $g,h \in \cp(\triangle)$. If $\norm{\log(g/h)}_\infty< +\infty$, then we have:
  \begin{equation} \label{eq:KL_min_logsq}
    \kl{g}{h} \geq \inv{2} \expp{-\norm{\log(g/h)}_\infty} \int_\triangle g \log^2\left(g/h\right),
  \end{equation}
  and for any $\kappa \in \R$:

  \begin{equation} \label{eq:KL_maj_logsq}
    \kl{g}{h} \leq \inv{2} \expp{\norm{\log(g/h)-\kappa}_\infty} \int_\triangle g \left( \log\left(g/h\right)-\kappa \right)^2,
  \end{equation}

 \begin{equation}\label{eq:sq_maj_logsq}
  \int_\triangle \frac{(g-h)^2}{g} \leq \expp{2\left(\norm{\log(g/h)-\kappa}_\infty-\kappa \right)} \int_\triangle g \left(\log\left(g/h\right)-\kappa \right)^2.
 \end{equation}
 \end{lem}
 Lemma \ref{lem:BS1} readily implies the following Corollary.
 
 \begin{cor}  \label{cor:L1-L2}
   Let $g, h\in \cp(\triangle)$.  If $\norm{\log(g/h)}_\infty< +\infty$, then we have, for any constant $\kappa \in \R$:
   \begin{equation} \label{eq:cor_KL}
     \kl{g}{h} \leq \inv{2} \expp{\norm{\log(g/h)-\kappa}_\infty} \norm{g}_\infty \int_\triangle \left( \log\left(g/h\right)-\kappa \right)^2,
   \end{equation}
   and:
   \begin{equation} \label{eq:cor_L2}
      \norm{ g-h}_{L^2} \leq \norm{g}_\infty \expp{\left(\norm{\log(g/h)-\kappa}_\infty-\kappa\right)} \norm{  \log\left(g/h\right)-\kappa}_{L^2}.
   \end{equation}
 \end{cor}
 
 Recall Definition \ref{defi:prod-f} for densities $f^0$ with a product form on $\triangle$. We give a few bounds between the $L^\infty$ norms of 
 $\log(f^0)$, $\ell ^0$ and the constant $\mathrm{a}_0$. 

 \begin{lem} \label{lem:f_l_a}
  Let $f^0 \in \cp(\triangle)$ given by Definition \ref{defi:prod-f}. Then we have:
  \begin{equation} \label{eq:f_l_a_1}
  \hspace{8mm} \val{\mathrm{a}_0} \leq \norm{\ell^0}_\infty + \val{\log(d!)}, \hspace{12mm} \norm{\log(f^0)}_\infty  \leq 2 \norm{\ell^0}_\infty + \val{\log(d!)},
  \end{equation} 
  \begin{equation} \label{eq:f_l_a_2}
   \val{\mathrm{a}_0} \leq \norm{\log(f^0)}_\infty , \hspace{30mm} \norm{\ell^0}_\infty \leq 2 \norm{\log(f^0)}_\infty.
  \end{equation}
 \end{lem}
 \begin{proof}
 The first part of \reff{eq:f_l_a_1} can be obtained by bounding $\ell^0$ with $\norm{\ell^0}_\infty$ in the definition of $\mathrm{a}_0$. The second part is a direct consequence of this.
 The first part of \reff{eq:f_l_a_2} can be deduced from the fact that  $\int_\triangle \ell^0 =0$. The second part is again a direct consequence of the first part.
  
 \end{proof}

Let $m\in (\N^*)^d$. Recall the application $A_m$ defined in
\reff{eq:alpha_m} and set $\Omega_m=A_m(\R^{\val{m}})$. For $\alpha\in
\R^{\val{m}}$, we define the function $\mathscr{F}_\alpha$ on $\R^{\val{m}}$ by:
\begin{equation}
   \label{eq:def-F}
\mathscr{F}_\alpha(\theta)=\theta \cdot \alpha - \psi(\theta).
 \end{equation} 
 Recall also the additive exponential series model $f_\theta$ given by \reff{eq:f_theta}.

\begin{lem}[Lemma 3 of \cite{barron1991approximation}] \label{lem:BS3}
Let $m\in (\N^*)^d$. The application $A_m$ is one-to-one from
$\R^{\val{m}}$ onto $\Omega_m$, with inverse say $\Theta_m$.   Let 
$f\in \cp(\triangle)$ such that $\alpha=\int_\triangle
\varphi_m f$ belongs to $\Omega_m$. Then for all $\theta\in
\R^{\val{m}}$, we have with $\theta^*=\Theta_m(\alpha)$:
  \begin{equation}\label{eq:pyth_eq}
    \kl{f}{f_\theta}= \kl{f}{f_{\theta^*}}+\kl{f_{\theta^*}}{f_\theta}.
  \end{equation}
Furthermore, $\theta^ *$ achieves $\max_{\theta\in \R^{\val{m}}}
\mathscr{F}_\alpha (\theta)$ as well as  $\min_{\theta\in \R^{\val{m}}}
\kl{f}{f_\theta}$. 
 \end{lem}

\begin{defi}
   \label{defi:info-proj}
 Let $m\in (\N^*)^d$. For $f \in \cp(\triangle)$ such that $\alpha=\int_\triangle
\varphi_m f\in \Omega_m$, the  probability density
$f_{\theta^*}$, with $\theta^*=\Theta_m(\alpha)$ (that is $\int_\triangle
\varphi_m f =\int_\triangle
\varphi_m f_{\theta^*}$), is called the
information projection of $f$. 
\end{defi}
The information projection of a density $f$ is the closest density in
the exponential family \reff{eq:f_theta} with respect to the
Kullback-Leibler distance to $f$.  
 
We consider the linear space of real valued functions defined on
$\triangle$ and generated by
$\varphi_m$:
\begin{equation} \label{eq:Sm}
S_m=\{\theta \cdot \varphi_m; \theta \in \R^{\val{m}}\}.
\end{equation}
Let $\kappa_m=\sqrt{2d!}\sqrt{\sum_{i=1}^d (m_i+d)^{2d}}$.
The following Lemma summarizes Lemmas \ref{lem:thetaphi_theta} and
\ref{lem:Am}. 

\begin{lem}
   \label{lem:rappel}
Let $m\in (\N^*)^d$. We have for all $g\in S_{m}$:
 \begin{equation} \label{eq:cond_norm}
   \norm{g}_{\infty} \leq  \kappa_m \norm{g}_{L^2},
 \end{equation}
For all $\theta\in \R^{\val{m}}$, we have:
 \begin{equation} \label{eq:upper_theta_phim}
 \frac{\norm{\theta }}{\sqrt{d}}  \leq \norm{\theta \cdot \phi_m}_{L^2}
 \leq \sqrt{d} \norm{\theta }. 
 \end{equation}
\end{lem}

Now we  give upper  and lower bounds  for the  Kullback-Leibler distance
between   two  members   of  the   exponential  family   $f_\theta$  and
$f_{\theta'}$    in    terms     of    the    Euclidean    distance
$\norm{\theta-{\theta'}}$.  Notice that for all $\theta \in \R^{\val{m}}$,
$\norm{\log(f_\theta)}_\infty
=\sup_{x \in \triangle} \val{\log(f_\theta(x))}$ is finite. 

 \begin{lem} \label{lem:BS4}
   Let $m\in (\N^*)^d$. For $\theta,\theta' \in \R^{\val{m}}$, we have:
   \begin{align} 
\label{eq:lem_4_1}
      \norm{\log(f_\theta/f_{\theta'})}_{\infty} 
&\leq 2 \sqrt{d}\,  \kappa_m \norm{\theta-\theta'},\\
 \label{eq:lem_4_2}
      \kl{f_\theta}{f_{\theta'}} &
\leq  \frac{d}{2}\, \expp{\norm{\log(f_{\theta})}_{\infty}+ \sqrt{d}\, 
  \kappa_m \norm{\theta-\theta'}}  \norm{\theta-\theta'}^2, \\
 \label{eq:lem_4_3}
       \kl{f_\theta}{f_{\theta'}} 
&\geq  \inv{2d} \, \expp{-\norm{\log(f_{\theta})}_{\infty}-2
  \sqrt{d}\,  \kappa_m \norm{\theta-\theta'}}  \norm{\theta-\theta'}^2. 
\end{align}
 \end{lem}
 
 \begin{proof}
Since $       \psi(\theta')-\psi(\theta) = \log \left( \int_\triangle
         \expp{(\theta'-\theta) \cdot \phi_m} f_\theta \right)$, 
we get  $\val{\psi(\theta')-\psi(\theta)} \leq \norm{(\theta'-\theta)
  \cdot \phi_m}_{\infty}$.  This implies that: 
\begin{align*}
      \norm{\log(f_\theta/f_{\theta'})}_{\infty} 
& \leq 2 \norm{(\theta-\theta') \cdot \phi_m}_{\infty} \\
& \leq 2  \kappa_m \norm{(\theta-\theta') \cdot \phi_m}_{L^2} \\
& \leq 2 \sqrt{d} \,  \kappa_m \norm{\theta-\theta'},
\end{align*}
where  we  used \reff{eq:f_theta} for the first inequality,  \reff{eq:cond_norm}   for  the  second  and
\reff{eq:upper_theta_phim}   for  the   third.   To   prove
\reff{eq:lem_4_2}, we  use  \reff{eq:KL_maj_logsq} with $\kappa=\psi(\theta')-\psi(\theta)$. This gives:
\begin{align*}
       \kl{f_\theta}{f_{\theta'}} 
& \leq \inv{2} \expp{\norm{(\theta-\theta') \cdot \phi_m}_{\infty}}
\int_\triangle f_\theta \left( (\theta-\theta') \cdot \phi_m \right)^2
\\ 
& \leq \inv{2} \, \expp{\norm{\log(f_{\theta})}_{\infty}+ \sqrt{d}\, 
  \kappa_m \norm{\theta-\theta'}} \norm{(\theta-\theta') \cdot \phi_m
}_{L^2}^2 \\ 
& \leq \frac{d}{2}\,  \expp{\norm{\log(f_{\theta})}_{\infty}+ \sqrt{d}\,
  \kappa_m \norm{\theta-\theta'}} \norm{\theta-\theta'}^2, 
\end{align*}
where we used \reff{eq:cond_norm} and \reff{eq:upper_theta_phim} for the
second inequality, and \reff{eq:upper_theta_phim} for the third.  To prove
\reff{eq:lem_4_3}, we use \reff{eq:KL_min_logsq}. We obtain:
\begin{align*}
     \kl{f_\theta}{f_{\theta'}} 
& \geq \inv{2} \expp{-\norm{\log(f_\theta/f_{\theta'}) }_{\infty}}
\int_\triangle 
f_\theta \left( (\theta-\theta') \cdot \phi_m
  -(\psi(\theta)-\psi(\theta'))\right)^2 \\  
& \geq \inv{2}\expp{-\norm{\log(f_{\theta})}_{\infty}
-2\sqrt{d}\,  \kappa_m \norm{\theta-\theta'}} \int_\triangle  \left(
  (\theta-\theta') \cdot \phi_m -(\psi(\theta)-\psi(\theta'))\right)^2
\\ 
& \geq \inv{2}\expp{-\norm{\log(f_{\theta})}_{\infty}
-2\sqrt{d}\,  \kappa_m \norm{\theta-\theta'}} 
\norm{(\theta-\theta')
  \cdot \phi_m }_{L^2}^2 \\ 
& \geq \inv{2d}\expp{-\norm{\log(f_{\theta})}_{\infty}
-2\sqrt{d} \, \kappa_m \norm{\theta-\theta'}} 
\norm{\theta-\theta' }^2, 
\end{align*}
where we used \reff{eq:lem_4_1} for the second inequality, the fact that
the functions  $(\phi_{[i],k}, 1\leq i  \leq d,  1\leq k \leq  m_i)$ are
orthogonal to the constant function with respect to the Lebesgue measure
on $\triangle$ for the  third inequality, and \reff{eq:upper_theta_phim}
for the fourth inequality.
\end{proof}
 
 Now we will show that the application $\Theta_m$  is locally Lipschitz. 
 
\begin{lem}\label{lem:BS5}
  Let $m\in (\N^*)^d$ and $\theta \in \R^{\val{m}}$. If  $\alpha \in \R^{\val{m}}$ satisfies:
\begin{equation} \label{eq:cond_BS5}
     \norm{A_m(\theta)-\alpha} \leq
     \frac{\expp{-(1+\norm{\log(f_{\theta})}_{\infty})}}{6
       d^{\frac{3}{2}} \kappa_m}, 
\end{equation}
  Then $\alpha$ belongs to $\Omega_m$ and $\theta^*=\Theta_m(\alpha)$
  exists. Let $\tau$ be such that:
  $$6d^{\frac{3}{2}}
  \expp{1+\norm{\log(f_{\theta})}_{\infty}}
  \kappa_{m}\norm{A_m(\theta)-\alpha}\leq \tau \leq 1.$$
  Then $\theta^*$ satisfies:  
\begin{align} 
\label{eq:lem_5_1}
    \norm{\theta-\theta^*} 
&\leq 3 d \expp{\tau+\norm{\log(f_{\theta})}_{\infty}}
\norm{A_m(\theta)-\alpha},    \\
\label{eq:lem_5_2}
    \norm{\log(f_\theta/f_{\theta^*})}_\infty 
&\leq  6d^{\frac{3}{2}}  \expp{\tau+\norm{\log(f_{\theta})}_{\infty}}  \kappa_m\norm{A_m(\theta)-\alpha}  \leq \tau,\\
 \label{eq:lem_5_3}
    \kl{f_\theta}{f_{\theta^*}} 
&\leq 3d  \expp{\tau+\norm{\log(f_{\theta})}_{\infty}}
\norm{A_m(\theta)-\alpha}^2. 
\end{align}
    
 \end{lem}

\begin{proof}
  Suppose that $\alpha \neq A_m(\theta)$ (otherwise the results are
  trivial). Recall $\mathscr{F}_\alpha$ defined in \reff{eq:def-F}. 
We have, for all $\theta' \in \R^{\val{m}}$:
\begin{align} \label{eq:BS5_F-F}
     \nonumber \mathscr{F}_\alpha(\theta)-\mathscr{F}_\alpha(\theta') 
& =(\theta-\theta')\cdot\alpha + \psi(\theta') -\psi(\theta) \\ 
& =  \kl{f_\theta}{f_{\theta'}}- (\theta-\theta') \cdot
(A_m(\theta)-\alpha) .
\end{align}
Using \reff{eq:lem_4_3} and the Cauchy-Schwarz inequality, we obtain the strict inequality:
\begin{equation*}
      \mathscr{F}_\alpha(\theta)-\mathscr{F}_\alpha(\theta')  >
\inv{3d} \expp{-\norm{\log(f_{\theta})}_{\infty} -2
        \sqrt{d} \, \kappa_m \norm{\theta-\theta'}} \norm{\theta-\theta'}^2 -
      \norm{\theta-\theta'}\norm{A_m(\theta)-\alpha} . 
\end{equation*}
We  consider the ball centered at $\theta$: $B_r=\{ \theta' \in
\R^{\val{m}},  \norm{\theta-\theta'}$ $ \leq r\}$ with radius $r = 3 d
\expp{\tau+\norm{\log(f_{\theta})}_{\infty}}
\norm{A_m(\theta)-\alpha}$. For all $\theta'\in \partial B_r$, we have:
\[
      \mathscr{F}_\alpha(\theta)-\mathscr{F}_\alpha(\theta') > \left(\expp{\tau-6 d^{\frac{3}{2}}
          \kappa_m
          \norm{A_m(\theta)-\alpha} \expp{\tau+\norm{\log(f_{\theta})}_{\infty}} }-1\right) 3 d
      \expp{\tau+\norm{\log(f_{\theta})}_{\infty}}
      \norm{A_m(\theta)-\alpha}^2 .
\]
The right hand side is non-negative as   $6d^{\frac{3}{2}}    \expp{1+\norm{\log(f_{\theta})}_{\infty}}
\kappa_m\norm{A_m(\theta)-\alpha}\leq \tau \leq 1$, see the
condition on $\tau$. Thus, the value
of  $\mathscr{F}_\alpha$ at  $\theta$, an  interior point  of $B_r$,  is larger  than the
values of $\mathscr{F}_\alpha$  on  $ \partial B_r$. Therefore $\mathscr{F}_\alpha$
is  maximal at a point, say $\theta^*$, in the interior  of $B_r$.  Since the gradient of
$\mathscr{F}_\alpha$ at $\theta^*$ equals $0$, we have $\nabla \mathscr{F}_\alpha(\theta^*)=
\alpha-\int_\triangle   \phi_m   f_{\theta^*}=0$,   which   means   that
$\alpha\in \Omega_m$ and  $\theta^*=\Theta_m(\alpha)$.  Since $\theta^*$
is inside $B_r$,   we  get   \reff{eq:lem_5_1}.   The   upper  bound
\reff{eq:lem_5_2} is  due to  \reff{eq:lem_4_1} of  Lemma \ref{lem:BS4}.
To prove  \reff{eq:lem_5_3}, we use  \reff{eq:BS5_F-F} and the  fact that
$\mathscr{F}_\alpha(\theta)-\mathscr{F}_\alpha(\theta^*) \leq 0$, which gives:
   \begin{equation*}
     \kl{f_\theta}{f_{\theta^*}}  \leq (\theta-\theta^*) \cdot (A_m(\theta)-\alpha) \leq \norm{\theta-\theta^*} \norm{A_m(\theta)-\alpha}  \leq 3d \expp{\tau+\norm{\log(f_{\theta})}_{\infty}} \norm{A_m(\theta)-\alpha}^2  .
   \end{equation*}
 \end{proof}

\section{Proof of Theorem \ref{theo:main_stat}} \label{sec:proof_theo_stat}

In this Section, we first show that the information projection $f_{\theta^*}$ of $f^0$ onto $\{f_\theta,  \theta \in \R^{\val{m}}  \}$ exists for all $m \in (\N^*)^d$. Moreover, the maximum likelihood estimator $\hat{\theta}_{m,n}$, defined in \reff{eq:max_likelihood} based on an i.i.d sample $\cx^n$, verifies almost surely $\hat{\theta}_{m,n} = \Theta_m(\hat{\mu}_{m,n})$ for $n \geq 2$ 
  with $\hat{\mu}_{m,n}$ the empirical mean given by \reff{eq:def_emp_mean}. Recall $\Omega_m = A_m(\R^{\val{m}})$ with $A_m$ defined by \reff{eq:alpha_m}.
  \begin{lem}\label{lem:max_likelihood_as}
    The mean $\alpha = \int_\triangle \phi_m f^0$ verifies $\alpha \in \Omega_m$ and the empirical mean $\hat{\mu}_{m,n}$ verifies $\hat{\mu}_{m,n} \in \Omega_m$ almost surely when $n \geq 2$. 
  \end{lem}
  
  \begin{rem} \label{rem:max_likelihood}
   By Lemma \ref{lem:BS3}, this also means that $\hat{\theta}_{m,n} = \argmax_{\theta \in \R^{\val{m}}} \cf_{\hat{\mu}_{m,n}}(\theta)$, and since $\cf_{\hat{\mu}_{m,n}}(\theta) = (1/n) \sum_{j=1}^n \log(f_\theta(X^j))$, the estimator $\hat{f}_{m,n} = f_{\hat{\theta}_{m,n}}$ is the maximum likelihood estimator of $f^0$ in the model $\{f_\theta, \theta \in \R^m\}$ based on $\cx^n$. 
  \end{rem}

  \begin{proof}
    Notice that $ \psi(\theta) = \log(\E[\exp(\theta \cdot \psi_m (U))])-\log(d!)$, where $U$ is a random vector uniformly distributed on $\triangle$. 
    The Hessian matrix $\nabla^2 \psi(\theta)$ is equal to the covariance matrix of $\phi_m(X)$, where $X$ has density $f_\theta$. Therefore 
    $\nabla^2 \psi(\theta)$ is positive semi-definite, and we show that it is positive definite too. Indeed, for $\lambda \in \R^{\val{m}}$, $\lambda^T \nabla^2 \psi(\theta) \lambda =0$ is equivalent to 
    $\E[(\lambda \cdot \phi_m(X))^2] =0$, which implies that $\lambda \cdot \phi_m(X)=0$ a.e. on $\triangle$. Since $(\phi_{i,k}, 1 \leq i \leq d, 1 \leq k \leq m_i)$ are linearly independent, this means $\lambda =0$. Thus $\nabla^2 \psi(\theta)$ is positive definite, providing that $\theta \mapsto \psi(\theta)$ is a strictly convex function. 
    
    Let $\psi^* : \R^{\val{m}} \rightarrow \R \cup \{+\infty\}$ denote the Legendre-Fenchel transformation of the function $\theta \mapsto \psi(\theta)$, i.e. for $\alpha \in \R^{\val{m}}$:
    \[
      \psi^*(\alpha) = \sup_{\theta \in \R^{\val{m}}} \alpha \cdot \theta - \psi(\theta) =  \sup_{\theta \in \R^{\val{m}}}  \cf_{\alpha}(\theta).
    \]
    Suppose that $\alpha \in \Omega_m$. Then according to Lemma \ref{lem:BS3}, $\psi^*(\alpha) = \cf_{\alpha}(\theta^*)$ with $\theta^* = \Theta_m(\alpha)$, thus  $\psi^*(\alpha)$ is finite. Therefore $\Omega_m \subseteq \Dom(\psi^*)$, where $\Dom(\psi^*)=\{\alpha \in \R^{\val{m}}: \psi^*(\alpha) < + \infty\}$. Inversely, let $\alpha \in \Dom(\psi^*)$. This ensures that $\theta^* = \argmax_{\theta \in \R^{\val{m}}} \cf_{\alpha}(\theta)$ exists uniquely, since $\cf_\alpha (\theta)$ is finite for all $\theta \in \R^{\val{m}}$, $\alpha \in \R^{\val{m}}$. This also implies that:
    \[
     0= \nabla \cf_\alpha(\theta^*) = \alpha - \int_\triangle \phi_m f_{\theta^*} =  \alpha - A_m(\theta^*),  
    \]
    giving $\alpha \in \Omega_m$. Thus we obtain $\Omega_m = \Dom(\psi^*)$. 
    By Lemma \ref{lem:BS5}, we have that $\Omega_m$ is an open subset of $\R^{\val{m}}$. Set $\Upsilon = \inter(\cv(\supp(\phi_m(U))))$, where $\inter(A)$ and $\cv(A)$ is the interior and convex hull of a set $A \subseteq \R^{\val{m}}$, respectively. Thanks to Lemma 4.1. of \cite{abraham2015critical}, we have $\Dom( \psi^*) = \Upsilon$. The proof is complete as soon as we prove that $\alpha \in \Upsilon$ and  $\hat{\mu}_{m,n} \in \Upsilon$ almost surely when $n \geq 2$. Since $(\phi_{i,k}, 1 \leq i \leq d, 1 \leq k \leq m_i)$ are linearly independent polynomials, they coincide only on a finite number of points. This directly implies that $\alpha \in \Upsilon$. To show that $\hat{\mu}_{m,n} \in \Upsilon$,  notice that the probability measures of $\phi_m(X)$ and $\phi_m(U)$ are equivalent. Therefore it is sufficient to prove that $(1/n) \sum_{j=1}^n \phi_m(U^j) \in \Upsilon$, with $(U^1, \hdots, U^n)$ i.i.d. random vectors uniformly distributed on $\triangle$. The linear independence of $(\phi_{i,k}, 1 \leq i \leq d, 1 \leq k \leq m_i)$ and the fact that $U^j$, $1 \leq j \leq n$ are uniformly distributed on $\triangle$ easily implies that for $n \geq 2 $,  $1/n \sum_{j=1}^n \phi_m(U^j) \in \Upsilon$, and the proof is complete. 
    
  \end{proof}

  We divide the proof of Theorem \ref{theo:main_stat} into two parts:  first we bound the error due to the
bias of the  proposed exponential model, then we bound  the error due to
the variance of  the sample estimation. We formulate the  results in two
general  Propositions, which  can  be later  specified  to get  Theorem
\ref{theo:main_stat}. 
  
\subsection{Bias of the estimator} \label{sec:bias}
  
The  bias  error    comes from  the  information
projection of the  true underlying density $f^0$ onto the  family of the
exponential series  model $\{f_\theta,  \theta \in \R^{\val{m}}  \}$. We
recall the  linear space $S_m$ spanned  by $(\phi_{[i], k},  1  \leq k
\leq  m_i, 1  \leq i \leq  d )$ where $\phi_{i,k}$ is a polynomial of degree $k$, and the form of the probability density
$f^0$ given in \reff{eq:Pi-form}. For $1\leq i\leq  d$, let $\ell^0_{i, m} $ be the orthogonal
projection in $L^2(q_i)$ of $\ell^0_i$ on the vector space spanned by
$(\varphi_{i,k}, 0\leq k\leq m_i)$ or equivalently on the vector space spanned by
$(\varphi_{i,k}, 1\leq k\leq m_i)$, as we assumed that $\int_I \ell_i^0
q_i=0$. We set $\ell^0_m=\sum_{i=1}^d \ell^0_{[i], m}$ the approximation
of $\ell^0$ on $S_m$. In particular we have $\ell ^0_m= \theta^0\cdot
\varphi_m$ for some $\theta^0\in \R^{\val{m}}$. 
 Let:
\[
       \Delta_{m}=\norm{\ell^0-\ell^0_{m}}_{L^2}
\quad\text{and}\quad
       \gamma_{m}= \norm{\ell^0-\ell^0_{m}}_{\infty}
\]
    denote the  $L^2$ and $L^\infty$ errors of the approximation of
    $\ell^0$ by $\ell^0_m$ on the simplex $\triangle$. 
   
\begin{prop} \label{prop:general_1}
Let $f^0\in \cp(\triangle)$ have a product form given by Definition \ref{defi:prod-f}.
Let $m\in (\N^*)^d$.   The  information  projection
    $f_{\theta^*}$ of $f^0$ exists  (with $\theta^*\in \R^{\val{m}}$ and
    $\int_\triangle  \varphi_m  f_{\theta^*}  =\int_\triangle  \varphi_m
    f^0$)    and  verifies, with $\mathfrak{A}_1=\frac{1}{2}\expp{\gamma_{m}+\norm{\log(f^0)}_\infty}$:
    \begin{equation} \label{eq:theo_gen_1}
      \kl{f^0}{f_{\theta^*}} \leq \mathfrak{A}_1 \Delta_{m}^2.
    \end{equation}  
\end{prop}

\begin{proof}

The existence of $\theta^*$ is due to Lemma \ref{lem:max_likelihood_as}.  
Thanks to Lemma \ref{lem:BS3} and \reff{eq:cor_KL} with $\kappa=\psi(\theta^0) - \mathrm{a}_0$, we
can deduce that: 
\begin{equation*}
\kl{f^0}{f_{\theta^*}}  \leq \kl{f^0}{f_{\theta_m^0}}  \leq  \inv{2} \expp{\norm{\ell^0-\ell^0_{m}}_\infty} \norm{f^0}_\infty
\norm{\ell^0-\ell^0_{m}}^2_{L^2} \leq \inv{2} \expp{\gamma_{m}+\norm{\log(f^0)}_\infty} \Delta_{m}^2.  
\end{equation*}

\end{proof}
    Set:
    \begin{equation}  \label{eq:def_eps_m}
      \varepsilon_{m} =  6d^{\frac{5}{2}} \kappa_m \Delta_{m} \expp{(4\gamma_{m}+2\norm{\log(f^0)}_\infty+1)} .     
    \end{equation}
    We need the following lemma to control $\norm{\log(f^0/f_{\theta^*})}_\infty$.
    
 \begin{lem}   
    If  $\varepsilon_{m}  \leq  1$, we also have: 
    \begin{equation} \label{eq:logf*_ftheta*}
      \norm{\log(f^0/f_{\theta^*})}_\infty \leq 2 \gamma_m + \varepsilon_m \leq 2 \gamma_m +1.
    \end{equation}
\end{lem}

\begin{proof}

To show \reff{eq:logf*_ftheta*}, let  $f^0_m=f_{\theta^0}$
  denote the density function in the exponential family corresponding to
  $\theta^0$, and $\alpha^0=\int _\triangle \varphi_m f^0$.   
  For each  $1
  \leq i \leq d$, the functions  $\phi_{i,m} = (\phi_{[i],k}$, $1 \leq k
  \leq m_i)$  form an orthonormal  set with respect  to the Lebesgue
  measure on $\triangle$.  We set $\alpha^0_{i,m}=\int_\triangle
  \varphi_{i,m} f^0$ and $A_{i,m}(\theta^0)=\int_\triangle
  \varphi_{i,m} f_{\theta^0}$. By  Bessel's inequality, we
  have for $1 \leq i \leq d$:
     \[
        \norm{\alpha^0_{i,m}-A_{i,m}(\theta^0)} \leq  \norm{
          f^0-f^0_m}_{L^2}.   
     \]
Summing up these inequalities for $1\leq i\leq d$, we get:
\begin{align*}
\norm{\alpha^0-A_{m}(\theta^0)}  
& \leq \sum_{i=1}^d \norm{\alpha^0_{i,m}-A_{i,m}(\theta^0)} \\
& \leq d \norm{ f^0-f^0_m}_{L^2} \\
& \leq d \norm{f^0}_\infty
\expp{\left(\norm{\ell^0-\ell^0_{m}}_{\infty}-(\psi(\theta^0) - \mathrm{a}_0)\right)} \norm{
  \ell^0-\ell^0_{m}}_{L^2} \\ 
& \leq d  \expp{\norm{\log(f^0)}_\infty + 2\gamma_{m}} \Delta_m,            
\end{align*}
where we used \reff{eq:cor_L2} with $\kappa=\psi(\theta^0) - \mathrm{a}_0$ for
the third inequality and $\val{\psi(\theta^0) - \mathrm{a}_0} \leq \gamma_m$ 
(due to $\psi(\theta^0) - \mathrm{a}_0 = \log(\int \exp(\ell^0_{m}-\ell^0) f^0 )$) for the fourth inequality. 
The  latter  argument also ensures  that  $\norm{\log(f^0/f^0_m)}_\infty  \leq
2\gamma_{m}$.  In order to   apply    Lemma     \ref{lem:BS5}    with
$\theta=\theta^0$,   $\alpha=\alpha^0$, we check condition \reff{eq:cond_BS5}, which is implied by:
\[
       d  \expp{\norm{\log(f^0)}_\infty+2\gamma_{m}} \Delta_m \leq
       \frac{\expp{-(1+\norm{\log(f^0_m)}_{\infty})}}{6 d^{\frac{3}{2}}
         \kappa_m} \cdot
\]
    Since $\norm{\log(f^0_m)}_{\infty} \leq \norm{\log(f^0)}_\infty + \norm{\log(f^0/f^0_m)}_\infty \leq \norm{\log(f^0)}_\infty + 2\gamma_m$, this condition is ensured whenever $\varepsilon_m \leq 1$.
In this case we deduce, thanks to \reff{eq:lem_5_2} with $\tau=1$, that $\norm{\log(f^0_m/f_{\theta^*})}_\infty \leq \varepsilon_{m} $. By the triangle inequality, we obtain 
$\norm{\log(f^0/f_{\theta^*})}_\infty \leq 2\gamma_{m} + \varepsilon_{m}$. This completes the proof.
\end{proof}

  \subsection{Variance of the estimator} \label{sec:variance}
  
  We control the variance error due to the parameter estimation by the size of the sample. We keep the notations used in Section \ref{sec:bias}. In particular $\varepsilon_m$ is defined by \reff{eq:def_eps_m} and 
    $\kappa_m=\sqrt{2d!}\sqrt{\sum_{i=1}^d (m_i+d)^{2d}}$. The results are summarized in the following proposition. 
     
 \begin{prop} \label{prop:general_2}
     Let $f^0\in \cp(\triangle)$ have a product form given by Definition \ref{defi:prod-f}.
     Let $m\in (\N^*)^d$ and suppose that $\varepsilon_m \leq 1$. Set:  
    \[
      \delta_{m,n}=6d^{\frac{3}{2}}\kappa_m\sqrt{\frac{\val{m}}{n}} \expp{2\gamma_{m}+\norm{\log(f^0)}_\infty +2}.
    \]
    If $\delta_{m,n} \leq 1$, then for every $0 < K \leq \delta_{m,n}^{-2}$, we have:
    \begin{equation} \label{eq:theo_gen_2}
      \P \left( \kl{f_{\theta^*}}{\hat{f}_{m,n}} \geq \mathfrak{A}_2 \frac{\val{m} }{n}K \right) \leq \exp(\norm{\log(f^0)}_\infty)/K .
    \end{equation}
    where $\mathfrak{A}_2=3 d \expp{2\gamma_{m} + \varepsilon_{m}+\norm{\log(f^0)}_\infty +\tau} $, and $\tau=\delta_{m,n} \sqrt{K} \leq 1$. 

  \end{prop}
  \begin{proof}
 
   Let $\theta^*$ be defined in Proposition \ref{prop:general_1}. Let $X=(X_1, \hdots, X_d)$ denote a random variable with density $f^0$.  Let $\theta$ in Lemma \ref{lem:BS5} be equal to $\theta^*$, which gives $A_m(\theta^*) = \alpha^0=\mathbb{E} [\phi_m(X)]$, and for $\alpha$, we take the empirical mean $\hat{\mu}_{m,n}$. With this setting, we have: 
   \[ 
     \norm{\alpha-\alpha^0}^2 = \sum_{i=1}^d \sum_{k=1}^{m_i} \left( \hat{\mu}_{m,n,i,k}-\mathbb{E} [ \phi_{i,k}(X_i) ]\right)^2.
   \]
    By Chebyshev's  inequality $\norm{\alpha-\alpha^0}^2 \leq \val{m}K/n $ except on a set whose probability verifies:
    \begin{align*}
       \P\left( \norm{\alpha-\alpha^0}^2 > \frac{\val{m}}{n}K \right)  \leq \frac{1}{\val{m}K} \sum_{i=1}^d \sum_{k=1}^{m_i} \sigma_{i,k}^2 .
     \end{align*}
     with $\sigma^2_{i,k}=\Var[\phi_{i,k}(X_i)]$. We have the upper bound $\sigma^2_{i,k} \leq \norm{f^0}_\infty \int_\triangle \phi_{[i],k}^2 \leq  \expp{\norm{\log(f^0)}_\infty}$ by the normality of $\phi_{i,k}$. Therefore we obtain:
     \begin{align*}
       \P\left( \norm{\alpha-\alpha^0}^2 > \frac{\val{m}}{n}K \right) & \leq \frac{\expp{\norm{\log(f^0)}_\infty}}{K} \cdot
     \end{align*}   
     We can apply Lemma \ref{lem:BS5} on the event $\{ \norm{\alpha-\alpha^0}  \leq \sqrt{\val{m}K/n} \}$ if:
     \begin{equation} \label{eq:var_cond_lem}
        \sqrt{ \frac{\val{m}}{n} K} \leq \frac{\expp{-(1+\norm{\log(f_{\theta^*})}_{\infty})}}{6 d^{\frac{3}{2}} \kappa_m} \cdot
     \end{equation}
     Thanks to \reff{eq:logf*_ftheta*} we have:
     \begin{equation} \label{eq:logf*}
      \norm{\log(f_{\theta^*})}_{\infty}  \leq \norm{\log(f^0/f_{\theta^*})}_{\infty}+\norm{\log(f^0)}_\infty \leq 2\gamma_{m} + \varepsilon_{m}+\norm{\log(f^0)}_\infty.
     \end{equation}
     Since $\varepsilon_m \leq 1$, \reff{eq:var_cond_lem} holds  if $\delta^2_{m,n} \leq 1/K$. Then except on a set of probability less than $\expp{\norm{\log(f^0)}_{\infty}}/K$, the maximum likelihood estimator $\hat{\theta}_{m,n}$ satisfies, thanks to \reff{eq:lem_5_3} with $\tau=\delta_{m,n}\sqrt{K}$:
     \begin{equation}
       \kl{f_{\theta^*}}{f_{\hat{\theta}_{m,n}}} \leq 3d\expp{\norm{\log(f_{\theta^*})}_{\infty}+\tau} \frac{\val{m}}{n} K \leq 3d \expp{2\gamma_{m} + \varepsilon_{m}+\norm{\log(f^0)}_\infty+\tau}  \frac{\val{m}}{n} K.
     \end{equation}
     
    \end{proof}

  \subsection{Proof of Theorem \ref{theo:main_stat}}    \label{sec:proof_stat}

     Recall that $r = (r_1, \hdots, r_d) \in \N^d$ is fixed. We assume $\ell_i^0 \in W^2_{r_i}(q_i)$ for all $1 \leq i \leq d$. Corollary \ref{cor:delta_m} ensures $\Delta_m = O(\sqrt{\sum_{i=1}^d m_i^{-2r_i}})$  and the boundedness of $\gamma_m$ when $m_i > r_i$ for all $1 \leq i \leq d$ is due to Corollary \ref{cor:gamma_m}. By Remark \ref{cor:Am}, we have that $\kappa_m=O(  \val{m}^{d})$. If \reff{eq:cond_m_theo_1} holds, then 
     $\kappa_m \Delta_m$ converges to 0. Therefore for $m$ large enough, we have that $\varepsilon_m$  defined in \reff{eq:def_eps_m} is less than $1$. By Proposition \ref{prop:general_1}, the information projection $f_{\theta^*}$ of $f^0$ exists. For such $m$, by Lemma \ref{lem:BS3}, we have that for all $\theta \in \R^{\val{m}}$: 
     \[
       \kl{f^0}{f_\theta}= \kl{f^0}{f_{\theta^*}}+\kl{f_{\theta^*}}{f_\theta}.
     \]
      Proposition \ref{prop:general_1} and $\Delta_m = O(\sqrt{\sum_{i=1}^d m_i^{-2r_i}})$ ensures that the $\kl{f^0}{f_{\theta^*}}=O(\sum_{i=1}^d m_i^{-2r_i})$. The condition $\delta_{m,n} \leq 1$ in Proposition \ref{prop:general_2} is verified for $n$ large enough since $\gamma_m$ is bounded and \reff{eq:cond_m_theo_2} holds, giving $\lim_{n \rightarrow \infty} \delta_{m,n} = 0$. Proposition \ref{prop:general_2} then ensures that $\kl{f_{\theta^*}}{\hat{f}_{m,n}}=O_\P(\val{m}/n)$.
     Therefore the proof is complete.

\section{Proof of Theorem \ref{theo:adapt}} \label{sec:proof_theo_adapt}

   In this section we provide the elements of the proof of Theorem \ref{theo:adapt}. We assume the hypotheses of Theorem \ref{theo:adapt}. Recall the notation of Section \ref{sec:adaptation}.  We shall stress out when we use the inequalities \reff{eq:R_n_1}, \reff{eq:R_n_2} and \reff{eq:R_n_3} to achieve uniformity in $r$ in Corollary \ref{cor:adapt_est}.

First recall that $\ell^0$ from \reff{eq:Pi-form} admits the following representation: $\ell^0=\sum_{i=1}^d \sum_{k=1}^\infty \theta^0_{i,k} \phi_{[i],k}$. For $m=(m_{1}, \hdots, m_{d}) \in (\N^*)^d$, let $\ell^0_m = \sum_{i=1}^d \sum_{k=1}^{m_{i}} \theta^0_{i,k} \phi_{[i],k}$ and $f^0_{m} = \exp( \ell^0_m - \psi(\theta^0_{m}))$. Using Corollary \ref{cor:gamma_m} and  $\val{\psi(\theta^0_m)-\mathrm{a}_0} \leq \norm{\ell^0_m - \ell^0}_\infty$,  we obtain that $\norm{\log(f^0_{m}/f^0)}_\infty$ is bounded for all $m \in (\N^*)^d$ such that $m_i \geq r_i$:
    \begin{equation} \label{eq:gamma}
      \norm{\log(f^0_{m}/f^0)}_\infty \leq 2 \gamma_{m} \leq 2 \gamma,
    \end{equation}
    with $\gamma_{m}=\norm{\ell^0_m - \ell^0}_\infty$, and $\gamma=\cc \sum_{i=1}^d \norm{\ell^{(r_i)}_{i}}_{L^2(q_i)}$ with $\cc$ defined in Corollary \ref{cor:gamma_m} which does not depend on $r$ or $m$.  
    For $m = (v,\hdots,v) \in \cm_n$, we have that $a_n \leq v \leq b_n$, 
    with $a_n,b_n$ given by:    
    \begin{equation} \label{eq:an_bn}
      a_n = \left\lfloor n^{1/(2(d+N_n)+1)} \right\rfloor \quad \text{ and } \quad b_n = \left\lfloor n^{1/(2(d+1)+1)} \right\rfloor.
    \end{equation}   
    The upper bound \reff{eq:gamma} is uniform over $m \in \cm_n$ and $r \in (\crr_n)^d$ when \reff{eq:R_n_2} holds.
    Since $N_n = o(\log(n))$, we have 
     $\lim_{n \rightarrow +\infty} a_n = + \infty$. Hence, for $n$ large enough, say $n \geq n^*$, we have $\varepsilon_m \leq 1 $ for all $m= (v,\hdots,v) \in \cm_n$ with $\varepsilon_m$ given by \reff{eq:def_eps_m}, since $\kappa_m \Delta_m = O(a_n^{d-\min(r)})$.  According to Proposition \ref{prop:general_1}, this means that the information projection $f_{\theta^*_{m}}$ of $f$ onto the set of functions $(\phi_{[i],k}, 1\leq i \leq d, 1 \leq k \leq v)$ verify, by \reff{eq:lem_5_2} with $\tau=1$, for all $m \in \cm_n$:
    \begin{equation} \label{eq:epsilon_m}
      \norm{\log(f_{\theta^*_{m}}/f^0_{m})}_\infty \leq 1. 
    \end{equation}
    
    Recall the notation $A^0_{m} = \int_\triangle \phi_{m} f^0 $ for the expected value of $\phi_{m}(X^1)$, $\hat{\mu}_{m,n}$ the corresponding empirical mean based on the sample $\cx^n_1$, and $\hat{\ell}_{m,n} = \hat{\theta}_{m,n} \cdot \phi_m$ where $\hat{\theta}_{m,n}$ is the maximum likelihood estimate given by \reff{eq:max_likelihood}.  Let $T_n >0 $ be defined as:   
    \begin{equation} \label{eq:T_cond}
      T_n = \frac{n_1 \expp{-4\gamma-4-2\norm{\log(f^0)}_\infty}}{36 d^5 d! b_n  (b_n+d)^{2d} \log(b_n)} ,
    \end{equation}
    with $b_n$ given by \reff{eq:an_bn} and $\gamma$ as in \reff{eq:gamma}.
    We define the sets:
    \[
       \cb_{m,n} = \{ \norm{A^0_{m}-\hat{\mu}_{m,n}}^2 > \val{m} T_n \log(b_n)/n_1 \} \quad \text{and} \quad \ca_n = \left(\bigcup_{m \in \cm_n} \cb_{m,n} \right)^c.
    \]
    We first show that with probability converging to $1$, the estimators are uniformly bounded. 
   
   \begin{lem} \label{lem:Mn_exist}
   
   Let $n \in \N^*$, $n \geq n^*$ and $\cm_n$ as in \reff{eq:def_Mn}. 
    Then we have:
    \[
         \P(\ca_n) \geq 1-N_n 2 dn^{C_{T_n}}, 
    \]
    with $C_{T_n}$ defined as: 
    \[
      C_{T_n} = \inv{2d+3}\left(1-\frac{T_n}{2\norm{f^0}_\infty + C\sqrt{T_n}} \right),
    \]
    with a finite constant $C$ given by \reff{eq:lemma_8_3_const}. Moreover, on the event $\ca_n$, we have the following uniform upper bound for $\norm{\hat{\ell}_{m,n}}_\infty$,  $m \in \cm_n$:
    \begin{equation} \label{eq:l_bound}
      \norm{\hat{\ell}_{m,n}}_\infty \leq 4 + 4 \gamma + 2 \norm{ \log(f^0)}_\infty.
    \end{equation}

   \end{lem}
   
   \begin{rem} \label{rem:Mn_exist} 
     Notice that by the definition of $b_n$, $\lim_{n \rightarrow \infty} T_n = +\infty$.
    For $n$ large enough, we have $C_{T_n} < -\varepsilon < 0$ for some positive $\varepsilon$, so that:
    \begin{equation} \label{eq:A_N_small}
       \lim_{n \rightarrow \infty} N_n 2 d n^{C_{T_n}} = 0.
    \end{equation}
     This ensures that $\lim_{n \rightarrow \infty} \P(\ca_n) =1$, that is $(\hat{\ell}_{m,n}, m \in \cm_n)$ are uniformly bounded with probability converging to $1$. 
     \end{rem}

   \begin{proof}
   
    For $m = (v,\hdots,v) \in \cm_n$ fixed, in order to bound the distance between the vectors $\hat{\mu}_{m,n}=(\hat{\mu}_{m,n,i,k}, 1 \leq i \leq d, 1 \leq k \leq v)$ and $A^0_{m}=\E[\hat{\mu}_{m,n}]=(\alpha^0_{i,k}, 1\leq i \leq d, 1 \leq k \leq v)$, we first consider a single term $\val{\alpha^0_{i,k} - \hat{\mu}_{m,n,i,k}}$. By Bernstein's inequality, we have for all $t>0$:  
    \begin{align*}
    \P\left(\val{\alpha^0_{i,k} - \hat{\mu}_{m,n,i,k}} > t\right) & \leq 2 \exp\left(-\frac{(n_1 t)^2/2}{ n_1 \Var \phi_{[i],k}(X^1) + 2n_1t\norm{\phi_{i,k}}_\infty/3}\right) \\
                                                     & \leq 2 \exp\left(-\frac{(n_1 t)^2/2}{n_1 \E \left[ \phi_{[i],k}^2(X^1)\right] +  2 n_1 t\sqrt{2(d-1)!} (b_n+d)^{d-\inv{2}}/3}\right) \\ 
                                                     & \leq 2 \exp\left(-\frac{n_1 t^2/2}{\norm{f^0}_\infty + 2t\sqrt{2(d-1)!} (b_n+d)^{d-\inv{2}}/3}\right), 
    \end{align*}
    where we used, thanks to \reff{eq:sup_phi_ik}:
    \[
      \norm{\phi_{i,k}}_\infty \leq \sqrt{(d-1)!} \sqrt{2k+d} \frac{(k+d-1)!}{k!} \leq \sqrt{2(d-1)!}(b_n+d)^{d-\inv{2}}
    \]
    for the second inequality, and the orthonormality of $\phi_{[i],k}$ for the third inequality. Let us choose $t = \sqrt{T_n\log(b_n)/n_1}$.
    This gives:
    \begin{align}
      \nonumber \P\left(\val{\alpha^0_{i,k} - \hat{\mu}_{m,n,i,k}} > \sqrt{\frac{T_n\log(b_n)}{n_1}} \right) & \leq 2 \exp\left(-\frac{T_n\log(b_n)/2}{\norm{f^0}_\infty + 2\sqrt{\frac{2 T_n\log(b_n) (d-1)! (b_n+d)^{2d-1}}{9n_1}}}\right) \\
                                                                                        & \leq 2 b_n^{-\frac{T_n}{2\norm{f^0}_\infty + C\sqrt{T_n}}},
    \end{align}
    with $C$ given by:
    \begin{equation} \label{eq:lemma_8_3_const}
      C= \sup_{n \in \N^*}  4\sqrt{\frac{2 \log(b_n) (d-1)! (b_n+d)^{2d-1}}{9n_1}} \cdot
    \end{equation}
    Notice $C<+\infty$ since the sequence $\sqrt{\log(b_n)(b_n+d)^{2d-1}/9n_1}$  is $o(1)$. 
    For the probability of $\cb_{n,m}$ we have:
    \begin{align*}
     \P\left(\cb_{n,m} \right) 
       & \leq  \sum_{i=1}^d \sum_{k=1}^{v} \P\left( \val{\alpha^0_{i,k} - \hat{\mu}_{m,n,i,k}}^2 > \frac{T_n\log(b_n)}{n_1} \right) \\
       & \leq  \sum_{i=1}^d \sum_{k=1}^{v} 2 b_n^{-\frac{T_n}{2\norm{f^0}_\infty + C \sqrt{T_n}}} \\ 
       & \leq  2 d n^{C_{T_n}}.
    \end{align*}
     This implies the following lower bound on $\P(\ca_n)$ :
    \begin{equation*} 
        \P(\ca_n) = 1-\P\left(\bigcup_{m \in \cm_n } \cb_{n,m}\right) \geq 1-\sum_{m \in \cm_n} \P(\cb_{n,m})  \geq 1-N_n 2 dn^{c_{T_n}}.
    \end{equation*}
     On $\ca_{n}$, by the definition of $T_n$,  we have for all $m \in \cm_n$:
    \begin{equation*} \label{eq:Am_mumn}
        \norm{A_{m}^0-\hat{\mu}_{m,n}} 6d^{2} \sqrt{2d!} (v+d)^{d} \expp{2\gamma_{m}+2}  
                                                                         \leq \sqrt{ b_n \frac{T_n \log(b_n)}{n_1}} 6d^{\frac{5}{2}} \sqrt{2d!} (b_n+d)^{d} \expp{2\gamma+2} = 1.
    \end{equation*}
     Notice that whenever \reff{eq:Am_mumn} holds, condition \reff{eq:cond_BS5} of Lemma \ref{lem:BS5} is satisfied with $\theta = \theta^*_m$ and $\alpha = \hat{\mu}_{m,n}$, thanks to $\kappa_m \leq \sqrt{d 2 d!}(b_n+d)^{d}$ and:
     \[
       \norm{\log(f_{\theta^*_m})}_\infty \leq \norm{\log(f_{\theta^*_m}/f^0_m)}_\infty + \norm{\log(f^0_m/f^0)}_\infty +\norm{\log(f^0)}_\infty \leq 1 +2 \gamma + \norm{\log(f^0)}_\infty .
     \]    
     According to Equation \reff{eq:lem_5_2} with $\tau =1$, we can deduce that  on $\ca_n$, we have:
     \[    
      \norm{ \log(\hat{f}_{m,n}/f_{\theta^*_{m}})}_\infty \leq 1 \quad \text{for all } m \in \cm_n, n \geq n^*.
      \]
      This, along with \reff{eq:gamma} and \reff{eq:epsilon_m}, provide the following uniform upper bound for $(\norm{\hat{\ell}_{m,n}}_\infty, m \in \cm_n )$  on $\ca_n$:
    \begin{align*}
   \inv{2}  \norm{\hat{\ell}_{m,n}}_\infty & \leq  \norm{ \log(\hat{f}_{m,n})}_\infty  \\
                             & \leq  \norm{ \log(\hat{f}_{m,n}/f_{\theta^*_{m}})}_\infty
     +  \norm{ \log(f_{\theta^*_{m}}/f^0_{m})}_\infty +  \norm{\log(f^0_{m}/f^0)}_\infty +  \norm{ \log(f^0)}_\infty \\
                             & \leq 2 + 2 \gamma +  \norm{ \log(f^0)}_\infty,
    \end{align*}
     where we used \reff{eq:f_l_a_2} for the first inequality.

     \end{proof}

     We also give a sharp oracle inequality for the convex aggregate estimator $f_{\hat{\lambda}^*_n}$ conditionally on $\ca_n$ with $n$ fixed . The following lemma is a direct application of Theorem 3.6. of \cite{butucea2016optimal} and \reff{eq:l_bound}. 

    \begin{lem} \label{lem:aggreg_exp}
    Let $n \in \N^*$ be fixed. Conditionally on  $\ca_n$, let $f_{\hat{\lambda}^*_n}$  be given by \reff{eq:aggr_f} with $\hat{\lambda}^*_n$ defined as in \reff{eq:aggr_lambda}. Then for any $x>0$ we have with probability greater than $1-\exp(-x)$:
    \begin{equation}
      \kl{f^0}{f_{\hat{\lambda}^*_n}} - \mathop{\min}_{m \in\cm_n} \kl{f^0}{\hat{f}_{m,n}} \leq   \frac{\beta (\log(N_n)+x)}{n_2},     
    \end{equation}
     with $\beta=2\exp(6K+2L)+4K/3$, and $L,K \in \R$ given by :
     \[
        L = \norm{\ell^0}_\infty,  \quad \quad \quad  K = 4+4\gamma + 2 \norm{\log(f^0)}_\infty,
     \]
     with $\gamma$ as in \reff{eq:gamma}.
    \end{lem}
    
    \bigskip
    
    Now we prove Theorem \ref{theo:adapt}. 
     For $n \in \N^*$ and $C>0$, we define the event $\cd_n(C)$ as:
    \[
       \cd_n(C) = \left\{\kl{f^0}{f_{\hat{\lambda}^*_n}} \geq C \left(  n^{-\frac{2 \min(r)}{2 \min(r) +1}} \right) \right\}.
    \]
    
     Let $\varepsilon > 0$. To prove \reff{eq:kl_adapt}, we need to find $C_\varepsilon>0$ such that for all $n$ large enough:
    \begin{equation} \label{eq:C_vareps}  
        \P\left( \cd_n(C_\varepsilon) \right) \leq \varepsilon.  
    \end{equation}
    We decompose the left hand side of \reff{eq:C_vareps} according to $\ca_n$:
    \begin{equation} \label{eq:kl_cond}
         \P\left( \cd_n(C_\varepsilon) \right) \leq  \P\left(\cd_n(C_\varepsilon) \, \middle| \, \ca_n \right)\P(\ca_n) + \P(\ca^c_n).
    \end{equation}
    The product $\P\left(\cd_n(C_\varepsilon) \, \middle| \, \ca_n \right)\P(\ca_n)$ is bounded by:
    \[
      \P\left(\cd_n(C_\varepsilon) \, \middle| \, \ca_n \right)\P(\ca_n) \leq  A_n(C_\varepsilon) +  B_n(C_\varepsilon),
    \]
    with $A_n(C_\varepsilon)$ and $ B_n(C_\varepsilon)$ defined by:
    \begin{align*}
      A_{n}(C_\varepsilon) & = \P\left( \kl{f^0}{f_{\hat{\lambda}^*_n}} - \min_{ m \in \cm_n} \kl{f^0}{\hat{f}_{m,n}}  \geq \frac{C_\varepsilon}{2} \left( n^{-\frac{2 \min(r)}{2 \min(r) +1}} \right) \, \middle| \, \ca_n \right), \\
      B_{n}(C_\varepsilon) & = \P\left( \min_{m \in \cm_n} \kl{f^0}{\hat{f}_{m,n}} \geq \frac{C_\varepsilon}{2} \left( n^{-\frac{2 \min(r)}{2 \min(r) +1}} \right)  \right).
    \end{align*}
     To bound $A_n(C_\varepsilon)$ we apply Lemma \ref{lem:aggreg_exp} with $x=x_\varepsilon=-\log(\varepsilon/4)$:
     \[
        \P\left( \kl{f^0}{f_{\hat{\lambda}^*_n}} - \min_{ m \in \cm_n} \kl{f^0}{\hat{f}_{m,n}}  \geq \frac{\beta(\log(N_n)+x_\varepsilon)}{n_2} \, \middle| \, \ca_n \right) \leq \frac{\varepsilon}{4} \cdot
     \]
     Let us define $C_{\varepsilon,1}$ as:
     \begin{equation}\label{eq:C1eps}
        C_{\varepsilon,1} = \sup_{n \in \N^*} \left(\frac{\beta(\log(N_n)+x_\varepsilon)}{n_2 n^{-\frac{2 \min(r)}{2 \min(r) +1}}}\right) . 
     \end{equation}
     Since $N_n = o(\log(n))$, we have $C_{\varepsilon,1} < +\infty$ as the sequence on the right hand side of \reff{eq:C1eps} is $o(1)$. This bound is uniform over regularities in $(\crr_n)^d$ thanks to \reff{eq:R_n_3} Therefore for all $C_\varepsilon \geq C_{\varepsilon,1}$, we have $A_n(C_\varepsilon) \leq \varepsilon/4$. 
     
     For $B_n(C_\varepsilon)$, notice that if $n\geq \bar{n}$ with $\bar{n}$ given by \reff{eq:def_bar_n}, then $m^* =(v^*, \hdots, v^*) \in \cm_n$ with $v^*= \lfloor n^{1/(2 \min(r) +1)} \rfloor$. This holds for all $r \in (\crr_n)^d$ due to \reff{eq:R_n_1}. By Remark \ref{rem:optim}, we have that $\kl{f^0}{\hat{f}_{m^*,n}} = O_\P(n^{-2 \min(r)/(2 \min(r) +1)})$.
     This ensure that there exists $C_{\varepsilon,2}$ such that for all $C_\varepsilon \geq C_{\varepsilon,2}$, $n \geq \bar{n}$ :
     \begin{equation*}
        B_n(C_\varepsilon)   \leq  \P\left( \kl{f^0}{\hat{f}_{m^*,n}} \geq \frac{C_{\varepsilon,2}}{2} \left( n^{-\frac{2 \min(r)}{2 \min(r) +1}} \right)\right)   \leq \frac{\varepsilon}{4} \cdot
     \end{equation*}    
     We also have by  \reff{eq:A_N_small} that there exists $\tilde{n} \in \N^*$ such that $\P(\ca^c_n) \leq  \varepsilon/2$ for all $n \geq \tilde{n}$. 
     Therefore by setting $C_\varepsilon = \max(C_{\varepsilon,1},C_{\varepsilon,2})$ in \reff{eq:kl_cond}, we have for all $n \geq \max(n^*,\bar{n}, \tilde{n})$:
     \[
      \P\left(\cd_n(C_\varepsilon) \right) \leq A_n(C_\varepsilon)+B_n(C_\varepsilon) + \P(\ca^c_n) \leq \frac{\varepsilon}{2}+\frac{\varepsilon}{2}=\varepsilon, 
     \]
     which gives \reff{eq:C_vareps} and thus concludes the proof.

\end{document}